\documentclass{article}

\usepackage[a4paper,top=2.54cm,bottom=2.54cm,left=3.17cm,right=3.17cm,%
            includehead,includefoot]{geometry}

\usepackage{amsmath,amssymb,amsfonts,amsthm}
\usepackage{graphicx}
\usepackage{subfigure}
\usepackage{float}
\usepackage{xcolor}
\usepackage[numbers,square,sort&compress]{natbib}
\usepackage{hyperref}
  \hypersetup{colorlinks,citecolor=blue,linkcolor=blue,breaklinks=true}
\usepackage{booktabs}
\usepackage{colortbl}
\usepackage{caption}
\usepackage{enumitem}
\usepackage{amsmath}
\numberwithin{equation}{section}
\usepackage{epstopdf}
\usepackage{algorithm}
\usepackage{algpseudocode}
\usepackage{array}
\usepackage{bbding}
\usepackage{placeins}
\usepackage{fancyhdr} 
\usepackage{fancyvrb} 
\usepackage{longtable}
\usepackage{listings}
\usepackage{rotating,rotfloat} 
\usepackage{yhmath} 
\newtheorem{conjecture}{Conjecture}

\usepackage{fancyhdr}
\allowdisplaybreaks

\newtheorem{theorem}{Theorem}[section]
\newtheorem{lemma}{Lemma}[section]

\newtheorem{definition}{Definition}[section]

\newcommand{\bbm}{\begin{bmatrix}}
\newcommand{\ebm}{\end{bmatrix}}

\begin{document}

	\title{Asymptotic Behaviors of Global Solutions to Fourth-order Parabolic and Hyperbolic Equations with Dirichlet Boundary Conditions}

 \author{Wenlong Wu\thanks{School of Mathematical Sciences, East China Normal University, Shanghai 200241,   P.R. China. Email: \texttt{51255500078@stu.ecnu.edu.cn}. },\,
 Yanyan Zhang\thanks{Corresponding author. School of Mathematical Sciences,  Key Laboratory of MEA(Ministry of Education) and Shanghai Key Laboratory of PMMP,  East China Normal University, Shanghai 200241, China. Email: \texttt{yyzhang@math.ecnu.edu.cn}. Y. Zhang is sponsored by
  NSFC [No.12271505] and
  STCSM  [No.22DZ2229014].}}

\date{}

\maketitle

\begin{center}
	\textbf{Abstract}
\end{center}

This paper investigates the asymptotic behaviors of global solutions to fourth-order parabolic and hyperbolic equations with Dirichlet boundary conditions. The equations model Micro-Electro-Mechanical Systems (MEMS) and are depending on a positive voltage parameter $\lambda$. We establish the convergence of global solutions to an equilibrium, along with the convergence rate estimates. Supporting numerical simulations are presented.

\vspace{0.5\baselineskip} 


\vspace{0.5\baselineskip} 


\section{Introduction}

Micro-Electro-Mechanical Systems are microscopic devices that combine electrostatic effects and precision machining technology. In recent years, MEMS devices have been widely used in electronic devices, aerospace, and medical fields, such as gyroscopes in smartphones, pressure sensors in aviation systems, and pacemakers. Due to the importance of MEMS to science and industry, many engineers have developed a strong interest in its mathematical modeling.

A simplified MEMS device is shown in Figure \ref{model} (see \cite[Figure 7.13]{ref6}). It consists of a fixed ground plate and a deformable elastic plate with fixed boundaries. When a voltage is applied, the elastic plate deflects towards the ground plate. A common feature of such devices is that when the applied voltage exceeds a certain threshold, the elastic plate may ​touch down​ on the ground plate (this phenomenon is called quenching), resulting in ​pull-in instability of the device.

\begin{figure}[!htp]
	\begin{center}
		\includegraphics[width=0.6\textwidth]{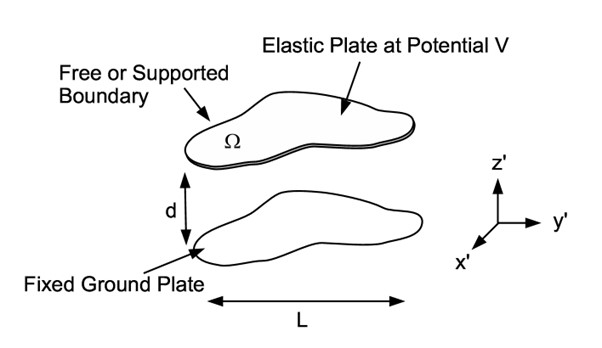}
		\caption{A simplified MEMS device\label{model}}
	\end{center}
\end{figure}
Over the past two decades, mathematical research on MEMS equations ​has expanded rapidly. Among these, the following second-order MEMS problem has been widely studied.
\begin{equation}\label{wxzseq1.5}
	\left.\left\{\begin{array}{ll}\varepsilon u_{tt}+u_{t}-\Delta u=\frac{\lambda f(x)}{(1-u)^{2}},\quad &t>0, \ x\in\Omega ,\\
		u=0,& x\in\partial\Omega ,\\
		u(0,\cdot)=u_{0}, \ \ u_{t}(0,\cdot)=u_{1},& x\in \Omega,\end{array}\right.\right.
\end{equation}
For example, scholars have investigated the quenching phenomena \cite{ref41,ref42,ref43,ref44}, the global existence of solutions \cite{ref43,ref44}, and the convergence of solutions to steady states \cite{ref45}. When the nonlinear term in equation \eqref{wxzseq1.5} ​is replaced by​ other nonlinear forms, there are also many results. We refer the reader to \cite{ref8,ref11,ref14,ref26,ref30,ref34,ref36} for more detalis. In contrast, research on fourth-order MEMS equations remains ​limited.

In 2010, Guo \cite{ref1} considers the fourth-order wave equation with Naiver boundary conditions.
	\begin{equation}\label{guoyujin}
	\left.\left\{\begin{array}{ll} \mu u_{tt}+u_{t}+B\Delta^2u-\Delta u=\frac{\lambda}{(1-u)^2},\quad &t>0, \ x\in\Omega ,\\
		u=\Delta u=0,& x\in\partial\Omega ,\\
		u(0,\cdot)=u_{0}\in [0,1),\quad u_{t}(0,\cdot)=u_{1}\ge 0 & x\in \Omega,\end{array}\right.\right.
\end{equation}
where $\mu,B,\lambda>0,\Omega\subset \mathbb{R}^{N}$ is a bounded domain with smooth boundary. When $1\le N\le 3$, the author obtains some results, including local existence of solution, global existence of solution, and the convergence of the global solution to a stationary solution with the decay rate.

 In 2017, Miyasita focuses on a nonlocal biharmonic MEMS equation with Navier boundary conditions.
 	\begin{equation}\label{wxzseq1.3}
 	\left.\left\{\begin{array}{ll}u_{tt}+\kappa u_{t}+\Delta^2u=G(\beta,\gamma,\nabla u)\Delta u +\lambda\frac{1+\delta|\nabla u|^{2}}{(1-u)^{\sigma}}I(\sigma,\chi,u),\quad &t>0, \ x\in\Omega ,\\
 		u=\Delta u=0,& x\in\partial\Omega ,\\
 		u(0,\cdot)=u_{0}, \ \ u_{t}(0,\cdot)=u_{1}& x\in \Omega,\end{array}\right.\right.
 \end{equation}
 where $\lambda,\delta,\beta,\gamma,\chi>0,\kappa\ge0,\,\sigma\ge 2$ is a constant, $\Omega\subset \mathbb{R}^{n}$ ($n\in \mathbb{N}$) is a bounded domain with smooth boundary. The functions $G,I,H$ are as follows.
 \begin{gather*}
 	G(\beta,\gamma,\nabla u)=\beta\int_{\Omega}|\nabla u|^{2}dx+\gamma,\\
 	I(\sigma,\chi,u)=\frac{1}{H(\sigma,\chi,u)^{\sigma}}, \qquad   H(\sigma,\chi,u)=1+\chi\int_{\Omega}\frac{dx}{(1-u)^{\sigma -1}}.
 \end{gather*}
 In this model, a capacitor control scheme is considered. The author studies the existence and continuity of local solutions, the dynamic properties of the $\omega$-limit set and the convergence rate of the global solution to the stationary solution.

 In this paper, We consider the fourth-order parabolic and hyperbolic equations with dirichlet boundary conditions.
 	\begin{equation}\label{neweq1.1}
 	\left.\left\{\begin{array}{ll} u_{t}+B\Delta^2u-T\Delta u=-\frac{\lambda}{(1+u)^2},\quad &t>0, \ x\in\Omega ,\\
 		u=\partial_\nu u=0,& x\in\partial\Omega ,\\
 		u(0,\cdot)=u_{0},& x\in \Omega,\end{array}\right.\right.
 \end{equation}
 \begin{equation}\label{neweq1.2}
 	\left.\left\{\begin{array}{ll} u_{tt}+u_{t}+B\Delta^2u-T\Delta u=-\frac{\lambda}{(1+u)^2},\quad &t>0, \ x\in\Omega ,\\
 		u=\partial_\nu u=0,& x\in\partial\Omega ,\\
 		u(0,\cdot)=u_{0}, \ \ u_{t}(0,\cdot)=u_{1},& x\in \Omega,\end{array}\right.\right.
 \end{equation}
 where $B>0,T\ge 0$ denotes the bending and stretching coefficients respectively, $\lambda$ is proportional to the square of the applied voltage. Here the unknown function $u=u(t,x)$ describes the displacement of the membrane on the general smooth bounded domain $\Omega\subset \mathbb{R}^{d}$ ($d=1,2$), $\nu$ is the unit outer normal vector on $\partial \Omega$.

 In fact, Laurençot and Walker \cite{ref2} have established some results on problems \eqref{neweq1.1} and \eqref{neweq1.2}, including the local and global well-posedness of solutions, quenching criteria, and sufficient conditions for the existence of the global solution. For details, see Section~\ref{preliminaries}. Based on this, the purpose of this paper is to study the asymptotic behaviors of the global solution to problems \eqref{neweq1.1} and \eqref{neweq1.2}. Before giving the main theorem, we first introduce some notations.
 	\begin{gather*}
 	H_{D}^{2}(\Omega):=\left\{u\in H^{2}(\Omega):u=\partial_{\nu}u=0\ \text{on} \ \partial \Omega \right\},\\
 	\left\|u\right\|_{H_{D}^{2}}:=\left(\int_{\Omega}(B\left|\Delta u \right|^{2}+T\left|\nabla u \right|^{2})dx\right)^{\frac{1}{2}},\\
 	H_{D}^{4}(\Omega):=\left\{u\in H^{4}(\Omega):u=\partial_{\nu}u=0 \ \text{on} \ \partial \Omega\right\},\\
 	\left\| u\right\|_{H_{D}^{4}}:=\left\| u\right\|_{H^{4}(\Omega)}, \quad
 	\left\|\cdot \right\|_{L^{2}}:=\left\|\cdot \right\|_{L^{2}(\Omega)}, \quad \left\|\cdot \right\|_{L^{\infty}}:=\left\|\cdot \right\|_{L^{\infty}(\Omega)}.
 \end{gather*}
 We identify $L^{2}(\Omega)$ with its dual and we denote by $V'$ the dual of $V$ . The inner product in $L^{2}(\Omega)$ and in $V'$ are respectively, denoted by $(\cdot,\cdot)_{L^{2}}$ and $(\cdot,\cdot)_{V'\times V'}$. According to Riesz representation theorem and Sobolev embedding theorem, we get $V\hookrightarrow L^{2}(\Omega)\cong {(L^{2}(\Omega))}' \hookrightarrow V'$. Then there exists two constants $\alpha_{1}, \alpha_{2}>0$ such that
 \begin{equation*}
 	\left\|w \right\|_{V'}\le  \alpha_{1}\left\|w \right\|_{L^{2}}, \quad \left\|w \right\|_{L^{2}}\le \alpha_{2} \left\|w \right\|_{V},\quad \forall w\in V.
 \end{equation*}
 We define the set of stationary solutions to \eqref{neweq1.1} by
 \begin{equation}
 	\mathcal{S}=\left\{\phi \ |\  B\Delta^{2}\phi-T\Delta \phi+\frac{\lambda}{(1+\phi)^{2}}=0,\phi=\partial_{\nu}\phi=0 \ \text{on} \ \partial \Omega 上\right\}.
 \end{equation}
 Let $-A=-(B\Delta^{2}-T\Delta)$, $\mathbf{u}=(u,u_{t})^{T}$, $\mathbf{u}_{0}=(u_{0},u_{1})^{T}$, $g(\mathbf{u})=(0,f(u))^{T}$ ($f(u)=-\lambda(1+u)^{-2}$), We define the matrix operator 	
 \begin{equation*}
 	\mathbb{A}:=\begin{pmatrix}
 		0 & -1\\
 		A & 1
 	\end{pmatrix}.
 \end{equation*}
 We reformulate \eqref{neweq1.2} as a Cauchy problem
 \begin{equation*}
 	\dot{\mathbf{u}}+\mathbb{A}\mathbf{u}=g(\mathbf{u}),\quad t>0,\quad \mathbf{u}(0)=\mathbf{u}_{0}.
 \end{equation*}
  We define the set of stationary solutions to \eqref{neweq1.2} by
 \begin{equation*}
 	\tilde{\mathcal{S}}=\left\{(\phi_{1},\phi_{2}) \ |\  \mathbb{A}(\phi_{1},\phi_{2})=g(\mathbf{u}),\phi_{1}=\partial_{\nu}\phi_{1}=0\ \text{on} \ \partial \Omega 上\right\}.
 \end{equation*}
Note that
 \begin{equation*}
 	\mathbb{A}(\phi_{1},\phi_{2})=g(\mathbf{u}) \Rightarrow 	\left.\left\{\begin{array}{ll} A\phi_{1}=f(\phi_{1}),\\
 		\phi_{2}=0.\end{array}\right.\right.
 \end{equation*}
Then $\tilde{\mathcal{S}}$ can be simplified to
 \begin{equation}
 	\tilde{\mathcal{S}}=\left\{(\phi_{1},\phi_{2}) \ | \  (\phi_{1},\phi_{2})\in  \mathcal{S}\times \left\{0\right\}\right\}.
 \end{equation}
 Note that $d=1,2$, it follows from Sobolev compact embedding theorem (\cite[Section~5.6.3]{refnew40}) that there is a constant $C_{0}=C_{0}(d,\Omega)>0$ such that
 \begin{equation*}
 	\left\|v \right\|_{L^{\infty}}\le C_{0} \left\|v \right\|_{H_{D}^{2}},\quad \forall v\in H_{D}^{2}(\Omega).
 \end{equation*}
 We define two space:
 \begin{equation*}
 	X(\kappa):=\left\{u\in H_{D}^{4}(\Omega):\left\|u \right\|_{H_{D}^{2}}^{2}\le \frac{(1-\kappa)^{2}}{C_{0}^{2}}\ \text{in} \ \Omega \right\}.
 \end{equation*}
 \begin{equation*}
 	Z(\kappa):=\left\{u\in H_{D}^{2}(\Omega)\times L^{2}(\Omega):\left\|u \right\|_{H_{D}^{2}}^{2}+\left\|u \right\|_{L^{2}}^{2}\le \frac{(1-\kappa)^{2}}{C_{0}^{2}}\ \text{in} \ \Omega \right\}.
 \end{equation*}
 The following theorem shows that the global solution of \eqref{neweq1.1} must converge to a stationary solution. In addition, theorem gives the convergence rate.
 	\begin{theorem}\label{theorem1.2}
 	Let $\Omega\subset \mathbb{R}^{d}$ \textup{($d=1,2$)} be an arbitrary bounded smooth domain. Let $B>0$, $T\ge 0$, $\kappa\in(0,1)$, $0<\lambda <\frac{\kappa^{2}(8-3\kappa)}{128C_{0}^{2}|\Omega|}$. For any given $u_{0}\in X(\kappa)$, if the solution $u$ to \eqref{neweq1.1} globally exists and satisfies $u\in X(\kappa)$ for all $t>0$ , then there exists $\psi\in \mathcal{S}$ such that
 	\begin{equation}\label{eq1.10}
 		\lim_{t\to\infty}\left\|u(t,\cdot)-\psi \right\|_{H_{D}^{4}}=0.
 	\end{equation}
 	Furthermore, there is $C_{1}>0$, $\gamma_{1}>0$, and $T_{1}>0$ such that when $t\ge T_{1}$,
 	\begin{equation}\label{eq1.11}
 		\left\|u(t,\cdot)-\psi\right\|_{H_{D}^{4}}\le C_{1}(1+t)^{-\gamma_{1} }.
 	\end{equation}
 \end{theorem}
 Similar to theorem~\ref{theorem1.2}, theorem~\ref{theorem1.3} shows that the global solution to \eqref{neweq1.2} must converge to a stationary solution. In addition, the convergence rate of the solution is also given.

 	\begin{theorem}\label{theorem1.3}
 	Let $\Omega\subset \mathbb{R}^{d}$ \textup{($d=1,2$)} be an arbitrary bounded smooth domain. Let $B>0$, $T\ge 0$, $\kappa\in(0,1)$, $0<\lambda <\frac{\kappa^{2}(8-3\kappa)}{128C_{0}^{2}|\Omega|}$. For any given $(u_{0},u_{1})\in Z(\kappa)$, if the solution $u$ to \eqref{neweq1.2} globally exists and satisfies $(u,u_{t})\in Z(\kappa)$ for all $t>0$, then there exists $\psi\in \mathcal{S}$ such that
 	\begin{equation}\label{eq1.14}
 		\lim_{t\to\infty}\left\{\left\|u_{t} \right\|_{L^{2}}+\left\|u(t,\cdot)-\psi \right\|_{H_{D}^{2}} \right\}=0.
 	\end{equation}
 	Furthermore, there is $C_{2}>0$, $\gamma_{2}>0$, 与 $T_{2}>0$ such that when $t\ge T_{2}$,
 	\begin{equation}\label{eq1.16}
 		\left\|u_{t} \right\|_{L^{2}}+\left\|u(t,\cdot)-\psi\right\|_{H_{D}^{2}}\le C_{2}(1+t)^{-\gamma_{2} }.
 	\end{equation}
 \end{theorem}
 In section~\ref{szmn666}, we consider the case when the dimension $d=1$ and the domain $\Omega=(-1,1)$ for problems \eqref{neweq1.1} and \eqref{neweq1.2}. By fixing parameters $B$ and $T$ while varying $\lambda$, we plot numerical solution profiles. Through observing these graphical results, we present some analysis and propose conjectures.

 The structure of this paper is organized as follows: In section~{preliminaries}, we introduce preliminaries. Section~\ref{dbcpwxfc} contains the proof of Theorem~\ref{theorem1.2}, followed by the proof of Theorem~\ref{theorem1.3} in section~\ref{dbjtjxdsqfc}. In section~\ref{szmn666}, we present relevant numerical simulation results. Finally, in section~\ref{zongjiezhanwang}, we conclude the paper and discuss future research directions."

\section{Preliminaries}\label{preliminaries}

	\begin{definition}[\cite{ref3}]
	\label{predef2.7}
	Suppose that $H$ is a complete metric space, $S(t)$ is a nonlinear $C_{0}$-semigroup defined on $H$. A continuous function $V:H\mapsto \mathbb{R}$ is called a Lyapunov function with respect to $S(t)$ if the following two conditions are satisfied:
	\begin{itemize}
		\item[(\romannumeral1)] For any $x\in H$, $V(S(t)x)$ is monotone non-increasing in $t$.
		\item[(\romannumeral2)] $V(x)$ is bounded from below, i.e., there is a constant $C$ such that for all $x\in H$, $V(x)\ge C $.
	\end{itemize}
\end{definition}

\begin{definition}[\cite{ref3}]
	\label{predef2.8}
	Suppose that $H$ is a complete metric space, $S(t)$ is a nonlinear $C_{0}$-semigroup defined on $H$ and $V(x)$ is a Lypunov function. Then the nonlinear semigroup $S(t)$, or more precisely system $(H,S(t),V)$ is called a gradient system if the following conditions are satisfied:
	\begin{itemize}
		\item[(\romannumeral1)] For any $x\in H$, there is $t_{0}>0$ such that
		\begin{equation*}
			\bigcup_{t \ge t_{0}}S(t)x
		\end{equation*}
		is relatively compact in $H$.
		\item[(\romannumeral2)] If for $t>0$, $V(S(t)x)=V(x)$, then $S(t)x=x$.
	\end{itemize}
\end{definition}

\begin{theorem}[\cite{ref3}]
	\label{prethm2.4}
	Suppose that $(H,S(t),V)$ is a gradient system. Then for any $x\in H$, $\omega$-limit set $\omega(x)$ is a connected compact invariant set, and it consists of the fixed points of $S(t)$.
\end{theorem}

\begin{definition}\label{neiji}
	Recall that $V=H_{D}^{2}(\Omega)$, We define the inner product in $V'$:
	\begin{equation*}
		(v,w)_{V'\times V'}:=(v,A_{\Delta^{2}}^{-1}w)_{V'\times V},\quad \forall v,w\in V',
	\end{equation*}
	where $(\cdot , \cdot)_{V'\times V}$ is the inner product between $V$ and $V'$, and $A_{\Delta^{2}}:V\longrightarrow V'$ is an isomorphism from $V$ to $V'$. We define
	\begin{equation*}
		( A_{\Delta^{2}}\psi,\phi )_{V'\times V}=\int_{\Omega}\psi \phi dx+T\int_{\Omega}\nabla \psi \nabla \phi dx+B\int_{\Omega}\Delta \psi \Delta \phi dx,\quad \forall \psi,\phi \in V.
	\end{equation*}
	Note that $A_{\Delta^{2}}:V\longrightarrow V'$ is the operator
	\begin{equation*}
		A_{\Delta^{2}}=I+B\Delta^{2}-T\Delta.
	\end{equation*}
\end{definition}
\begin{theorem}[\mbox{\cite[Theorem~4.C]{ref10}}]
	\label{prethm2.6}
	Let the map $f:U\subseteq X\mapsto Y$ be $C^{n}$ on the open convex set $U$, where $X$ and $Y$ are Banach spaces over $\mathbb{K}$. Then $f$ can be expanded into the following form
	\begin{align*}
		f(u+h)=f(u)+\sum_{k=1}^{n-1}\frac{1}{k!}f^{(k)}(u)h^{k}+R_{n}, \\
		R_{n}=\int_{0}^{1}\frac{(1-\tau)^{n-1}}{(n-1)!}f^{(n)}(u+\tau h)h^{n}d\tau,
	\end{align*}
	where $f^{(k)}(u)$ represents the $k$-th Fr{\'e}chet derivative at $u$, and $f^{(k)}(u)h^{k}:=f^{(k)}(u)(h,\cdots,h)$.
\end{theorem}
\begin{definition}[\mbox{\cite[Theorem~2.4]{ref5}}]\label{predef2.10}
	A map $T:U\rightarrow Y$ is called analytic at $x_{0}\in U$, if there exist $r\in(0,dist(x_{0},\partial U))$ and symmetric $T_{k}(x_{0})\in \mathcal{B}_{k}(X,Y)$ for any $k\ge 1$, such that
	\begin{equation*}
		T(x_{0}+h)=T(x_{0})+\sum_{k=1}^{+\infty} T_k(x_{0})(h,\cdots,h)
	\end{equation*}
	uniformly for $h\in X$ with $\left\|h \right\|_{X}<r$, where the convergence is in $Y$. The map $T:U\longrightarrow Y$ is called
	analytic in $U$, if $T$ is analytic at every $x_{0}\in U$.
\end{definition}
\begin{theorem}[\mbox{\cite[Lojasiewicz-Simon Theorem]{ref3}}]
	\label{yibandeLS}
	Suppose that $F:\mathbb{R}^{m}\to \mathbb{R}$ is an analytic function in a
	neighborhood of a critical point $a$ (i.e., $\nabla F(a)=0$) in $\mathbb{R}^{m}$, Then there exist constants $\sigma>0$ and $\theta\in (0,\frac{1}{2})$, such that when $\left\|x-a \right\|_{\mathbb{R}^{m}}\le \sigma$,
	\begin{equation*}
		\left|F(x)-F(a) \right|^{1-\theta}\le \left\|\nabla F(x) \right\|_{\mathbb{R}^{m}}.
	\end{equation*}
\end{theorem}
\begin{theorem}[\mbox{\cite[Theorem~2.2]{ref5}}]
	\label{prethm2.8}
	Let $d$ be an integer, let $\Omega\subset \mathbb{R}^{d}$ be a finitely measured space.
	space, and consider a real Hilbert space $V\subset L^{2}(\Omega,\mathbb{R}^{d})$ satisfying
	\begin{itemize}
		\item[(i)] $V$ is dense in $L^{2}.(\Omega,\mathbb{R}^{d})$.
		\item[(ii)] the embedding $V\hookrightarrow L^{2}(\Omega,\mathbb{R}^{d})$ is compact.
	\end{itemize}
	Let $a(u,v)$ be a symmetric and coercive bilinear continuous form on $V$ defined by
	\begin{equation*}
		a(u,v)=(Au,v)_{L^{2}},\quad \forall u,v\in V,
	\end{equation*}
	$A$ is an algebraic and topological isomorphism from $V$ into $V'$ and it can also be considered as a self-adjoint unbounded operator in $L^{2}(\Omega,\mathbb{R}^{d})$ with
	domain $D=D(A)\subset V$,
	\begin{equation*}
		D=\left\{ v\in V|Av\in L^{2}(\Omega,\mathbb{R}^{d}) \right\}.
	\end{equation*}
	Assume that there exists $p\ge 2$ such that
	\begin{itemize}
		\item[(H1)] the embedding $A^{-1}(L^{p}(\Omega,\mathbb{R}^{d}))\hookrightarrow L^{\infty}(\Omega,\mathbb{R}^{d})$ is continuous.
		\item[(H2)] For all $a\in L^{\infty}(\Omega,\mathbb{R}^{d}\times \mathbb{R}^{d})$, $h\in L^{\infty}(\Omega,\mathbb{R}^{d})$, if $u\in D$ is a solution of $Au+a(x)u=h$, then $u\in L^{p}(\Omega,\mathbb{R}^{d})$.\\
		Let
		\begin{align*}
			F: \Omega &\times \mathbb{R}^{d}\longrightarrow L^{2}(\Omega)\notag\\&(x,s)\longmapsto F(x,s)	\end{align*}
		which satisfies
		\item[(H3)] $F$ is analytic with respect to $s$ "uniformly" in $x\in \Omega$, and $\nabla F(\cdot,\cdot)$, $\nabla^{2}F(\cdot,\cdot)$ are bounded on $\Omega\times (-\beta,\beta)^{d}$, $\forall$ $\beta\in \mathbb{R}^{+}$. $(\nabla F=(\frac{\partial F}{\partial s_{1}},\cdots,\frac{\partial F}{\partial s_{d}}))$.
	\end{itemize}
	Consider the elliptic-like problem
	\begin{equation*}
		Au=f(u),
	\end{equation*}
	where $A:A^{-1}(L^{p}(\Omega,\mathbb{R}^{d}))\longrightarrow L^{p}(\Omega,\mathbb{R}^{d})$ is related to $F$ through the formula
	\begin{equation*}
		f(u)(x)=\nabla F(x,u(x)).
	\end{equation*}
	We denote by
	\begin{equation*}
		E(u)=\frac{1}{2}\int \left((Au)u\right)dx-\int_{\Omega}F(x,u)dx.
	\end{equation*}and $S=\left\{\psi\in D \cap L^{\infty}(\Omega,\mathbb{R}^{d}) \ | \  A\psi=f(\psi) \right\}$.
	Let $A,F$ and $f$ be as above. Let $\varphi \in S$, and assume that there exists a ball $B=\left\{u\in V,\left\|u-\varphi \right\|_{V}<\rho \right\}$ such that $f\in C^{1}(B,V')$. Then there exists $\theta\in (0,\frac{1}{2})$ and $\sigma>0$ such that for all $v\in V$, $\left\|u-\varphi \right\|_{V}<\sigma$
	\begin{equation*}
		\left\|-Au+f(u) \right\|_{V'}\ge \left|E(u)-E(\varphi) \right|^{1-\theta}.
	\end{equation*}
\end{theorem}
\begin{theorem}[\mbox{\cite[Theorem~1.4]{ref2}}]
	\label{newtheorem10.1}
	Let $\Omega\subset \mathbb{R}^{d}$ \textup{($d=1,2$)} be an arbitrary bounded smooth domain. Let $B>0$, $T\ge 0$, $\lambda>0$, $\kappa\in(0,1)$. For any given $u_{0}\in H_{D}^{2}(\Omega)$ satisfying $u_{0}\ge -1+\kappa$. Then the following hold:
	\begin{itemize}
		\item[(\romannumeral1)] There are $\tau_{m}>0$ and a unique maximal solution $u$ to \eqref{neweq1.1} with regularity
		\begin{equation*}
			u\in C([0,\tau_{m}),H^{2}(\Omega))\cap C((0,\tau_{m}),H^{4}(\Omega)) \cap C^{1}((0,\tau_{m}),L^{2}(\Omega)).
		\end{equation*}
		\item[(\romannumeral2)] If $\tau_{m}<\infty$, then
		\begin{equation*}
			\liminf_{t \to \tau_{m}}\left(\min_{\Omega}u(t)\right)=-1.
		\end{equation*}
	\end{itemize}
\end{theorem}
\begin{theorem}[\mbox{\cite[Proposition~3.1]{ref2}}]\label{newtheorem2.2}
	Let \begin{equation*}
		S(\kappa):=\left\{u\in H_{D}^{2}(\Omega): u>-1+\kappa\,\,\textup{in $\Omega$}\right\}\times L^{2}(\Omega).
	\end{equation*}
	Let $\Omega\subset \mathbb{R}^{d}$ ($d=1,2$) be an arbitrary bounded smooth domain. Let $B>0$, $T\ge 0$, $\lambda>0$, $\kappa\in(0,1)$. For any given $(u_{0},u_{1})\in S(\kappa)$, the following hold:
	\begin{itemize}
		\item[\textup{(\romannumeral1)}] There are $\tau_{m}>0$ and a unique maximal mild solution  $\mathbf{u}=(u,u_{t})$ to \eqref{neweq1.2} and
		\begin{equation*}
			u\in C([0,\tau_{m}),H^{2}(\Omega))\cap C^{1}([0,\tau_{m}),L^{2}(\Omega)).
		\end{equation*}
		\item[\textup{(\romannumeral2)}] If $\tau_{m}<\infty$, then
		\begin{equation*}
			\liminf_{t \to \tau_{m}}\left(\min_{\Omega}u(t)\right)=-1
		\end{equation*}
		or
		\begin{equation*}
			\limsup_{t \to \tau_{m}} \left\| (u(t),\partial_{t}u(t))\right\|_{H_{D}^{2}(\Omega)\times L^{2}(\Omega)}=\infty.
		\end{equation*}
	\end{itemize}
\end{theorem}

\begin{theorem}[\mbox{\cite[Corollary~3.2]{ref2}}]\label{newtheorem18.2}
	Based on the conditions listed in Theorem~\ref{newtheorem2.2}, if $$(u_{0},u_{1})\in (H_{D}^{4}(\Omega)\times H_{D}^{2}(\Omega))\cap S(\kappa),$$
	then the mild solution $\mathbf{u}$ to \eqref{neweq1.2} is a classical solution with regularity
	\begin{equation*}
		u\in C([0,\tau_{m}),H^{4}(\Omega))\cap C^{1}([0,\tau_{m}),H^{2}(\Omega))\cap C^{2}([0,\tau_{m}),L^{2}(\Omega)).
	\end{equation*}
\end{theorem}

\section{Parabolic Problem}\label{dbcpwxfc}
In this section, we will prove Theorem~\ref{theorem1.2}. First, we show that the parabolic problem \eqref{neweq1.1} defines a gradient system. Next, we prove the Lojasiewicz Simon inequality corresponding to \eqref{neweq1.1}. Based on these two results,
we prove that the global solution to \eqref{neweq1.1} must converge to a stationary solution and obtain the corresponding convergence rate.
\subsection{Gradient system}
	In this subsection, we will show that problem \eqref{neweq1.1} defines a gradient system. The core point is to prove the uniform boundedness of $\left\|u(t) \right\|_{H^{5}}$, which leads to the precompactness of the orbit in $X(\kappa)$.
	
	Recall that $-A=-(B\Delta^{2}-T\Delta)$, it follows from \cite[Section~3.2]{ref2} that $-A$ generates an analytic semigroup $\left\{e^{-tA}: t\ge 0\right\}$ on $L^{2}(\Omega)$, and the domain is $D(A)=H_{D}^{4}(\Omega)$. There are also exist constant $M>0$ and $\alpha>0$ such that
\begin{equation*}
	\left\|e^{-tA} \right\|_{\mathcal{L}(L^{2}(\Omega))}\le Me^{-\alpha t},\quad t\ge 0.
\end{equation*}
For convenience, let
$$R(t)=\left\{e^{-tA}: t\ge 0\right\}.$$
Then it follows from the dedfinition of analytic semigroup (\cite[Section~2.2.5]{ref15}) that  $R(t)$ is a $C_{0}$-semigroup.

Under the assumptions of Theorem~\ref{theorem1.2}, the global solution $u$ satisfies
\begin{equation*}
	\left\|u \right\|_{H_{D}^{2}}\le \frac{(1-\kappa)^{2}}{C_{0}^{2}},\quad \forall t\ge 0.
\end{equation*}
Since $H^{2}(\Omega)\hookrightarrow L^{\infty}(\Omega)$, we have
\begin{equation*}
	\left\|u \right\|_{L^{\infty}}\le C_{0}\left\|u \right\|_{H_{D}^{2}}\le 1-\kappa,\quad \forall t\ge 0.
\end{equation*}
When $u_{0}\in X(\kappa)$, we can deduce from Theorem~\ref{newtheorem10.1} that $u$ has the following regularity.
\begin{equation*}
	u\in C([0,\tau_{m}),H^{4}(\Omega))\cap C^{1}([0,\tau_{m}),L^{2}(\Omega)).
\end{equation*}
Let $f(u)=-1/(1+u)^{2}$, then it follows from \cite[Section~3.2]{ref2} that $u$ satisfies the integral equation.
\begin{equation*}
	u(t)=R(t)u_{0}+\int_{0}^{t}R(t-\tau)f(u(\tau))d\tau.
\end{equation*}
Note that the gradient system consists of three elements that are function space, $C_{0}$-semigroup and Lyapunov function. Under the assumptions of Theorem~\ref{theorem1.2}, the function space is
\begin{equation*}
	X(\kappa):=\left\{u\in H_{D}^{4}(\Omega):\left\|u \right\|_{H_{D}^{2}}^{2}\le \frac{(1-\kappa)^{2}}{C_{0}^{2}}\textup{ in $\Omega$}\right\}.
\end{equation*}
By the definition of $X(\kappa)$, we easily get that $X(\kappa)$ is a complete metric space. In addition, under the assumptions of Theorem~~\ref{theorem1.2}, we can deduce from Theorem~\ref{newtheorem10.1} that $S(t)$ defined by $$S(t):X(\kappa)\to X(\kappa)$$ is $C_{0}$-semigroup.

In addition to $X(\kappa)$ and $S(t)$, we also need to prove the existence of Lyapunov function. We multiply the equation in \eqref{neweq1.1} by $u_{t}$, and integrate over $\Omega$ to get
\begin{equation}\label{eqnew111}
	\frac{d}{dt}\int_{\Omega}\left( \frac{1}{2}B|\Delta u|^2+\frac{1}{2}T|\nabla u|^2-\frac{\lambda}{1+u}\right)dx+\left\|u_{t} \right\|_{L^{2}(\Omega)}^{2}=0.
\end{equation}
The calculation details are as follows:
\begin{equation*}
	\begin{aligned}
		&\int_{\Omega} (\Delta ^{2}u)u_{t}dx\\
		&=\int_{\Omega} \Delta u \Delta u_{t}dx+\int_{\partial \Omega}u_{t}\frac{\partial (\Delta u)}{\partial \nu}dS-\int_{\partial \Omega}\Delta u \frac{\partial (u_{t})}{\partial \nu}dS\\
		&=\int_{\Omega} \Delta u \Delta u_{t}dx.\\
		&\int_{\Omega} -(\Delta u) u_{t}dx=\int_{\Omega}\nabla u \nabla u_{t}dx-\int_{\partial \Omega} u_{t} \frac{\partial u}{\partial \nu}dS\\
		&=\int_{\Omega}\nabla u \nabla u_{t}dx.
	\end{aligned}
\end{equation*}
Based on equation~\eqref{eqnew111}, we define the following energy functional:
\begin{equation}\label{eqnewnew555}
	E(u)=\int_{\Omega}\left(\frac{1}{2}B|\Delta u|^2+\frac{1}{2}T|\nabla u|^2-\frac{\lambda}{1+u}\right)dx.
\end{equation}
We will use the definition~\ref{predef2.7} to show that $E:X(\kappa)\to \mathbb{R}$ is a Lyapunov function.
\begin{lemma}\label{newnlfh}
	$E:X(\kappa)\to \mathbb{R}$ is a Lyapunov function with respect to $S(t)$.
\end{lemma}
\begin{proof}
	We will prove the lemma in three steps.
	\begin{itemize}
		\item[(\romannumeral 1)] Continuity: Fix $y_{0}\in X(\kappa)$. Suppose that there is sequence $y_{n}\in X(\kappa)$ such that $y_{n}\to y_{0}$ in $X(\kappa)$. We claim that $E(y_{n})\to E(y_{0})$.
		
		By the assumptions we get $y_{n}\to y_{0}$ in $H_{D}^{4}(\Omega)$. Since $H^{4}(\Omega)\hookrightarrow L^{\infty}(\Omega)$ and $H^{4}(\Omega)\hookrightarrow H^{2}(\Omega)$, we get  $y_{n}\to y_{0}$ in $L^{\infty}(\Omega)$, and $y_{n}\to y_{0}$ in $H_{D}^{2}(\Omega)$. Thus, when $n\to +\infty$,
		\begin{gather*}
			\left\|y_{n} \right\|_{H_{D}^{2}}^{2}-	\left\|y_{0} \right\|_{H_{D}^{2}}^{2} \to 0,\\
			\left\|y_{n}-y_{0} \right\|_{L^{\infty}}\to 0.
		\end{gather*}
		In addition, since $y_{0},y_{n}\in X(\kappa)$, we have
		\begin{gather*}
			\left\|y_{0} \right\|_{L^{\infty}}\le C_{0}\left\|y_{0} \right\|_{H_{D}^{2}}\le 1-\kappa,\\
			\left\|y_{n} \right\|_{L^{\infty}}\le C_{0}\left\|y_{n} \right\|_{H_{D}^{2}}\le 1-\kappa.
		\end{gather*}
		A direct calculation shows that
		\begin{align*}
			\Bigg|E(y_{n})-E(y_{0}) \Bigg|&=\left|\frac{1}{2}\left\|y_{n} \right\|_{H_{D}^{2}}^{2}-\frac{1}{2}\left\|y_{0} \right\|_{H_{D}^{2}}^{2}+\int_{\Omega}\left(\frac{\lambda}{1+y_{0}}-\frac{\lambda}{1+y_{n}}\right)dx \right|\\
			&\le \frac{1}{2}\Bigg|\left\|y_{n} \right\|_{H_{D}^{2}}^{2}-\left\|y_{0} \right\|_{H_{D}^{2}}^{2} \Bigg|+\lambda \Bigg|\int_{\Omega}\frac{y_{n}-y_{0}}{(1+y_{0})(1+y_{n})} dx\Bigg|\\
			&\le \frac{1}{2}\Bigg|\left\|y_{n} \right\|_{H_{D}^{2}}^{2}-\left\|y_{0} \right\|_{H_{D}^{2}}^{2} \Bigg|+\lambda \kappa^{-2}|\Omega|\big|y_{n}-y_{0} \big|_{L^{\infty}}.
		\end{align*}
		Therefore, when $y_{n}\to y_{0}$ , $E(y_{n})\to E(y_{0})$.
		\item[(\romannumeral 2)] Dissipation: By \eqref{eqnew111} and the definition of $E$, we immediately get
		\begin{equation*}
			\frac{d}{dt}E(u(t))=-\left\|u_{t} \right\|_{L^{2}}^{2}\le 0.
		\end{equation*}
		Therefore for any given $u_{0}\in X(\kappa)$, $E(S(t)u_{0})$ is monotone non-increasing in $t$.
		\item[(\romannumeral 3)] Lower bound: Since $u_{0}\in X(\kappa)$, we have
		\begin{equation*}
			\left\|u_{0} \right\|_{L^{\infty}}\le C_{0}\left\|u_{0} \right\|_{H_{D}^{2}}\le 1-\kappa.
		\end{equation*}
		Then
		\begin{equation*}
			E(u_{0})\ge -\frac{\lambda |\Omega|}{\kappa},
		\end{equation*}.
	\end{itemize}
In summary, $E$ is a Lyapunov function with respect to $S(t)$.
\end{proof}

Now we are ready to prove that \eqref{neweq1.1} defines a gradient system.
\begin{lemma}\label{eqnewlemma35.1}
	Under the assumption of Theorem~\ref{theorem1.2}, $(X(\kappa),S(t),E)$ is a gradient system.
\end{lemma}
\begin{proof}
	First, integrating \eqref{eqnew111} with respect to $t$ to get
		\begin{equation}\label{eqnew30.1}
		E(u(t))+\int_{0}^{t}\left\|u_t\right\|_{L^2(\Omega)}^2d\tau=E(u_{0}).
	\end{equation}
		This indicates that if there is $t_{0}>0$ such that $E(u(t_0))=E(S(t_0)u_0)=E(u_{0})$, then for all $0\le t\le t_0$, $u_{t}=0$. Therefore, $u_{0}$ must be an
		equilibrium ($S(t)u_{0}=u_{0}$). Then we will show that $\left\|u(t) \right\|_{H^{5}}$ is uniformly bounded to get the precompactness of the orbit in $X(\kappa)$.
		
		Due to the limited regularity of $u$, the density argument will be employed to complete the proof. By \cite[Theorem~2.5.2]{ref3}, $D(A^{2})$ is dense in $D(A)$. Then there is a sequence $u_{0}^{(n)}\in D(A^{2})$ such that $u_{0}^{(n)}\to u_{0}$ in $D(A)$. Since the embedding $H^{4}(\Omega)\hookrightarrow H^{2}(\Omega)$ is continuous, $u_{0}^{(n)}\to u_{0}$ in $H_{D}^{2}(\Omega)$. By the definition of strong convergence in Banach space, let $\varepsilon_{0}=\frac{(1-\kappa/2)}{C_{0}}-\frac{(1-\kappa)}{C_{0}}$, there is $n_{0}\in \mathbb{N}^{*}$ such that when $n\in \mathbb{N}$ and $n\ge n_{0}$,
	\begin{equation*}
		\big\|u_{0}^{(n)}-u_{0} \big\|_{H_{D}^{2}}\le \varepsilon_{0}.
	\end{equation*}
	By the triangle inequality, we get
	\begin{equation*}
		\big\|u_{0}^{(n)} \big\|_{H_{D}^{2}}\le \big\|u_{0} \big\|_{H_{D}^{2}}+\varepsilon_{0}\le \frac{1-\kappa/2}{C_{0}}.
	\end{equation*}
	Thus, we have
	\begin{equation*}
		\big\| u_{0}^{(n)}\big\|_{L^{\infty}}\le C_{0}\big\|u_{0}^{(n)} \big\|_{H_{D}^{2}}\le 1-\frac{\kappa}{2}.
	\end{equation*}
	Similarly, there is $n_{1}\in \mathbb{N}^{+}$ such that when $n\in \mathbb{N}$ and $n\ge n_{1}$,
	\begin{equation*}
		\big\|u_{0}^{(n)} \big\|_{H^{4}}\le \big\|u_{0} \big\|_{H^{4}}+\varepsilon_{0}.
	\end{equation*}
	Here we still use $n_{0}$ to denote $\max \left\{n_{0},n_{1}\right\}$. For $n\ge n_{0}$, we consider the problem where the initial value of \eqref{neweq1.1} becomes $u_{0}^{(n)}$, i.e.
	\begin{equation}\label{eqnew35.3}
		\left.\left\{\begin{array}{ll} u_{t}^{(n)}+B\Delta^2u^{(n)}-T\Delta u^{(n)}=-\frac{\lambda}{(1+u^{(n)})^2},\quad &t>0, \ x\in\Omega ,\\
			u^{(n)}=\partial_\nu u^{(n)}=0,& x\in\partial\Omega ,\\
			u^{(n)}(0,\cdot)=u_{0}^{(n)},& x\in \Omega,\end{array}\right.\right.
	\end{equation}
	Since $u_{0}^{(n)}\in X(\frac{\kappa}{2})$, then it follows from Theorem~\ref{newtheorem10.1} that there are $\tau_{m}^{(n)}>0$ and a unique maximal solution $u^{(n)}$ satisfying \eqref{eqnew35.3} with regularity
	\begin{equation*}
		u^{(n)}\in C([0,\tau_{m}),H^{4}(\Omega))\cap C^{1}([0,\tau_{m}),L^{2}(\Omega)).
	\end{equation*}
	Then we show that when $n\ge n_{0}$, $\tau_{m}^{(n)}=\infty$. Since $	\big\|u_{0}^{(n)} \big\|_{L^{\infty}}\le 1-\frac{\kappa}{2}$, by the regularity of $u^{(n)}$ we get
	\begin{equation*}
		T_{0}^{(n)}:=\sup\left\{\tau\in (0,\tau_{m}^{(n)}) \ : \ u^{(n)}(t)\ge -1+\frac{\kappa}{4}, \ t\in[0,\tau)\right\}>0,
	\end{equation*}
	In addition, for $t\in [0,T_{0}^{(n)})$, $1+u^{(n)}\ge \frac{\kappa}{4}$. Multiply the equation in \eqref{eqnew35.3} by $u_{t}^{(n)}$, integrating with respect to $x$ and $t$, we get for all $t \in(0,\tau_{m}^{(n)})$,
	\begin{align}\label{eqnew35.6}
		E(u^{(n)})-E(u_{0})=\int_{0}^{t}\big\|u_{t}^{(n)} \big\|_{L^{2}}^{2}.
	\end{align}
	By \eqref{eqnew35.6} and the definition of $E$, we have
	\begin{align}
		\big\|u^{(n)} \big\|_{H_{D}^{2}}^{2}&\le 	\big\|u_{0}^{(n)} \big\|_{H_{D}^{2}}^{2}+2\lambda\int_{\Omega}\left(\frac{1}{1+u^{(n)}}-\frac{1}{1+u_{0}^{(n)}} \right)dx \notag\\
		&\le \frac{\left(1-\kappa/2\right)^{2}}{C_{0}^{2}}+8\lambda \kappa^{-1}|\Omega|. \label{eqnew35.7}
	\end{align}
	A direct calculation yields that
	\begin{equation}\label{eqnew35.8}
		\frac{\left(1-\kappa/2\right)^{2}}{C_{0}^{2}}+8\lambda \kappa^{-1}|\Omega| \le \frac{\left(1-\kappa/4\right)^{2}}{C_{0}^{2}} \iff \lambda \le \frac{\kappa^{2}(8-3\kappa)}{128C_{0}^{2}|\Omega|},
	\end{equation}
	which is consistent with the assumption in~\ref{theorem1.2}. Thus $T_{0}^{(n)}\ge \tau_{m}^{(n)}$. On the other hands, by the definition of $T_{0}^{(n)}$, $T_{0}^{(n)}\le \tau_{m}^{(n)}$. Thus $T_{0}^{(n)}=\tau_{m}^{(n)}$. Combining \eqref{eqnew35.7} with \eqref{eqnew35.8} yields
	\begin{equation*}
		\big\|u^{(n)} \big\|_{L^{\infty}}\le C_{0} \big\|u^{(n)} \big\|_{H_{D}^{2}}\le 1-\frac{\kappa}{4}.
	\end{equation*}
	Therefore,
	$$\liminf_{t\to \tau_{m}^{(n)}}\left(\min_{\Omega}u^{(n)}(t)\right)\ge -1+\frac{\kappa}{4}>-1.$$
	It follows from Theorem~\ref{newtheorem10.1}(\romannumeral2) that $\tau_{m}^{(n)}=\infty$. In addition, the solution to \eqref{eqnew35.3} globally exists and strictly greater than $-1$. Differentiate the equation in \eqref{eqnew35.3} with respect to $t$, we get
	\begin{equation*}
		u^{(n)}_{tt}+Au^{(n)}_{t}=f'(u^{(n)})u^{(n)}_{t}.
	\end{equation*}
	where
	\begin{gather*}
		A=B\Delta^{2}-T\Delta,\quad f'(u^{(n)})=\frac{2\lambda}{(1+u^{(n)})^{3}}.
	\end{gather*}
	Then we show $u_{t}^{(n)}$ satisfies Dirichlet boundary condition. For convenience, we prove that $$u=\partial_{\nu}u \Longrightarrow u_{t}=\partial_{\nu}(u_{t})$$  First, by $u=0$ on $\partial \Omega$, we get $u_{t}=0$. On the other hands, it follows from the definition of directional derivative and partial derivative that
	\begin{align*}
		\frac{\partial (u_t)}{\partial \nu}&=\lim_{\rho  \to 0^{-}}\frac{u_{t}(x+\rho \nu,t)-u_{t}(x,t)}{\rho}\\
		&=\lim_{\rho  \to 0^{-}} \lim_{\Delta t \to 0}\frac{u(x+\rho \nu,t+\Delta t)-u(x+\rho \nu,t)-u(x,t+\Delta t)+u(x,t)}{\rho \Delta t}.
	\end{align*}
	Define the binary function
	\begin{equation*}
		h(\rho,\Delta t)=\frac{u(x+\rho \nu,t+\Delta t)-u(x+\rho  \nu ,t)-u(x,t+\Delta t)+u(x,t)}{\rho \Delta t}.
	\end{equation*}
	Then we have
	\begin{gather*}
		\lim_{\rho  \to 0^{-}} h(\rho ,\Delta t)=\lim_{\Delta t \to 0 }\left[\frac{1}{\Delta t}\cdot\left( \frac{\partial u}{\partial \nu}(t+\Delta t)-\frac{\partial u}{\partial \nu}(t)\right) \right]=(\frac{\partial u}{\partial  \nu })_t=0, \  \text{关于}\  \Delta t \  \text{一致}.
	\end{gather*}
	On one hand,
	\begin{equation*}
		\lim_{\Delta t \to 0}h(\rho,\Delta t)=\frac{u_{t}(x+\rho \nu,t)-u_{t}(x,t)}{\rho}, \quad \rho\neq 0.
	\end{equation*}
	It follows from \cite[Moore-Osgood~Theorem]{ref16} that
	\begin{equation*}
		\frac{\partial (u_t)}{\partial \nu}=\lim_{\rho  \to 0^{-}} \lim_{\Delta t \to 0}h(\rho ,\Delta t)=\lim_{\Delta t  \to 0} \lim_{\rho \to 0^{-}}h(\rho ,\Delta t)=0,\quad x\in \partial \Omega.
	\end{equation*}
	Therefore $u_{t}^{(n)}$ satisfies Dirichlet boundary condition. Combining the equation in \eqref{eqnew35.3}, we immediately get
	\begin{equation*}
		u_{t}^{(n)}\big|_{t=0}=f(u_{0}^{(n)})-Au_{0}^{(n)}.
	\end{equation*}
	Let $v=u_{t}^{(n)}$, for $n\ge n_{0}$, $v$ satisfies
	\begin{equation}\label{eqnew35.4}
		\left.\left\{\begin{array}{ll} v_{t}+Av=f'(u^{(n)})v,\quad &t>0, \ x\in\Omega ,\\
			v=\partial_\nu v=0,& x\in\partial\Omega ,\\
			v(0,\cdot)=f(u_{0}^{(n)})-Au_{0}^{(n)},& x\in \Omega,\end{array}\right.\right.
	\end{equation}
	Recall that $-A$ is the infinitesimal generator of $R(t)$, then by \cite[Remark~2.2.1]{ref3}, we know that $-A$ is a maximal accretive operator(see \cite[Definition~2.2.1]{ref3}). Note that for any $v_{1},v_{2}\in L^{2}(\Omega)$,
	\begin{equation*}
		\big\|f'(u^{(n)})v_{1}-f'(u^{(n)})v_{2} \big\|_{L^{2}}\le \big|f'(u^{(n)}) \big|_{L^{\infty}}\big\|v_{1}-v_{2} \big\|_{L^{2}}\le 64\lambda \kappa^{-3}\big\|v_{1}-v_{2} \big\|_{L^{2}}.
	\end{equation*}
	Thus $f'(u^{(n)})v$ satisfies global Lipschitz condition. Since $u_{0}^{(n)}\in D(A^{2})$ and $u^{(n)}\ge -1+\kappa/4$, we get $v\big|_{t=0}\in D(A)$. In addition, $L^{2}(\Omega)$ is reflexive Banach space. It then follows from \cite[Corollary~2.5.2]{ref3} that problem \eqref{eqnew35.4} admits a classical solution $v$ with regularity
	\begin{equation*}
		v\in C([0,+\infty),H^{4}(\Omega))\cap C^{1}([0,+\infty),L^{2}(\Omega)).
	\end{equation*}
	Since $v=u_{t}^{(n)}$, we have
	\begin{equation*}
		u^{(n)}\in C^{1}([0,+\infty),H^{4}(\Omega))\cap C^{2}([0,+\infty),L^{2}(\Omega)).
	\end{equation*}
	For $n\ge n_{0}$, consider the following problem:
	\begin{equation}\label{eq3.8}
		\left.\left\{\begin{array}{ll}B\Delta^2 u^{(n)}-T\Delta u^{(n)}=-\frac{\lambda}{(1+u^{(n)})^2}-u_t^{(n)},\quad &x\in\Omega ,\\u^{(n)}=\partial_\nu u^{(n)}=0,& x\in\partial\Omega ,\end{array}\right.\right.
	\end{equation}
	By the priori estimate of ellptic equation (see \cite[Section~1.3.4]{ref3}), there exist two positive constants $C_{1},C_{2}$ independent of $u^{(n)}$, such that
	\begin{equation}\label{eqnewnew1000}
		\left\| u^{(n)}(t) \right\|_{H^5}\le C_{1}\left(\left\| -\frac{\lambda }{(1+u^{(n)})^2}   \right\| _{H^{1}}+\left\| u_t^{(n)}  \right\|_{H^{1}} \right)+C_{2}\left\|u^{(n)} \right\|_{L^{2}}.
	\end{equation}
	
	Now we estimate the right terms in \eqref{eqnewnew1000}. Recall that $\left\|u^{(n)} \right\|_{H_{D}^{2}}\le \frac{1-\kappa/4}{C_{0}}$. We use the definition of the norm in Sobolev space to split $\left\|-\lambda(1+u^{(n)})^{-2} \right\|_{H^{1}}$ into two items:
	\begin{gather*}
		\left\|\frac{-\lambda}{(1+u^{(n)})^{2}} \right\|_{L^{2}}\le 16\lambda \kappa^{-2}|\Omega|^{\frac{1}{2}},\\
		\left\| \frac{\lambda}{(1+u^{(n)})^{3}}\nabla u \right\|_{L^{2}}\le 64\lambda \kappa^{-3} \max\left\{1,T\right\}\left\|u^{(n)} \right\|_{H_{D}^{2}}\le \max\left\{1,T\right\}\frac{64\lambda(1-\kappa)}{\kappa^{3}C_{0}}.
	\end{gather*}
	Thus by the definition of $\left\|\cdot \right\|_{H^{1}}$ we get
	\begin{equation}\label{eqnew35.1}
		\left\|\frac{\lambda}{(1+u^{(n)})^{2}} \right\|_{H^{1}}\le C_{3}.
	\end{equation}
	where $C_{3}=C_{3}(\lambda,\kappa,\Omega,C_{0},T)$. since the embedding $H^{2}(\Omega)\hookrightarrow L^{2}(\Omega)$ is continuous, there is $C_{4}>0$ such that
	\begin{equation}\label{eqnew35.2}
		\left\|u^{(n)} \right\|_{L^{2}}\le C_{4}\left\|u^{(n)} \right\|_{H_{D}^{2}}\le \frac{C_{4}(1-\kappa/4)}{C_{0}}.
	\end{equation}
	Next we estimate $\big\|u_{t}^{(n)} \big\|_{H^{1}}$. Differentiate the equation in \eqref{eqnew35.3} with respect to $t$ , multiply the resultant by $u_{t}$, and integrat over $\Omega$, we get
	\begin{align}\label{eqnew35.11}
		\frac{1}{2}\frac{d}{dt}\big\|u_{t}^{(n)} \big\|_{L^{2}}^{2}+\big\|u_{t}^{(n)} \big\|_{H_{D}^{2}}^{2}=\int_{\Omega}\frac{2\lambda}{(1+u^{(n)})^{3}}(u_{t}^{(n)})^{2}dx\le 128\lambda \kappa^{-3}\big\|u_{t}^{(n)} \big\|_{L^{2}}^{2}.
	\end{align}
	The calculation details are as follows:
	\begin{equation*}
		\begin{aligned}
			&\int_{\Omega} (\Delta ^{2}u_{t}^{(n)})u_{t}^{(n)}dx\\
			&=\int_{\Omega} \Delta u_{t}^{(n)} \Delta u_{t}^{(n)}dx+\int_{\partial \Omega}u_{t}^{(n)}\frac{\partial (\Delta u_{t}^{(n)})}{\partial \nu}dS-\int_{\partial \Omega}\Delta u_{t}^{(n)} \frac{\partial (u_{t}^{(n)})}{\partial \nu}dS\\
			&=\int_{\Omega} |\Delta u_{t}^{(n)}|^{2}dx.\\
			&\int_{\Omega} -(\Delta u_{t}^{(n)}) u_{t}^{(n)}dx=\int_{\Omega}\nabla u_{t}^{(n)} \nabla u_{t}^{(n)}dx-\int_{\partial \Omega} u_{t}^{(n)} \frac{\partial (u_{t}^{(n)})}{\partial \nu}dS\\
			&=\int_{\Omega}|\nabla u_{t}^{(n)}|^{2}dx.
		\end{aligned}
	\end{equation*}
	By \eqref{eqnew35.6}, we get that for any $t>0$,
	\begin{align*}
		\int_{0}^{t}\big\|u_{t}^{(n)} \big\|_{L^{2}}^{2}d\tau&=E(u_{0}^{(n)})-E(u^{(n)})\\
		& \le 2E(u_{0}^{(n)})\\
		&\le \big\|u_{0}^{(n)} \big\|_{H_{D}^{2}}^{2}-\int_{\Omega}\frac{2\lambda}{1+u_{0}^{(n)}}dx\\
		&\le \frac{(1-\kappa/2)^{2}}{C_{0}^{2}}:=C_{5}.
	\end{align*}
	Mutiplying \eqref{eqnew35.11} by $t$, then adding $\big\|u_{t}^{(n)} \big\|_{L^{2}}^{2}$ to both sides yields
	\begin{equation*}
		\frac{d}{dt}\left(t\big\|u_{t}^{(n)} \big\|_{L^{2}}^{2}\right)+2t\big\|u_{t}^{(n)} \big\|_{H_{D}^{2}}^{2}\le \big\|u_{t}^{(n)} \big\|_{L^{2}}^{2}+256\lambda \kappa^{-3}t\big\|u_{t}^{(n)} \big\|_{L^{2}}^{2}.
	\end{equation*}
	Integrating the above inequality with respect to $t$ , we get for all $t>0$,
	\begin{align*}
		t\big\|u_{t}^{(n)} \big\|_{L^{2}}^{2}+2\int_{0}^{t}\tau\big\|u_{t}^{(n)} \big\|_{H_{D}^{2}}^{2}d\tau&\le \int_{0}^{t}\big\|u_{t}^{(n)} \big\|_{L^{2}}^{2}d\tau+256\lambda \kappa^{-3}\int_{0}^{t}\tau\big\|u_{t}^{(n)} \big\|_{L^{2}}^{2}d\tau\\
		&\le C_{5}+256\lambda \kappa^{-3}t\int_{0}^{t}\big\|u_{t}^{(n)} \big\|_{L^{2}}^{2}d\tau\\
		&\le C_{5}+256\lambda\kappa^{-3}C_{5}t.
	\end{align*}
	Thus, for $t\ge \delta>0$,
	\begin{equation}\label{eqnew35.12}
		\big\|u_{t}^{(n)} \big\|_{L^{2}}^{2}\le \frac{C_{5}}{t}+256\lambda \kappa^{-3}C_{5}\le C_{\delta}:=\frac{C_{5}}{\delta}+256\lambda \kappa^{-3}C_{5}.
	\end{equation}
	Note that
	\begin{align*}
		\big\|u_{t}^{(n)}(0) \big\|_{L^{2}}&=\left\|-Au_{0}^{(n)}-\frac{\lambda}{(1+u_{0}^{(n)})^{2}} \right\|_{L^{2}}\\
		&\le \big\|B\Delta^{2} u_{0}^{(n)}-T\Delta u_{0}^{(n)} \big\|_{L^{2}}+\left\|-\frac{\lambda}{(1+u_{0}^{(n)})^{2}} \right\|_{L^{2}}\\
		&\le C_{6}\big\|u_{0}^{(n)} \big\|_{H^{4}}+4\lambda \kappa^{-2}|\Omega|^{\frac{1}{2}}\\
		&\le \big\|u_{0} \big\|_{H^{4}}+\varepsilon_{0}+4\lambda \kappa^{-2}|\Omega|^{\frac{1}{2}}:=C_{7},
	\end{align*}
	where $C_{7}=C_{7}(B,T,\lambda,\Omega,\kappa,u_{0})$. Integrate \eqref{eqnew35.11} with respect to $t$, we get for all $t>0$,
	\begin{align*}
		\frac{1}{2}\big\|u_{t}^{(n)} \big\|_{L^{2}}^{2}+\int_{0}^{t}\big\|u_{t}^{(n)} \big\|_{H_{D}^{2}}^{2}d\tau&\le \frac{1}{2}\big\|u_{t}^{(n)}(0) \big\|_{L^{2}}^{2}+128\lambda \kappa^{-3}\int_{0}^{t}\big\|u_{t}^{(n)} \big\|_{L^{2}}^{2}d\tau\\
		&\le \frac{C_{7}^{2}}{2}+128\lambda \kappa^{-3}C_{5}.
	\end{align*}
	This means
	\begin{equation}\label{eqnew35.20}
		\int_{0}^{t}\big\|u_{t}^{(n)} \big\|_{H_{D}^{2}}^{2}d\tau\le \frac{C_{7}^{2}}{2}+128\lambda \kappa^{-3}C_{5}:=C_{8},\quad \forall t>0.
	\end{equation}
	Differentiating the equation in \eqref{eqnew35.3} $t$ with respect to $t$, then multiplying the resultant by $u_{tt}^{(n)}$, and integrating over $\Omega$, we get
	\begin{align*}
		\big\|u_{tt}^{(n)} \big\|_{L^{2}}^{2}+\frac{1}{2}\frac{d}{dt}\big\|u_{t}^{(n)} \big\|_{H_{D}^{2}}^{2}&=\int_{\Omega}\frac{2\lambda}{(1+u^{(n)})^{3}}u_{t}^{(n)}u_{tt}^{(n)}dx\\
		&\le 128\lambda \kappa^{-3}\int_{\Omega}|u_{t}^{(n)}| |u_{tt}^{(n)}|dx\\
		&\le 128\lambda \kappa^{-3}\int_{\Omega}\left(\varepsilon^{2}|u_{t}^{(n)}|^{2}+\frac{1}{4\varepsilon^{2}}|u_{tt}^{(n)}|^{2}\right)dx,
	\end{align*}
	where Young's inequality is employed. Let $\varepsilon=8\sqrt{\lambda \kappa^{-3}}$, we get
	\begin{equation}\label{eqnew35.13}
		\big\|u_{tt}^{(n)} \big\|_{L^{2}}^{2}+\frac{d}{dt}\big\|u_{t}^{(n)} \big\|_{H_{D}^{2}}^{2}\le 128\lambda^{2}\kappa^{-6}	\big\|u_{t}^{(n)} \big\|_{L^{2}}^{2}.
	\end{equation}
	Similarly, mutiply \eqref{eqnew35.13} by $t$, add $\big\|u_{t}^{(n)} \big\|_{H_{D}^{2}}^{2}$ to both sides, and integrate with respect to $t$, we get for all $t>0$,
	\begin{align*}
		\int_{0}^{t}\tau	\big\|u_{tt}^{(n)} \big\|_{L^{2}}^{2}+t\big\|u_{t}^{(n)} \big\|_{H_{D}^{2}}^{2}&\le \int_{0}^{t}\big\|u_{t}^{(n)} \big\|_{H_{D}^{2}}^{2}d\tau+128\lambda^{2}\kappa^{-6}\int_{0}^{t}\tau\big\|u_{t}^{(n)} \big\|_{L^{2}}^{2}d\tau\\
		&\le C_{8}+128\lambda^{2}\kappa^{-6}t\int_{0}^{t}\big\|u_{t}^{(n)} \big\|_{L^{2}}^{2}d\tau\\
		&\le C_{8}+128\lambda^{2}\kappa^{-6}C_{5}t.
	\end{align*}
	Then for all $t\ge \delta>0$,
	\begin{equation*}
		\big\|u_{t}^{(n)} \big\|_{H_{D}^{2}}^{2}\le \frac{C_{8}}{t}+128\lambda^{2}\kappa^{-6}C_{5}\le \tilde{C}_{\delta}:= \frac{C_{8}}{\delta}+128\lambda^{2}\kappa^{-6}C_{5}.
	\end{equation*}
	Since $H^{2}(\Omega)\hookrightarrow H^{1}(\Omega)$, there is a constant $C_{9}>0$ such that
	\begin{equation}\label{eqnew35.24}
		\big\|u_{t}^{(n)} \big\|_{H^{1}}^{2}\le C_{9}\big\|u_{t}^{(n)} \big\|_{H_{D}^{2}}^{2}.
	\end{equation}
	Therefore, combining \eqref{eqnewnew1000}-\eqref{eqnew35.2} with \eqref{eqnew35.24}, for fixed $\delta>0$,
	\begin{equation}\label{eqnew35.25}
		\big\|u^{(n)} \big\|_{H^{5}}\le C_{10},\quad  \forall t\ge \delta.
	\end{equation}
	Let $w=u^{(n)}-u$, then $w$ satisfies
	\begin{equation}\label{eqnew100.1}
		\left.\left\{\begin{array}{ll} w_{t}+B\Delta^{2}w-T\Delta w=Fw,\quad &t>0, \ x\in\Omega ,\\
			w=\partial_\nu w=0,&  x\in\partial \Omega ,\\
			w(0,\cdot)=u_{0}^{(n)}-u_{0},& x\in \Omega,\end{array}\right.\right.
	\end{equation}
	where
	\begin{equation*}
		F=\frac{\lambda(2+u^{(n)}+u)}{(1+u^{(n)})^{2}(1+u)^{2}}.
	\end{equation*}
	Multiplying \eqref{eqnew100.1} by $w$, then integrating over  $\Omega$ yields
	\begin{align*}
		\frac{1}{2}\frac{d}{dt}\left\| w\right\|_{L^{2}}+\left\| w\right\|_{H_{D}^{2}}^{2}&=\int_{\Omega}Fw^{2}dx\\
		&\le|F|\left\|w \right\|_{L^{2}}^{2}\\
		&\le \frac{\lambda(4-\frac{5}{4}\kappa)}{\kappa^{2}(\frac{\kappa}{4})^{2}}\left\|w \right\|_{L^{2}}^{2}\\
		&= \lambda\kappa^{-4} (64-20\kappa)\left\|w \right\|_{L^{2}}^{2}\\
		&\le C_{10}\left\|w \right\|_{L^{2}}^{2}.
	\end{align*}
	By Gronwall inequality (参见 \cite[Section~1.3.6]{ref3}), for all $t\in [0,\tilde{T}_{0}]$,
	\begin{equation*}
		\big\|w(t) \big\|_{L^{2}}=\big\|u^{(n)}(t)-u(t) \big\|_{L^{2}}^{2}\le e^{2C_{10}\tilde{T}_{0}}\big\|u_{0}^{(n)}-u_{0} \big\|_{L^{2}}^{2}\to 0.
	\end{equation*}
	Thus, $u_{n}\to u$ in $L^{2}(\Omega)$. Since $\big\| u^{(n)}\big\|_{H^{5}}$ is uniformly boundness, then by Eberlein-Sumulian Theorem, there is a sequence $u^{(nk)}$ such that $u^{(nk)}\rightharpoonup \tilde{u}$ ( ``$\rightharpoonup$" denotes the weak convergence in Banach space). Since $H^{5}(\Omega)\hookrightarrow L^{2}(\Omega)\cong (L^{2}(\Omega))'\hookrightarrow (H^{5}(\Omega))'$, $u^{(nk)}\rightharpoonup \tilde{u}$ in $L^{2}(\Omega)$. Note that $u_{n}\to u$ in $L^{2}(\Omega)$, we obtain that $\tilde{u}=u$. It then follows from weak lower semi continuity of norm(\cite[Remark~1.5.4]{refnew99}) that
	\begin{equation*}
		\big\|u \big\|_{H^{5}}=	\big\|\tilde{u} \big\|_{H^{5}}\le \liminf_{n \to \infty}	\big\|u^{(nk)} \big\|_{H^{5}}\le C_{10}.
	\end{equation*}
	Therefore we obtain the uniform boundness of $\left\| \right\|_{H^{5}}$. Since $H^{5}(\Omega)\hookrightarrow H^{4}(\Omega)$, for fixed $t_{1}>0$, we get
	\begin{equation*}
		\bigcup_{t\ge t_{1}}S(t)u_{0}
	\end{equation*}
	is relatively compact in $X(\kappa)$. Therefore we can deduce from Definition~\ref{predef2.8} that $(X(\kappa),S(t),E)$ is a gradient system.
\end{proof}
\subsection{Lojasiewicz-Simon inequality}\label{paowuLSbudengshi}

We define the $\omega$-limit set by
\begin{equation*}
	\omega(u_{0})=\left\{\psi :\exists\,t_{n}\to +\infty \ \text{such that}\,\lim_{n\to \infty}\left\|u(t_{n},\cdot)-\psi \right\|_{H_{D}^{4}}=0\right\}.
\end{equation*}
By lemma~\ref{eqnewlemma35.1} and Theorem~\ref{prethm2.4}, $\omega(u_{0})$ is consists of equilibrium, i.e. there are $\psi\in \omega(u_{0})\subset\mathcal{S}$ and a sequence $t_{n}\to +\infty$, such that
\begin{equation*}
	u(t_{n},\cdot)\to \psi\,\,\text{in $X(\kappa)$}.
\end{equation*}
Since $H^{4}(\Omega)\hookrightarrow L^{\infty}(\Omega)$, we get $u(t_{n},\cdot)\to \psi$ in $L^{\infty}(\Omega)$. Let $\varepsilon_{2}=\frac{\kappa}{2}$, then there  is $n_{2}\in \mathbb{N}^{*}$ such that when $n\ge N$,
\begin{equation*}
	\left\|u(t_{n})-\psi \right\|_{L^{\infty}}\le \frac{\kappa}{2}.
\end{equation*}
By the triangle inequality,
\begin{equation*}
	\left\|\psi \right\|_{L^{\infty}}\le 	\left\| u(t_{n}) \right\|_{L^{\infty}}+\frac{\kappa}{2}\le 1-\frac{\kappa}{2}.
\end{equation*}
Fix $\psi\in \omega(u_{0})\subset\mathcal{S}$, we consider the linearized problem \eqref{neweq1.1} at $\psi$:
\begin{equation}\label{linearhua}
	\left.\left\{\begin{array}{ll} Lw\equiv B\Delta^2w-T\Delta w-\frac{2\lambda w}{(1+\psi)^3}=0,\quad &x\in\Omega ,\\
		w=\partial_\nu w=0,& x\in\partial\Omega,\end{array}\right.\right.
\end{equation}
For any $h\in L^{2}(\Omega)$, consider the equation
\begin{equation}\label{linearhua1}
	\left.\left\{\begin{array}{ll} Lw+\mu w=h,\quad &x\in\Omega ,\\
		w=\partial_\nu w=0,& x\in\partial\Omega,\end{array}\right.\right.
\end{equation}
and the homogeneous equation
\begin{equation}\label{linearhua2}
	\left.\left\{\begin{array}{ll} Lw+\mu w=0,\quad &x\in\Omega ,\\
		w=\partial_\nu w=0,& x\in\partial\Omega,\end{array}\right.\right.
\end{equation}
Multiplying the equation in \eqref{linearhua2} by $w$, then integrating over $\Omega$ yields
\begin{equation*}
	\int_{\Omega}\left(B|\Delta w|^{2}+T|\nabla w|^{2}+\left(-\frac{2\lambda}{(1+\psi)^{3}}+\mu\right)|w|^{2}\right)dx=0.
\end{equation*}
Note that $\left\|\psi \right\|_{L^{\infty}}\le 1-\frac{\kappa}{2}$, we have
\begin{equation*}
	-\frac{2\lambda}{(1+\psi)^{3}}\ge -16\lambda \kappa^{-3}.
\end{equation*}
Thus for any constant $\mu>16\lambda\kappa^{-3}$, problem \eqref{linearhua2} has only the trivial solution $w\equiv 0$. Then by classical $L^{p}$ theory in elliptic equation(\cite[Corollary~2.21]{refnew101}), problem \eqref{linearhua1} admits a unique solution $w\in H^{4}(\Omega)\cap H_{0}^{2}(\Omega)=H_{D}^{4}(\Omega)$. Based on this, we will prove the following Fredholm alternative result:
\begin{lemma}\label{fredholm 2zeyi}
	Precisely one of the following two situations holds:
	\begin{itemize}
		\item[(\romannumeral1)] for any $f\in L^{2}(\Omega)$, there is unique $w$ satisfying
		\begin{equation}\label{linearhua3}
			\left.\left\{\begin{array}{ll} Lw=f,\quad &x\in\Omega ,\\
				w=\partial_\nu w=0,& x\in\partial\Omega,\end{array}\right.\right.
		\end{equation}
		\item[(\romannumeral2)]there is a weak solution $w\not \equiv 0$ satisfying
		\begin{equation}\label{linearhua4}
			\left.\left\{\begin{array}{ll} Lw=0,\quad &x\in\Omega ,\\
				w=\partial_\nu w=0,& x\in\partial\Omega,\end{array}\right.\right.
		\end{equation}
	\end{itemize}
	In addition, if (\romannumeral1) holds, then $Ker(L)=\left\{0\right\}$; If (\romannumeral2) holds, then $dim(Ker(L))$ is finite. Furthermore, the necessary and sufficient condition for the solvability of \eqref{linearhua3} is $f\in Ker(L)^{\perp}$.
\end{lemma}
\begin{proof}
	Our proof is motivated by \cite[Theorem~6.2.4]{refnew40}. Let $\gamma>16\lambda\kappa^{-3}$ and define the operator $L_{\gamma}w:=Lw+\gamma w$. Then we can deduce from the discussion about \eqref{linearhua1} that, for any $h\in L^{2}(\Omega)$, $w\in H_{D}^{4}$ satisfies
	\begin{equation*}
		L_{\gamma}w=h.
	\end{equation*}
	We rewrite this equation as
	\begin{equation*}
		w=L_{\gamma}^{-1}h.
	\end{equation*}
	Note that $w\in H_{D}^{4}(\Omega)$ is the solution to \eqref{linearhua3} if and only if
	\begin{equation}\label{linearhua100}
		Lw+\gamma w=f+\gamma w,
	\end{equation}
	i.e.
	\begin{equation}\label{linearhua101}
		w=L_{\gamma}^{-1}(f+\gamma w).
	\end{equation}
	We rewrite this equation as
	\begin{equation}\label{linearhua102}
		w-Kw=g,
	\end{equation}
	where
	\begin{equation}\label{linearhua103}
		Kw:=\gamma L_{\gamma}^{-1}w,\quad g:=L_{\gamma}^{-1}f.
	\end{equation}
	We rewrite \eqref{linearhua101} as \eqref{linearhua102} to employ Fredholm Theorem (\cite[Appendix D: Theorem~5]{refnew40}). We claim that $K:L^{2}(\Omega)\to L^{2}(\Omega)$ is a linear compact operator.
	\begin{itemize}
		\item[(1)] Linearity: For any $f_{1},f_{2}\in L^{2}(\Omega)$, let $Kf_{1}=w_{1}$, $Kf_{2}=w_{2}$, then we have
		\begin{align*}
			K(f_{1}+f_{2})&=K\left(\frac{L_{\gamma}w_{1}}{\gamma}+\frac{L_{\gamma}w_{2}}{\gamma}\right)\\
			&=\frac{1}{\gamma}K(L_{\gamma}(w_{1}+w_{2}))\\
			&=\frac{1}{\gamma}\gamma L_{\gamma}^{-1}\gamma(w_{1}+w_{2})\\
			&=w_{1}+w_{2}=Kf_{1}+Kf_{2}.
		\end{align*}
		\item[(2)] Boundness: for any $f_{3}\in L^{2}(\Omega)$, Let $Kf_{3}=w_{3}$, then $w_{3}\in H_{D}^{4}$. Note that $Ker(L_{\gamma})=\left\{0\right\}$, by the priori estimate of ellptic equation, there is a constant $C_{1}>0$ such that
		\begin{equation*}
			\left\|w_{3} \right\|_{H^{4}}\le C_{1}\left\|f_{3} \right\|_{L^{2}}.
		\end{equation*}
		Note that $Kf_{3}=w_{3}$, and $H^{4}(\Omega)\hookrightarrow L^{2}(\Omega)$, then there is a constant $C_{2}>0$ such that
		\begin{equation*}
			\left\|Kf_{3}\right\|_{L^{2}}\le C_{2}\left\|Kf_{3}\right\|_{H^{4}}\le C_{1}C_{2}\left\|f_{3} \right\|_{L^{2}}.
		\end{equation*}
		Thus $K$ is bounded.
		\item[(3)] Compactness: Since the embedding $H^{4}(\Omega)\hookrightarrow L^{2}(\Omega)$ is compact, we immediately obtain that $K$ is compact.
	\end{itemize}
	Then by Fredholm Theorem, precisely one of the following two situations holds:
	\begin{itemize}
		\item[($\alpha$)] For every $h\in L^{2}(\Omega)$, the equation $w-Kw=h$ admits a unique solution $w\in L^{2}(\Omega)$.
		\item[($\beta$)] The equation $w-Kw=0$ has nontrivial solution $w\in L^{2}(\Omega)$.
	\end{itemize}
	If $(\alpha)$ holds, then by \eqref{linearhua100}-\eqref{linearhua103}, problem \eqref{linearhua3} admits a unique solution, at this time (\romannumeral2) does not hold, we obtain $Ker(L)=\left\{0\right\}$. On the other hands, if $(\beta)$ hold, then $dim(Ker(L))$ is finite. In addition, we need to show that the sufficient and necessary condition for solvability of \eqref{linearhua3} is $f\in Ker(L)^{\perp}$. Thus, we prove $L_{\gamma}$ defined on $H_{D}^{4}(\Omega)\subset L^{2}(\Omega)$ is a self
	adjoint operator. In fact we can show that $-L_{\gamma}$ is selfadjoint with upper bound 0.(\cite[Theorem~B.14]{refnew137}).
	\begin{itemize}
		\item[(a)] Symmetry: For any $w,v\in D(-L_{\gamma})=H_{D}^{4}(\Omega)$,
		\begin{align*}
			(-L_{\gamma}w,v)_{L^{2}}&=\int_{\Omega}\left(-B\Delta ^{2}w+T\Delta w+\frac{2\lambda w}{(1+\psi)^{3}}-\gamma w\right)vdx\\
			&=\int_{\Omega}\left(-B\Delta w \Delta v-T\nabla w \nabla v+\frac{2\lambda wv}{(1+\psi)^{3}}-\gamma wv\right)dx\\
			&=\int_{\Omega}\left(-B\Delta ^{2}v+T\Delta v+\frac{2\lambda v}{(1+\psi)^{3}}-\gamma v\right)wdx\\
			&=(w,-L_{\gamma}v)_{L^{2}}.
		\end{align*}
		\item[(b)] Upper bound 0: For any $v\in H_{D}^{4}(\Omega)$,
		\begin{align*}
			(-L_{\gamma}v,v)_{L^{2}}&=\int_{\Omega}\left(-B\Delta ^{2}v+T\Delta v+\frac{2\lambda v}{(1+\psi)^{3}}-\gamma v\right)vdx\\
			&=\int_{\Omega}\left(-B|\Delta v|^{2}-T|\nabla v|^{2}+\left(\frac{2\lambda}{(1+\psi)^{3}}-\gamma\right)|v|^{2}\right)dx\\
			&\le 0 \le (w,w)_{L^{2}}.
		\end{align*}
		\item[(c)] There is $\lambda>0$ such that $Ran(\lambda-(-L_{\gamma}))=L^{2}(\Omega)$ (where $Ran$ denotes image space), by the discussion above, we can get for all $\mu>0$, $Ran(\mu+L_{\gamma})=L^{2}(\Omega)$.
	\end{itemize}
	Therefore we can deduce from \cite[Theorem~B.14]{refnew137} that $-L_{\gamma}$ is selfadjoint with upper bound 0, which means $-L_{\gamma}$ is self-adjoint. Then $K$ is self-adjoint. Then by Fredholm Theorem, $Ran(I-K)=Ker(I-K^{*})^{\perp}=Ker(I-K)^{\perp}$. Thus, $Ker(I-K)=\left\{0\right\}$ if and only if $Ran(I-K)=L^{2}(\Omega)$. This means the the the sufficient and necessary condition for solvability of \eqref{linearhua3} is $f\in Ker(L)^{\perp}$. Lemma~\ref{fredholm 2zeyi} is proved.
\end{proof}

Now we are ready to prove the core lemma in this subsection, which establishs the Lojasiewicz-Simon inequality for problem \eqref{neweq1.1}.
\begin{lemma}\label{LSbudengshipaowupaowu}
	There are constants $\sigma$ and $\theta\in (0,\frac{1}{2})$ depending on $\psi$ such that for all $u\in X(\kappa)$ satisfying $\left\|u-\psi \right\|_{H_{D}^{4}}<\sigma$,
	\begin{equation}
		\left\|B\Delta^{2}u-T\Delta u+\frac{\lambda}{(1+u)^{2}} \right\|_{L^{2}}\ge |E(u)-E(\psi)|^{1-\theta}.
	\end{equation}
	where $E$ is defined by \eqref{eqnewnew555}.
\end{lemma}
\begin{proof}
	Hereafter we use $C$ to denote a positive constant that may vary in different places. Based on Lemma~\ref{fredholm 2zeyi}, We will discuss in two situations:
	1. $Ker(L)=\left\{0\right\}$. Let $u=v+\psi$. Consider the map
	\begin{equation*}
		M(v)=B\Delta ^{2}u-T\Delta u+\frac{\lambda}{(1+u)^{2}}: H_{D}^{4}(\Omega)\to L^{2}(\Omega).
	\end{equation*}
	A direct calculation yields that $M\in C^{1}(H_{D}^{4}(\Omega),L^{2}(\Omega))$ and for any $v,h\in H_{D}^{4}(\Omega)$,
	\begin{equation}\label{eqnew39.1}
		DM(v)(h)=B\Delta^{2}h-T\Delta h-\frac{2\lambda h}{(1+u)^{3}}.
	\end{equation}
	where $DM$ denotes the Fr{\'e}chet derivative of $M$, and we get $DM(0)=L$. By the definition of $L$ and the regularity of $u$, we obtain that $L\in \mathcal{L}(H_{D}^{4}(\Omega),L^{2}(\Omega))$. By Lemma~\ref{fredholm 2zeyi}, $Ker(L)={0}$ and $Ran(L)=L^{2}(\Omega)$, which means $L$ is bijective. It then follows from Banach inverse theorem, $L\in \mathrm{Inv}(H_{D}^{4}(\Omega),L^{2}(\Omega))$. Thus, by the local inverse theorem, there is a neighborhood $W_{1}$ of the origin in $H_{D}^{4}(\Omega)$ and a neighborhood $W_{2}$ of the origin in $L^{2}(\Omega)$, and the map
	\begin{equation*}
		\Psi:W_{2}\to W_{1}
	\end{equation*}
	such that
	\begin{gather*}
		M(\Psi(g))=g, \quad \forall g\in W_{2},\\
		\Psi(M(v))=v,\quad \forall v\in W_{1}.
	\end{gather*}
	In addition, there is $C>0$ such that
	\begin{gather}
		\left\|\Psi(g_{1})-\Psi(g_{2}) \right\|_{H_{D}^{4}}\le C\left\|g_{1}-g_{2} \right\|_{L^{2}},\quad \forall g_{1},g_{2\in W_{2}},\label{eqnew39.2}\\
		\left\|M(v_{1})-M(v_{2}) \right\|_{L^{2}}\le C\left\|v_{1}-v_{2} \right\|_{H_{D}^{4}},\quad \forall v_{1},v_{2\in W_{1}}.\label{eqnew39.3}
	\end{gather}
	For $w\in H_{D}^{4}(\Omega)$, by integral by parts, we get
	\begin{align*}
		DE(w)\cdot v&=\int_{\Omega}\left(B\Delta w \Delta v+T\nabla w \nabla v+\lambda \frac{v}{(1+w)^{2}}\right)dx\\
		&=\int_{\Omega}\left(B\Delta^{2}w-T\Delta w+\frac{\lambda}{(1+w)^{2}}\right)vdx.
	\end{align*}
	Then we can deduce from \eqref{eqnew39.1} and \eqref{eqnew39.3} that
	\begin{align}
		|E(u)-E(\psi)|&=\left|\int_{0}^{1}\frac{d}{dt}E(tu+(1-t)\psi)dt \right|\notag\\
		&=\left|DE(tu+(1-t)\psi)\cdot(u-\psi)dt \right|\notag\\
		&\le \max_{0\le t\le 1}\left\|M(tu+(1-t)\psi) \right\|_{L^{2}}\left\|u-\psi \right\|_{L^{2}}\notag\\
		&\le C\left\|u-\psi \right\|_{H_{D}^{4}}\left\|u-\psi \right\|_{L^{2}}\notag\\
		&\le C\left\|u-\psi \right\|_{H_{D}^{4}}^{2}\label{eqnew39.4}.
	\end{align}
	Since $\psi=\Psi(0)$, $u=\Psi(M(v))$, we have
	\begin{equation}\label{eqnew39.5}
		\left\|u-\psi \right\|_{H_{D}^{4}}=\left\|\Psi(M(v))-\Psi(0) \right\|_{H_{D}^{4}}\le C\left\|M(v) \right\|_{L^{2}}.
	\end{equation}
	Then it follows from \eqref{eqnew39.4} and \eqref{eqnew39.5} that
	\begin{equation*}
		|E(u)-E(\psi)|^{\frac{1}{2}}\le C\left\|M(v) \right\|_{L^{2}}.
	\end{equation*}
	This indicates
	\begin{equation*}
		\left\|M(v) \right\|_{L^{2}}\ge C|E(u)-E(\psi)|^{\frac{1}{2}}.
	\end{equation*}
	Let $\varepsilon>0$ and $\sigma$ be small enough such that $\left\|v \right\|_{H_{D}^{4}}<\sigma$, we obtain that
	\begin{equation*}
		C\left|E(u)-E(\psi) \right|^{-\varepsilon}\ge 1.
	\end{equation*}
	Thus,
	\begin{equation*}
		\left\|M(v) \right\|_{L^{2}}\ge \left|E(u)-E(\psi)\right|^{1-\tilde{\theta}},
	\end{equation*}
	where
	\begin{equation*}
		0<\tilde{\theta}=\theta-\varepsilon<\frac{1}{2}.
	\end{equation*}
	
	2. $dim(Ker(L))=m>0$ : Recall that the the the sufficient and necessary condition for solvability of equation \eqref{linearhua3} is
	\begin{equation*}
		f\in (Ker(L))^{\perp}.
	\end{equation*}
	Let $(\phi_{1},\cdots,\phi_{m})$ denote normalized orthogonal basis of $Ker(L)$ in $L^{2}(\Omega)$, let $\Pi$ be the projection from $L^{2}(\Omega)$ to $Ker(L)$. Now we define
	\begin{align*}
		\tilde{\mathcal{L}}:H_{D}^{4}(\Omega)&\to L^{2}(\Omega)\\
		w &\mapsto \Pi w+Lw.
	\end{align*}
	Then $\tilde{\mathcal{L}}$ is bijective. Similar to case 1, we define $u=v+\psi$ and
	\begin{equation}\label{eqnewnew39.3}
		M(v)=B\Delta ^{2}u-T\Delta u+\frac{\lambda}{(1+u)^{2}}: H_{D}^{4}(\Omega)\to L^{2}(\Omega).
	\end{equation}
	Recall that $DM(0)=L$, let
	\begin{align}
		\mathcal{N}:H_{D}^{4}(\Omega)&\to L^{2}(\Omega)\notag\\
		v &\mapsto M(v)+\Pi v. \label{eqnewnew39.4}
	\end{align}
	Thus, $D\mathcal{N}(0)=\tilde{\mathcal{L}}$. Then by local inverse theorem, there is a neighborhood $W_{1}$ of the origin
	in $H_{D}^{4}(\Omega)$ and a neighborhood $W_{2}$ of the origin in $L^{2}(\Omega)$, and the map
	\begin{equation*}
		\Psi:W_{2}\to W_{1}
	\end{equation*}
	such that
	\begin{gather*}
		\mathcal{N}(\Psi(g))=g, \quad \forall g\in W_{2},\\
		\Psi(\mathcal{N}(v))=v,\quad \forall v\in W_{1}.
	\end{gather*}
	In addition, there is a constant $C>0$ such that
	\begin{gather}
		\left\|\Psi(g_{1})-\Psi(g_{2}) \right\|_{H_{D}^{4}}\le C\left\|g_{1}-g_{2} \right\|_{L^{2}},\quad \forall g_{1},g_{2}\in W_{2},\label{eqnew399.1}\\
		\left\|\mathcal{N}(v_{1})-\mathcal{N}(v_{2}) \right\|_{L^{2}}\le C\left\|v_{1}-v_{2} \right\|_{H_{D}^{4}},\quad \forall v_{1},v_{2}\in W_{1}.\label{eqnew399.2}
	\end{gather}
	To employ the Lojasiewicz-Simon inequality in $\mathbb{R}^{m}$ , i.e. Theorem~\ref{yibandeLS}, we will prove $M(v)$ is analytic in a neighborhood of $0$. We will use Theorem~\ref{prethm2.6} and Definition~\ref{predef2.10}.
	
	A direct calculation yields
	\begin{equation*}
		D^{(n)}M(0)(v_{1},\cdots,v_{n})=(-1)^{n}\frac{(n+1)!(v_{1}v_{2}\cdots v_{n})}{(1+\psi)^{n+2}},\quad n=2,3,\cdots.
	\end{equation*}
	Let $v_{n}=v$ ($n=2,3,\cdots$). By the property of Fr{\'e}chet derivative ( \cite[Chapter~1]{ref22}), for any $n\ge 1$, $D^{n}M(0)\in \mathcal{B}_{n}(H_{D}^{4}(\Omega),L^{2}(\Omega))$ is symmetry. In addition, by theorem~\ref{prethm2.6}, we know that
	\begin{align}
		M(0+v)=M(0)+\sum_{k=1}^{n-1}\frac{1}{k!}D^{k}M(0)v^{k}+R_{n},\label{eqnew399.3}\\
		R_{n}=\int_{0}^{1}\frac{(1-\tau)^{n-1}}{(n-1)!}D^{k}M(0+\tau v)v^{n}d\tau \notag,
	\end{align}
	where $D^{k}M(0)v^{k}:=D^{k}M(0)(v,\cdots,v)$. Thus we need to show that when $n\to +\infty$, the remainder $R_{n}\to 0$. In fact, since $\int_{0}^{1}(1-\tau)^{n-1}d\tau=\frac{1}{n}$, we have
	\begin{align*}
		\left\| R_{n}\right\|_{L^{2}}&\le \frac{1}{n!}\sup_{0\le \tau \le 1} \left \| M^{(n)}(0+\tau v)v^{n} \right \| _{L^{2}}\\
		&\le \sup_{0\le \tau \le 1}\left\| \frac{n+1}{(1+0+\tau v)^{n+2}}v^{n}\right\|_{L^{2}}.
	\end{align*}
	Since the embedding $H^{4}(\Omega)\hookrightarrow L^{\infty}(\Omega)$ is continuous, in fact we can choose $\left\|v \right\|_{H_{D}^{4}}$ be small enough such that $\left\|v \right\|_{L^{\infty}}<\kappa$, and for all $0\le \tau \le 1$
	\begin{equation*}
		1+0+\tau v>\kappa.
	\end{equation*}
	At this time we have
	\begin{equation*}
		\sup_{0\le \tau \le 1}\left|\frac{v}{1+0 +\tau v} \right|<1.
	\end{equation*}
	This indicates that when $\|v\|_{L^{\infty}}<\frac{\kappa}{2}$,
	\begin{equation*}
		\lim_{n \to +\infty}\left\| R_{n}\right\|_{L^{2}}\le 4\kappa ^{-2}|\Omega|^{\frac{1}{2}}	\lim_{n \to +\infty}\left[\left(n+1\right)\sup_{0\le \tau \le 1}\left|\frac{v}{1+0 +\tau v} \right|^{n}\right] = 0.
	\end{equation*}
	Thus, if $\left\|v \right\|_{H_{D}^{4}}$ is small enough, the convergence of \eqref{eqnew399.3} with respect to $v\in H_{D}^{4}(\Omega
	)$ is uniform. Therefore $M(v)$ is analytic at $0$. Then it follows from local inverse theorem in Banach space (\cite[Section~4.3 Standard Example 9]{newrefzei}) that $\Psi$ is analytic at the origin of $L^{2}(\Omega)$. Let
	\begin{equation*}
		\xi=(\xi_{1},\cdots,\xi_{m})\in \mathbb{R}^{m},\quad \Pi v=\sum_{j=1}^{m}\xi_{j} \phi_{j}.
	\end{equation*}
	We can choose $\xi$ be small enough such that
	\begin{equation*}
		\sum_{j=1}^{m}\xi_{j} \phi_{j}\in W_{2}.
	\end{equation*}
	Define $\Gamma$:
	\begin{equation*}
		\Gamma(\xi)=E(\Psi(\sum_{j=1}^{m}\xi_{j} \phi_{j})+\psi).
	\end{equation*}
	Similarly, we can verify $E$ is analytic. By the chain rule, we get $\Gamma$ is analytic at the origin of $\mathbb{R}^{m}$. Furthermore, by \eqref{eqnew399.1} and \eqref{eqnew399.2}, we obtain that
	\begin{gather*}
		D\mathcal{N}(v)\in \mathcal{L}(H_{D}^{4}(\Omega),L^{2}(\Omega)),\quad \forall v\in W_{1},\\
		D\Psi(g)\in \mathcal{L}(L^{2}(\Omega),H_{D}^{4}(\Omega)),\quad \forall g\in W_{2}.
	\end{gather*}
	For any $j\in \left\{1,\cdots,m\right\}$, we have
	\begin{equation}\label{eqnewnew39.1}
		\frac{\partial \Gamma}{\partial \xi_{j}}=DE(\Psi(\sum_{j=1}^{m}\xi_{j} \phi_{j})+\psi)\cdot D\Psi \cdot \phi_{j}.
	\end{equation}
	Similar to case 1, for $w,v\in H_{D}^{4}(\Omega)$,
	\begin{equation}\label{eqnewnew39.2}
		DE(w)\cdot v =\int_{\Omega}\left(B\Delta^{2}w-T\Delta w+\frac{\lambda}{(1+w)^{2}}\right)vdx.
	\end{equation}
	Since $\psi \in \mathcal{S}$, by the definition of $\mathcal{S}$, $\nabla \Gamma (0)=0$. Thus, for all $j$, $1\le j\le m$, we can deduce from \eqref{eqnewnew39.3}, \eqref{eqnewnew39.1}, \eqref{eqnewnew39.2} and $|\phi_{j}|=1$ that
	\begin{align*}
		\left|\frac{\partial \Gamma}{\partial \xi_{j}} \right|&\le \|M(\Psi(\sum_{j=1}^{m}\xi_{j} \phi_{j})) \|_{L^{2}} \|D\Psi(\sum_{j=1}^{m}\xi_{j} \phi_{j}) \|_{\mathcal{L}(L^{2}(\Omega),H_{D}^{4}(\Omega)}\\
		&\le  C \|M(\Psi(\sum_{j=1}^{m}\xi_{j} \phi_{j})) \|_{L^{2}}.
	\end{align*}
Recall that $\Pi v=\sum_{j=1}^{m}\xi_{j} \phi_{j}$, we have
	\begin{align}
		|\nabla \Gamma(\xi)|&\le C\left\|M(\Psi(\Pi v)) \right\|_{L^{2}}\notag\\
		&\le \left\|M(\Psi(\Pi v))-M(v)+M(v) \right\|_{L^{2}}\notag\\
		&\le C(\left\|M(v) \right\|_{L^{2}}+\left\|M(\Psi(\Pi v))-M(v) \right\|_{L^{2}}).\label{eqeqnewnew39.1}
	\end{align}
	Then it follows from \eqref{eqnewnew39.4} and \eqref{eqnew399.2} that
	\begin{align}
		\|M(v_{1})-M(v_{2}) \|_{L^{2}}&\le\|\Pi v_{1}-\Pi v_{2} \|_{L^{2}}+\|\mathcal{N}(v_{1})-\mathcal{N}(v_{2}) \|_{L^{2}} \notag\\
		&\le \|v_{1}-v_{2} \|_{L^{2}} +C\|v_{1}-v_{2} \|_{H_{D}^{4}}\notag\\
		&\le C\|v_{1}-v_{2} \|_{H_{D}^{4}}.\label{eqnewnew39.5}
	\end{align}
	By \eqref{eqnew399.1}, we obtain
	\begin{equation}\label{eqeqnewnew39.2}
		\left\|\Psi(\Pi v)-v \right\|_{H_{D}^{4}}=\left\|\Psi (\Pi v)-\Psi (\mathcal{N}(v)) \right\|_{H_{D}^{4}}\le C\left\|\Pi v-\mathcal{N}(v) \right\|_{L^{2}}=C\left\|M(v) \right\|_{L^{2}}.
	\end{equation}
	Thus, combining \eqref{eqeqnewnew39.1}-\eqref{eqeqnewnew39.2} yields
	\begin{align}
		\left|\nabla \Gamma(\xi) \right|&\le C(\left\|M(v) \right\|_{L^{2}}+C\left\|\Psi (\Pi v)-v \right\|_{H_{D}^{4}})\notag\\
		&\le C(\left\|M(v) \right\|_{L^{2}}+C^{2}\left\|M(v) \right\|_{L^{2}})\notag\\
		&\le C\left\|M(v) \right\|_{L^{2}}\label{eqeqnewnew25.1}.
	\end{align}
	On the other hands,
	\begin{align}
		\left|E(u)-\Gamma(\xi) \right|&=\left|E(u)-E(\Psi(\Pi v)+\psi) \right|\notag\\
		&\le \left|\int_{0}^{1}\frac{d}{dt}E(u+(1-t)(\Psi (\Pi(v)-v)) \cdot (\Psi (\Pi(v)-v))dt \right|\notag\\
		&\le \left|\int_{0}^{1}DE(u+(1-t)(\Psi (\Pi v)-v))\cdot(\Psi (\Pi v)-v) dt\right|\notag\\
		&\le \max_{0\le t \le 1}\left\|M(v+(1-t)(\Psi (\Pi v)-v)) \right\|_{L^{2}} \left\|\Psi (\Pi v)-v \right\|_{L^{2}}\notag\\
		&\le (\left\|M(v) \right\|_{L^{2}}+C\left\|\Psi (\Pi v)-v \right\|_{H_{D}^{4}})\cdot C\left\|\Psi (\Pi v)-v \right\|_{H_{D}^{4}}\notag \\
		&\le (\left\|M(v) \right\|_{L^{2}}+C^{2}\left\|M(v) \right\|_{L^{2}})\cdot  C^{2}\left\|M(v) \right\|_{L^{2}}\notag\\
		&\le C\left\|M(v) \right\|_{L^{2}}^{2}\label{eqnewnew25.2}.
	\end{align}
	By Theorem~\ref{yibandeLS}, there are constants $\sigma>0$ and $\theta\in (0,\frac{1}{2})$ such that
	\begin{equation*}
		\left|\nabla \Gamma (\xi) \right|\ge \left|\Gamma(\xi)-\Gamma(0) \right|^{1-\theta}.
	\end{equation*}
	We can choose $\xi$ be small enough such that	\begin{equation*}
		|\xi|<\sigma \Longrightarrow  \Pi v=\sum_{j=1}^{m}\xi_{j}\phi_{j}\in W_{2}.
	\end{equation*}
	By the definition of $\Gamma$, $\Gamma(0)=E(\psi)$, this means
	\begin{equation}\label{eqneweqnew32.1}
		\left|\nabla \Gamma (\xi) \right|\ge \left|\Gamma(\xi)-E(\psi) \right|^{1-\theta}.
	\end{equation}
	Combining \eqref{eqeqnewnew25.1}-\eqref{eqneweqnew32.1} yields
	\begin{align}
		\left|E(u)-E(\psi) \right|^{1-\theta}&\le \left|E(u)-\Gamma(\xi)+\Gamma(\xi)-E(\psi) \right|^{1-\theta}\notag\\&\le \left|E(u)-\Gamma(\xi) \right|^{1-\theta}+\left|\Gamma(\xi)-E(\psi) \right|^{1-\theta}\notag\\&\le C^{1-\theta}\left\|M(v) \right\|_{L^{2}}^{2(1-\theta)}+C\left\|M(v) \right\|_{L^{2}}.\label{eqneweqnew25.3}
	\end{align}
	Since $\theta\in (0,\frac{1}{2})$, $2(1-\theta)>0$. Then when $v\in W_{1}$, we have
	\begin{equation}\label{eqeqeqnewnew26.6}
		\left\|M(v) \right\|_{L^{2}}^{2(1-\theta)}\le \left\|M(v) \right\|_{L^{2}}.
	\end{equation}
	Thus combining \eqref{eqneweqnew25.3} with \eqref{eqeqeqnewnew26.6}, we obtain
	\begin{equation*}
		\left\|M(v) \right\|_{L^{2}}\ge C\left|E(u)-E(\psi) \right|^{1-\theta}.
	\end{equation*}
	Let $\varepsilon>0$ be small enough, we can choose $\sigma$ small enough such that $\left\|v \right\|_{H_{D}^{4}}\le \sigma$. At this time we have
	\begin{equation*}
		C\left|E(u)-E(\psi) \right|^{-\varepsilon}\ge 1.
	\end{equation*}
	Thus we get
	\begin{equation*}
		\left\|M(v) \right\|_{L^{2}}\ge \left|E(u)-E(\psi) \right|^{1-\theta'},
	\end{equation*}
	where
	\begin{equation*}
		0<\theta'=\theta-\varepsilon<\frac{1}{2}.
	\end{equation*}
	Lemma~\ref{LSbudengshipaowupaowu} is proved.
\end{proof}
\subsection{Convergence rate}
In this subsection, we will prove the global solution to \eqref{neweq1.1} must converge to a stationary solution. Furthermore, we will establish the convergence rate. Based on the discussion above, we are now ready to prove the convergence of solution established in Theorem~\ref{theorem1.2}.
\begin{theorem}\label{shoulianxingpaowu}
	Under the assumptions of Theorem~\ref{theorem1.2}, let $u$ be the unique global solution to \eqref{neweq1.1} satisfying $u\in X(\kappa)$, then there exists $\psi\in \mathcal{S}$ such that
	\begin{equation*}
		\lim_{t \to+\infty}\left\|u(t,\cdot)-\psi \right\|_{H_{D}^{4}}=0.
	\end{equation*}
\end{theorem}
\begin{proof}
	Recall that in section~\ref{paowuLSbudengshi}, there are $\psi\in \omega(u_{0})\subset \mathcal{S}$ and a sequence $t_{n}\to +\infty$, such that
	\begin{equation*}
		u(t_{n},\cdot)\to \psi\,\,\textup{in }X(\kappa).
	\end{equation*}
	Let
	\begin{equation*}
		H(t)=E(u)-E(\psi).
	\end{equation*}
	where $E(u)$ is as \eqref{eqnewnew555}, and $\theta$ is as in Lemma~\ref{LSbudengshipaowupaowu}. Thus, we have
	\begin{align}
		-\frac{d}{dt}(H(t))^{\theta}&=-\theta (H(t))^{\theta-1}H'(t)\notag\\
		&-\theta (E(u)-E(\psi))^{\theta-1}\frac{d}{dt}(E(u)-E(\psi)).\label{eqwen16.7}
	\end{align}
	We can deduce from \eqref{neweq1.1}, \eqref{eqnew111} and \eqref{eqnewnew555} that, for any $u\in X(\kappa)$ satisfying $\left\|u-\psi \right\|_{H_{D}^{4}}<\sigma$,
	\begin{align}
		\frac{d}{dt}(E(u)-E(\psi))&=\frac{d}{dt}E(u)\notag\\
		&=-\left\| u_{t}\right\|_{L^{2}}^{2}\notag\\
		&=-\left\| u_{t}\right\|_{L^{2}} \left\| -B\Delta^{2}u+T\Delta u-\frac{\lambda}{(1+u)^{2}}\right\|_{L^{2}}\notag\\
		&\le -\left\|u_{t} \right\|_{L^{2}}\left|E(u)-E(\psi) \right|^{1-\theta}.\label{eqwen16.6}
	\end{align}
Combining \eqref{eqwen16.7} with \eqref{eqwen16.6} yields
	\begin{equation}\label{eqwen16.9}
		-\frac{d}{dt}(H(t))^{\theta}\ge \theta\left\|u_{t} \right\|_{L^{2}}.
	\end{equation}
	Since $u(t_{n})\to \psi$ in $X(\kappa)$, and $E(u)$ is continuous, we get $E(u(t_{n}))\to E(\psi)$. Thus, for any $0<\varepsilon<\sigma$ ($\sigma$ is as Lemma~\ref{LSbudengshipaowupaowu}), there is $N\in \mathbb{N}$ such that when $n\ge N$,
	\begin{equation}\label{eqwen16.8}
		\left\|u(t_{n},\cdot)-\psi \right\|_{H_{D}^{4}}<\frac{\varepsilon}{2},\quad \frac{1}{\theta}(E(u(t_{n}))-E(\psi))^{\theta}<\frac{\varepsilon}{2}.
	\end{equation}
	For $n\ge N$, let
	\begin{equation*}
		\bar{t}_{n}=\sup\left\{t\ge t_{n}:\left\|u(s,\cdot)-\psi \right\|_{H_{D}^{4}}<\sigma, \ \forall s\in[t_{n},t]\right\}.
	\end{equation*}
	We claim that $\bar{t}_{n}=+\infty$. Otherwise, if $\bar{t}_{n}<+\infty$, when $t\in [t_{n},\bar{t}_{n}]$, integrating \eqref{eqwen16.9} with respect to $t$ yields
	\begin{equation}\label{eqwen17.1}
		\int_{t_{n}}^{\bar{t}_{n}}\left\|u_{t} \right\|_{L^{2}}d\tau \le \frac{1}{\theta}(E(u(t_{n}))-E(\psi))^{\theta}.
	\end{equation}
	Then by \eqref{eqwen16.8} and \eqref{eqwen17.1}, we obtain
	\begin{align*}
		\left\|u(\bar{t}_{n})-\psi \right\|_{L^{2}}&\le \int_{t_{n}}^{\bar{t}_{n}}\left\|u_{t} \right\|_{L^{2}}d\tau+\left\|u(t_{n})-\psi \right\|_{L^{2}}\\
		&\le \frac{1}{\theta}(E(u(t_{n}))-E(\psi))^{\theta}+\left\|u(t_{n})-\psi \right\|_{L^{2}}\\
		&<\varepsilon.
	\end{align*}
	This indicates that when $n\to +\infty$,
	\begin{equation*}
		u(\bar{t}_{n})\to \psi\,\,\textup{in }L^{2}(\Omega).
	\end{equation*}
	By the relative compactness of the orbi in $X(\kappa)$, there is a subsequence of $u(\bar{t}_{n})$, still denoted by itself, such that
	\begin{equation*}
		u(\bar{t}_{n})\to \psi\,\,\textup{in }X(\kappa).
	\end{equation*}
	Thus, there is $N'\ge N$ such that when $n\ge N'$,
	\begin{equation*}
		\left\|u(\bar{t}_{n},\cdot)-\psi \right\|_{H_{D}^{4}}<\frac{\varepsilon}{2}<\frac{\sigma}{2}.
	\end{equation*}
	which contradicts the definition of $\bar{t}_{n}$. Thus, there is $N_{0}\ge N$ such that $\bar{t}_{n_{0}}=\infty$. Since $\frac{d}{dt}E\le 0$, we have
	\begin{equation}\label{eqwen17.5}
		E(u(t))\le E(u_{0})=\frac{1}{2}\left\|u_{0} \right\|_{H_{D}^{2}}^{2}-\int_{\Omega}\frac{\lambda}{1+u_{0}}<\infty.
	\end{equation}
	Combining \eqref{eqwen17.1} and \eqref{eqwen17.5}, we obtain that
	\begin{equation}\label{eqwen17.6}
		\int_{t_{n_{0}}}^{+\infty}\left\|u_{t} \right\|_{L^{2}}d\tau\le \frac{1}{\theta}(E(u(t_{n_{0}}))-E(\psi))^{\theta}<\infty.
	\end{equation}
	Then for any $t\ge t_{n_{n}}$,
	\begin{equation}\label{eqwen17.2}
		\left\|u(t)-\psi \right\|_{L^{2}}\le \int_{t}^{+\infty}\left\|u_{t} \right\|_{L^{2}}d\tau.
	\end{equation}
	Therefore, by \eqref{eqwen17.6} and \eqref{eqwen17.2}, we obtain the convergence of $u$ in $L^{2}(\Omega)$. Furthermore, by the relative compactness of the orbit, we obtain the desired conclusion \eqref{eq1.10}. Theorem~\ref{shoulianxingpaowu} is proved.
\end{proof}

Next we prove the convergence rate established in \eqref{eq1.11}. Before proving \eqref{eq1.11}, we need to prove the convergence rate in $L^{2}(\Omega)$, as stated in the following theorem.
\begin{lemma}\label{L2shoulian}
	There are constants $C_{1}>0$, $\gamma_{1}>0$ and $T_{1}>0$, such that when $t\ge T_{0}$,
	\begin{equation}\label{L2shouliansulv}
		\left\|u(t,\cdot)-\psi \right\|_{L^{2}}\le C_{1}(1+t)^{-\gamma_{1}}.
	\end{equation}
\end{lemma}
\begin{proof}
	The idea of proving  Corollary~6.3.3 in \cite{ref3} can be easily adopted to prove the inequality~\ref{L2shouliansulv}
	in our case, and hence we omit the details of the proof.
\end{proof}

Based on Lemma~\ref{L2shoulian}, we are now ready to obtain the convergence rate in $H_{D}^{4}(\Omega)$.
\begin{theorem}\label{shoulianH4}
	there are constants $C_{2}>0$, $\gamma_{2}>0$ and $T_{2}>0$, such that when $t\ge T_{0}$,
	\begin{equation}\label{H4shoulian}
		\left\|u(t,\cdot)-\psi \right\|_{H_{D}^{4}}\le C_{2}(1+t)^{-\gamma_{2}}.
	\end{equation}
\end{theorem}
\begin{proof}
	Let $w=u-\psi$, then $w$ satisfies
	\begin{equation}\label{eq3.16.1}
		\left.\left\{\begin{array}{ll} w_{t}+B\Delta^2w-T\Delta w=Fw,\quad &t>0, \ x\in\Omega ,\\
			w=\partial_\nu w=0,& x\in\partial\Omega ,\\
			w(0,\cdot)=u_{0}-\psi,& x\in \Omega,\end{array}\right.\right.
	\end{equation}
	where
	\begin{equation*}
		F=\frac{\lambda(2+u+\psi)}{(1+u)^{2}(1+\psi)^{2}}.
	\end{equation*}
	Mutiplying the equation in \eqref{eq3.16.1} by $w$, then integrating over $\Omega$ yields
	\begin{equation}\label{eqnew3.16.6}
		\frac{1}{2}\frac{d}{dt}\left\|w \right\|_{L^{2}}^{2}+\left\|w \right\|_{H_{D}^{2}}^{2}=\int_{\Omega}Fw^{2}dx.
	\end{equation}
	Recall that in Section~\ref{paowuLSbudengshi}, $\left\|\psi \right\|_{L^{\infty}}<1-\frac{\kappa}{2}$, then we get
	\begin{equation}\label{eq3.16.5}
		|F|\le \frac{\lambda(2+1-\kappa+1-\kappa/2)}{(\kappa)^{2}(\kappa/2)^{2}}:=\frac{c_{1}}{2}.
	\end{equation}
	Combining \eqref{eqnew3.16.6} with \eqref{eq3.16.5}, we obtain
	\begin{equation}\label{eqnew3.16.8}
		\frac{d}{dt}\left\|w \right\|_{L^{2}}^{2}+2\left\|w \right\|_{H_{D}^{2}}^{2}\le c_{1}\left\|w \right\|_{L^{2}}^{2}.
	\end{equation}
	Multiplying the equation in \eqref{eq3.16.1} by $w_{t}$, then integrating over $\Omega$ to get
	\begin{equation}\label{eqnew3.16.9}
		\frac{1}{2}\frac{d}{dt}\left\|w \right\|_{H_{D}^{2}}^{2}+\left\|w_{t} \right\|_{L^{2}}^{2}=\int_{\Omega}Fww_{t}dx.
	\end{equation}
	Note that
	\begin{align}
		\int_{\Omega}Fww_{t}dx&\le\frac{c_{1}}{2}\int_{\Omega}|w||w_{t}|dx\notag\\
		&\le \frac{c_{1}}{2}\int_{\Omega}\left(\frac{1}{c_{1}}|w_{t}|^{2}+\frac{c_{1}}{4}|w|^{2}\right)dx\notag\\
		&=\frac{1}{2}\left\|w_{t} \right\|_{L^{2}}^{2}+\frac{c_{1}^{2}}{8}\left\|w \right\|_{L^{2}}^{2}.\label{eqnew3.17.1}
	\end{align}
	Then by \eqref{eqnew3.16.9} and \eqref{eqnew3.17.1}, we obtain
	\begin{equation}\label{eqnew3.17.2}
		\frac{d}{dt}\left\|w \right\|_{H_{D}^{2}}^{2}+\left\|w_{t} \right\|_{L^{2}}^{2}\le \frac{c_{1}^{2}}{4}\left\|w \right\|_{L^{2}}^{2}.
	\end{equation}
	We use density argument. For convenience, we assume that $u_{0}\in D(A^{2})\cap X(\kappa)$, by the proof of Lemma~\ref{eqnewlemma35.1}, the global solution $u$ satisfies
	\begin{equation*}
		u\in C^{1}([0,+\infty),H^{4}(\Omega))\cap C^{2}([0,+\infty),L^{2}(\Omega)).
	\end{equation*}
	 Differentiate the equation in \eqref{eq3.16.1} with respect to $t$, multiply the resultant by $w_{t}$, and integrate over $\Omega$, we get
	\begin{align*}
		\frac{1}{2}\frac{d}{dt}\left\|w_{t} \right\|_{L^{2}}^{2}+\left\|w_{t} \right\|_{H_{D}^{2}}^{2}&=\int_{\Omega}\frac{2\lambda}{(1+u)^{3}}w_{t}^{2}dx\\
		&\le 2\lambda\kappa^{-3}\left\|w_{t} \right\|_{L^{2}}^{2}.
	\end{align*}
	This means
	\begin{equation}\label{eq3.17.5}
		\frac{d}{dt}\left\|w_{t} \right\|_{L^{2}}^{2}\le c_{2}\left\|w_{t} \right\|_{L^{2}}^{2},
	\end{equation}
	where $c_{2}=4\lambda \kappa^{-3}$. we can choose a constant $c_{0}\in (0,+\infty)$ such that $c_{0}\left\|w_{t} \right\|_{L^{2}}^{2}>c_{2}\left\|w_{t} \right\|_{L^{2}}^{2}$. Multiplying Inequality \eqref{eqnew3.17.2} by $c_{0}$ yields
	\begin{equation}\label{eq7.12}
		\frac{d}{dt}\left(c_{0}\left\|w \right\|_{H_{D}^{2}}^{2}\right)+c_{0}\left\|w_{t} \right\|_{L^{2}}^{2}\le \frac{c_{0}c_{3}}{4}\left\|w \right\|_{L^{2}}^{2}:=c_{4}\left\|w \right\|_{L^{2}}^{2}.
	\end{equation}
	将 \eqref{eqnew3.16.8}, \eqref{eq3.17.5} 与 \eqref{eq7.12} 相加, 我们得到
	\begin{align}
		&\frac{d}{dt}\left(\left\|w \right\|_{L^{2}}^{2}+\left\|w_{t} \right\|_{L^{2}}^{2}+c_{0}\left\|w \right\|_{H_{D}^{2}}^{2}\right)+c_{0}\left\|w_{t} \right\|_{L^{2}}^{2}+2\left\|w \right\|_{H_{D}^{2}}^{2}\notag\\
		&\le (c_{2}+c_{5})\left\|w \right\|_{L^{2}}^{2}+c_{3}\left\|w_{t} \right\|_{L^{2}}^{2}.\label{eq7.13}
	\end{align}
	Since $c_{0}>c_{2}$, adding $\left\|w \right\|_{L^{2}}^{2}$ on both sides of \eqref{eq7.13} to get
	\begin{align}
		&\frac{d}{dt}\left(\left\|w \right\|_{L^{2}}^{2}+\left\|w_{t} \right\|_{L^{2}}^{2}+c_{0}\left\|w \right\|_{H_{D}^{2}}^{2}\right)+\left\|w \right\|_{L^{2}}^{2}+(c_{0}-c_{3})\left\|w_{t} \right\|_{L^{2}}^{2}+2\left\|w \right\|_{H_{D}^{2}}^{2}\notag\\&\le (c_{2}+c_{5}+1)\left\|w \right\|_{L^{2}}^{2}.\label{eq7.14}
	\end{align}
	Let $p(t)=\left\|w \right\|_{L^{2}}^{2}+\left\|w_{t} \right\|_{L^{2}}^{2}+c_{0}\left\|w \right\|_{H_{D}^{2}}^{2}$. Then there is a constant $\alpha>0$ such that
	\begin{equation}\label{eq7.15}
		\alpha \left(\left\|w \right\|_{L^{2}}^{2}+\left\|w_{t} \right\|_{L^{2}}^{2}+c_{0}\left\|w \right\|_{H_{D}^{2}}^{2}\right)\le \left\|w \right\|_{L^{2}}^{2}+(c_{0}-c_{2})\left\|w_{t} \right\|_{L^{2}}^{2}+2\left\|w \right\|_{H_{D}^{2}}^{2}.
	\end{equation}
	By \eqref{L2shouliansulv}, \eqref{eq7.14} and \eqref{eq7.15}, we obtain that when $t\ge T_{1}$,
	\begin{equation*}
		\frac{d}{dt}p(t)+\alpha p(t)\le c_{5}(1+t)^{-2\gamma_{1}},
	\end{equation*}
	where $\gamma_{1}$ is as Lemma~\ref{L2shoulian}. A direct calculation yields that when $t\ge 3T_{1}$,
	\begin{align*}
		e^{\alpha t}p(t)-e^{\alpha T_{1}}p(T_{1})&\le\int_{T_{1}}^{t}c_{5}e^{\alpha \tau}(1+\tau)^{-2\gamma_{1}}d\tau\\
		&\le c_{5}\int_{T_{1}}^{\frac{t}{2}}e^{\alpha \tau}(1+\tau)^{-2\gamma_{1}}d\tau+c_{5}\int_{\frac{t}{2}}^{t}e^{\alpha \tau}(1+\tau)^{-2\gamma_{1}}d\tau\\
		&\le c_{5}(1+T_{1})^{-2\gamma_{1}}(\frac{t}{2}-T_{1})e^{\frac{\alpha t}{2}}+c_{5}(1+\frac{t}{2})^{-2\gamma_{1}}\int_{\frac{t}{2}}^{t}e^{\alpha \tau}d\tau\\
		&= c_{5}(1+T_{1})^{-2\gamma_{1}}(\frac{t}{2}-T_{1})e^{\frac{\alpha t}{2}}+\frac{c_{5}}{\alpha}(1+\frac{t}{2})^{-2\gamma_{1}}(e^{\alpha t}-e^{\alpha \frac{t}{2}})\\
		&\le c_{5}(1+T_{1})^{-2\gamma_{1}}(\frac{t}{2}-T_{1})e^{\frac{\alpha t}{2}}+\frac{c_{5}}{\alpha}(1+\frac{t}{2})^{-2\gamma_{1}}e^{\alpha t}.
	\end{align*}
	This indicates
	\begin{equation*}
		p(t)\le c_{6}e^{-\alpha t}+c_{7}e^{-\frac{\alpha t}{2}}t+c_{8}(1+\frac{t}{2})^{-2\gamma_{1}},\quad t\ge 3T_{1}.
	\end{equation*}
	Note that
	\begin{gather*}
		\lim_{t \to +\infty}\frac{e^{-\alpha t}}{(1+t)^{-2\gamma_{1}}}=0,\\
		\lim_{t \to +\infty}\frac{e^{-\frac{\alpha}{2} t}}{(1+t)^{-2\gamma_{1}}}=0,\\
		0<\lim_{t \to +\infty}\frac{(1+\frac{t}{2})^{-2\gamma_{1}}}{(1+t)^{-2\gamma_{1}}}<\infty,
	\end{gather*}
	Thus, there is $c_{9}>0$ such that
	\begin{equation*}
		p(t)\le c_{9}^{2}(1+t)^{-2\gamma_{1}},\quad t\ge 3T_{1}.
	\end{equation*}
	Therefore, $\left\|w_{t} \right\|_{L^{2}}\le c_{9}(1+t)^{-\gamma_{1}}$. For problem \eqref{eq3.16.1}, by the priori estimate of ellptic equation, there are constants $c_{10},c_{11}>0$ independent of $w$ such that when $t\ge 3T_{1}$,
	\begin{align}
		\left\|w \right\|_{H_{D}^{4}}&\le c_{10}\left(\left\|w_{t} \right\|_{L^{2}}+\left\|F w \right\|_{L^{2}}\right)+c_{11}\left\|w \right\|_{L^{2}}\notag\\
		&\le c_{10}\left(\left\|w_{t} \right\|_{L^{2}}+\frac{c_{1}}{2}\left\|w \right\|_{L^{2}}\right)+c_{11}\left\|w \right\|_{L^{2}}\notag\\
		&\le c_{12}(1+t)^{-\gamma_{1}}.\label{eq7.77}
	\end{align}
	Since $D(A^{2})$ is dense in $D(A)$, there is a sequence $u_{0n}\in D(A^{2})$ such that $u_{0n}\to u_{0}$ in $D(A)=H_{D}^{4}(\Omega)$. Similar to the proof of Lemma~\ref{eqnewlemma35.1}, we obtain that when $u_{0}\in X(\kappa)$, \eqref{H4shoulian} still holds. Theorem~\ref{theorem1.2} is completely proved.
\end{proof}


\section{Hyperbolic Problem}\label{dbjtjxdsqfc}
	In this section, we will prove Theorem~\ref{theorem1.3}. First we show that \eqref{neweq1.2} defines a gradient system, then we prove the corresponding Lojasiewicz-Simon inequality. Based on the two results, we will prove the global solution to \eqref{neweq1.2} must converge to a stationary solution. In addition, we establish the convergence rate.
	
		\subsection{Energy Dissipation Identity}\label{nlhds}
	In this subsection, we will prove that \eqref{neweq1.2} defines a gradient system. First, we establish the energy dissipation identity, then we show that the energy is a Lyapunov function. Finally we prove the orbit starting from $(u_{0},u_{1})$ is relatively compact in $Z(\kappa)$. By the definition~\ref{predef2.8}, we prove the existence of the gradient system.
	
	Recall that $-A=-(B\Delta^{2}-T\Delta)$,
	\begin{equation*}
		\mathbb{A}:=\begin{pmatrix}
			0 & -1\\
			A & 1
		\end{pmatrix}.
	\end{equation*}
	By \cite[Section~3.1]{ref2}, $\mathbb{A}$ generates a $C_{0}$-semigroup $T(t)$ on a Hilbert space $\mathbb{H}:=H_{D}^{2}(\Omega)\times L^{2}(\Omega)$ with the domain $D(\mathbb{A}):=H_{D}^{4}(\Omega)\times H_{D}^{2}(\Omega)$. Recall that
	\begin{gather*}
		\mathbf{u}=(u,u_{t})^{T},\quad \mathbf{u}_{0}=(u_{0},u_{1})^{T},\quad V=H_{D}^{2}(\Omega),
	\end{gather*}
	\begin{equation*}
		g(\mathbf{u})=(0,f(u))^{T},\quad \text{其中} \ f(u):=-1/(1+u)^{2}.
	\end{equation*}
	Define the norm in $\mathbb{H}$:
	\begin{equation*}
		\left\|(v_{1},v_{2}) \right\|_{\mathbb{H}}=\left(\left\|v_{1} \right\|_{H_{D}^{2}}^{2}+\left\|v_{2} \right\|_{L^{2}}^{2} \right)^{\frac{1}{2}}.
	\end{equation*}
	Then under the assumptions of Theorem~\ref{theorem1.3}, the global solution $\mathbf{u}$ satisfies
	\begin{gather*}
		\left\|\mathbf{u}(t) \right\|_{\mathbb{H}}^{2}\le \frac{(1-\kappa)^{2}}{C_{0}^{2}},\quad \forall t\ge 0.
	\end{gather*}
	Since $H^{2}(\Omega)\hookrightarrow L^{\infty}(\Omega)$, for any $t\ge 0$,
	\begin{gather*}
		\left\|u_{t} \right\|_{L^{2}}\le \left\|\mathbf{u}(t) \right\|_{\mathbb{H}}\le \frac{1-\kappa}{C_{0}},\\
		\left\|u(t) \right\|_{L^{\infty}}\le C_{0}\left\|u(t) \right\|_{H_{D}^{2}}\le C_{0}\left\|\mathbf{u}(t) \right\|_{\mathbb{H}}\le 1-\kappa.
	\end{gather*}
	Then by Theorem~\ref{newtheorem2.2}, solution $\mathbf{u}$ to \eqref{neweq1.2} is the mild solution. Then by the definition of mild soution, $\mathbf{u}$ satisfies the following integral equation:
	\begin{equation*}
		u(t)=T(t)\mathbf{u}_{0}+\int_{0}^{t}T(t-\tau)g(\mathbf{u}(\tau))d\tau.
	\end{equation*}
	To get energy functional, it suffices to prove the following lemma.
	\begin{lemma}\label{newnewlemma 11111}
		Under the assumption of Theorem~\ref{theorem1.3}, the following identity holds for all $t>0$:
		\begin{equation}\label{eqnewnewnewnew4.1}
			\frac{d}{dt}\int_{\Omega}\left(\frac{1}{2}\left|u_{t} \right|^{2}+\frac{1}{2}B\left|\Delta u \right|^{2}+\frac{1}{2}T\left|\nabla u \right|^{2}-\frac{\lambda}{1+u}\right)dx+\left\|u_{t} \right\|_{L^{2}}^{2}=0.
		\end{equation}
				In addition, $Q(t)$ defined by
		\begin{equation*}
			\mathbf{u}(t)=Q(t)\mathbf{u}_{0}
		\end{equation*}
	is a nonlinear $C_{0}$-semigroup in $Z(\kappa)$.
	\end{lemma}
	\begin{proof}
		We still use the density argument. Since $\mathbb{A}$ generates a $C_{0}$-semigroup on $\mathbb{H}$, by Hille-Yosida Theorem (\cite[Theorem~1.3.1]{ref15}),
		\begin{equation*}
			\overline{D(\mathbb{A})}=\mathbb{H}.
		\end{equation*}
		Since $Z(\kappa)\subset \mathbb{H}$, when $\mathbf{u}_{0}\in D(\mathbb{A}) \cap Z(\kappa)$ 时, there is a sequence $\mathbf{u}_{0}^{(n)}=(u_{0}^{(n)},u_{1}^{(n)})\in D(\mathbb{A})$, such that $\mathbf{u}_{0}^{(n)} \to \mathbf{u}_{0}$ in $\mathbb{H}$. By the definition of strong convergence and the commutativity of limits and norms, we have
		\begin{equation*}
			\lim_{n\to +\infty}\big\| \mathbf{u}_{0}^{(n)}\big\|_{\mathbb{H}}^{2}=\big\|\lim_{n\to +\infty}\mathbf{u}_{0}^{(n)} \big\|_{\mathbb{H}}^{2}=\big\|\mathbf{u}_{0} \big\|_{\mathbb{H}}^{2}.
		\end{equation*}
		By the definition of the limit of a sequence, for fixed $\varepsilon_{0}=\frac{\left(1-\kappa/2\right)^{2}}{C^{2}}-\frac{\left(1-\kappa\right)^{2}}{C^{2}}>0$, there is $N_{0}\in \mathbb{N}$ such that when $n\ge N_{0}$, $$		\left| \big\| \mathbf{u}_{0}^{(n)}\big\|_{\mathbb{H}}^{2}-\big\|\mathbf{u}_{0} \big\|_{\mathbb{H}}^{2}\right|\le \varepsilon_{0}.$$
		By triangle inequality,
		\begin{align*}
			\big\|\mathbf{u}_{0}^{(n)}\big\|_{\mathbb{H}}^{2}&\le 	\big\|\mathbf{u}_{0}\big\|_{\mathbb{H}}^{2}+\varepsilon_{0}\\
			&\le \frac{\left(1-\kappa\right)^{2}}{C_{0}^{2}}+\frac{\left(1-\kappa/2\right)^{2}}{C_{0}^{2}}-\frac{\left(1-\kappa\right)^{2}}{C_{0}^{2}}=\frac{\left(1-\kappa/2\right)^{2}}{C_{0}^{2}}.
		\end{align*}
		Since $H^{2}(\Omega)\hookrightarrow L^{\infty}(\Omega)$, we have
		\begin{equation*}
			\big\|u_{0}^{(n)} \big\|_{L^{\infty}}\le C_{0} \big\|u_{0}^{(n)} \big\|_{H_{D}^{2}}\le 	C_{0}\big\|\mathbf{u}_{0}^{(n)}\big\|_{\mathbb{H}}\le 1-\frac{\kappa}{2}.
		\end{equation*}
		This menas $u_{0}^{(n)}+1$ has a positive lower bound of $\frac{\kappa}{2}$. When $n \ge N_{0}$, consider the following problem
		\begin{equation}\label{neweq5.1}
			\left.\left\{\begin{array}{ll} v_{tt}+v_{t}+B\Delta^2v-T\Delta v=-\frac{\lambda}{(1+v)^2},\quad &t>0, \ x\in\Omega ,\\
				v=\partial_\nu v=0,& x\in\partial\Omega ,\\
				v(0,\cdot)=u_{0}^{(n)}, \ \ v_{t}(0,\cdot)=u_{1}^{(n)},& x\in \Omega,\end{array}\right.\right.
		\end{equation}
		Then by Theorem~\ref{newtheorem18.2}, there are $\tau_{m}^{(n)}>0$ and a unique solution to $u^{(n)}$ satisfying \eqref{neweq5.1} and
		\begin{equation*}
			u^{(n)}\in C([0,\tau_{m}^{(n)}),H^{4}(\Omega))\cap C^{1}([0,\tau_{m}^{(n)}),H^{2}(\Omega))\cap C^{2}([0,\tau_{m}^{(n)}),L^{2}(\Omega)).
		\end{equation*}
		For convenience, let
		$$\mathbf{u}^{(n)}=(u^{(n)},u_{t}^{(n)}).$$
		Then we show that when $n\ge N_{0}$, $\tau_{m}^{(n)}=\infty$. Since $	\big\|u_{0}^{(n)} \big\|_{L^{\infty}}\le 1-\frac{\kappa}{2}$, by the regularity of $u^{(n)}$,
		\begin{equation*}
			T_{0}^{(n)}:=\sup\left\{\tau\in (0,\tau_{m}^{(n)}) \ : \ u^{(n)}(t)\ge -1+\frac{\kappa}{4}, \ t\in[0,\tau)\right\}>0,
		\end{equation*}
		In addition, for $t\in [0,T_{0}^{(n)})$,  $1+u^{(n)}\ge \frac{\kappa}{4}$. Multiply the equation in \eqref{neweq5.1} by $u_{t}^{(n)}$, then integrate with respect to $x$ and $t$, we get for $\forall t \in(0,\tau_{m}^{(n)})$,
		\begin{align}
			&\frac{1}{2}\left(\big\|u_{t}^{(n)} \big\|_{L^{2}}^{2}+\big\|u^{(n)} \big\|_{H_{D}^{2}}^{2}-\int_{\Omega}\frac{2\lambda}{1+u^{(n)}}dx\right)+\int_{0}^{t}\big\|u_{t}^{(n)} \big\|_{L^{2}}^{2}\notag\\
			&=\frac{1}{2}\left(\big\|u_{1}^{(n)} \big\|_{L^{2}}^{2}+\big\|u_{0}^{(n)} \big\|_{H_{D}^{2}}^{2}-\int_{\Omega}\frac{2\lambda}{1+u_{0}^{(n)}}dx\right).\label{eqnewnew4.1}
		\end{align}
		By \eqref{eqnewnew4.1} and the definition of the norm in $\mathbb{H}$, we have
		\begin{align}
			\big\|\mathbf{u}^{(n)} \big\|_{\mathbb{H}}^{2}&\le 	\big\|\mathbf{u}_{0}^{(n)} \big\|_{\mathbb{H}}^{2}+2\lambda\int_{\Omega}\left(\frac{1}{1+u^{(n)}}-\frac{1}{1+u_{0}^{(n)}} \right)dx \notag\\
			&\le \frac{\left(1-\kappa/2\right)^{2}}{C_{0}^{2}}+8\lambda \kappa^{-1}|\Omega|. \label{neweq10.2}
		\end{align}
		A direct calculation yields that the inequality
		\begin{equation}\label{neweq10.3}
			\frac{\left(1-\kappa/2\right)^{2}}{C_{0}^{2}}+8\lambda \kappa^{-1}|\Omega| \le \frac{\left(1-\kappa/4\right)^{2}}{C_{0}^{2}}
		\end{equation}
		holds if and only if $$\lambda \le \frac{\kappa^{2}(8-3\kappa)}{128C_{0}^{2}|\Omega|}.$$
		This is consistent with the assumption in Theorem~\ref{theorem1.3}, at this time we get $T_{0}^{(n)}\ge \tau_{m}^{(n)}$. On the other hands, by the definition of $T_{0}^{(n)}$, $T_{0}^{(n)}\le \tau_{m}^{(n)}$. Thus $T_{0}^{(n)}=\tau_{m}^{(n)}$. By \eqref{neweq10.2} and \eqref{neweq10.3},
		\begin{equation*}
			\big\|u^{(n)} \big\|_{L^{\infty}}\le C_{0} \big\|u^{(n)} \big\|_{H_{D}^{2}}\le 	C_{0}\big\|\mathbf{u}^{(n)}\big\|_{\mathbb{H}}\le 1-\frac{\kappa}{4}.
		\end{equation*}
		Therefore,
		$$\liminf_{t\to \tau_{m}^{(n)}}\left(\min_{\Omega}u^{(n)}(t)\right)\ge -1+\frac{\kappa}{4}>-1.$$
		By Theorem~\ref{newtheorem2.2}(\romannumeral2), $\tau_{m}^{(n)}=\infty$. The solution to \eqref{neweq5.1} global exists and satisfying $u\ge -1+\frac{\kappa}{4}$.
		
		Let $$r=u^{(n)}-u^{(m)},\quad n,m\ge N_{0}.$$
		and
		$$\mathbf{r}_{0}=(u_{0}^{(n)}-u_{0}^{(m)},u_{1}^{(n)}-u_{1}^{(m)}),\quad \mathbf{r}=(r,r_{t})=(u^{(n)}-u^{(m)},u_{t}^{(n)}-u_{t}^{(m)}).$$
		Then $r$ satisfies
		\begin{equation}\label{eqnewnew4.2}
			\left.\left\{\begin{array}{ll} r_{tt}+r_{t}+B\Delta^2 r-T\Delta r=Gr,\quad &t>0, \ x\in\Omega ,\\
				r=\partial_\nu r=0,& x\in\partial\Omega ,\\
				r(0,\cdot)=u_{0}^{(n)}-u_{0}^{(m)}, \ \ r_{t}(0,\cdot)=u_{1}^{(n)}-u_{1}^{(m)},& x\in \Omega,\end{array}\right.\right.
		\end{equation}
		where $$G=\frac{\lambda(2+u^{(n)}+u^{(m)})}{(1+u^{(n)})^{2}(1+u^{(m)})^{2}}.$$
		Multiplying the equation in \eqref{eqnewnew4.2} by $r_{t}$, integrating over $\Omega$ to get
		\begin{align}\label{neweq10.4}
			\frac{1}{2}\frac{d}{dt}\left(\left\|r_{t} \right\|_{L^{2}}^{2}+\left\|r \right\|_{H_{D}^{2}}^{2}\right)+\left\|r_{t} \right\|_{L^{2}}^{2}=\int_{\Omega} Grr_{t}dx, \ \forall t>0.
		\end{align}
		We estimate $G$:
		\begin{equation*}
			\left| G\right|\le \lambda
			\frac{\left(2+\big\|u^{(n)} \big\|_{L^{\infty}}+\big\|u^{(m)} \big\|_{L^{\infty}}\right)}{\left(\kappa/4\right)^{4}}\le \frac{128\lambda(8-\kappa)}{\kappa^{4}}:=C_{1}.
		\end{equation*}
		By Young's inequality,
		\begin{equation}\label{neweq10.7}
			\left|\int_{\Omega}Grr_{t}dx\right|\le C_{1}\left(\frac{C_{1}}{2}\left\|r \right\|_{L^{2}}^{2}+\frac{1}{2C_{1}}\left\|r_{t} \right\|_{L^{2}}^{2}\right).
		\end{equation}
		Combining \eqref{neweq10.4} with \eqref{neweq10.7}, we obtain
		\begin{equation*}
			\frac{d}{dt}\left( \left\|r_{t} \right\|_{L^{2}}^{2}+\left\|r \right\|_{H_{D}^{2}}^{2}\right)\le C_{1}^{2}\left\|r \right\|_{L^{2}}^{2}\le C_{1}^{2}\left( \left\|r_{t} \right\|_{L^{2}}^{2}+\left\|r \right\|_{H_{D}^{2}}^{2}\right).
		\end{equation*}
		Solving the differential inequality, we obtain
		\begin{equation}\label{neweq10.11}
			\left\|\mathbf{r}(t) \right\|_{\mathbb{H}}^{2}\le e^{C_{1}^{2}t} \left\|\mathbf{r}_{0} \right\|_{\mathbb{H}}^{2}.
		\end{equation}
		Note that for any $T_{1}>0$, when $n,m\to +\infty$ ,$$e^{C_{1}^{2}T_{1}^{2}}\left\|\mathbf{r}_{0} \right\|_{\mathbb{H}}^{2} \to 0.$$
		Thus $\mathbf{u}^{(n)}$ is a Cauchy sequence in $C([0,T_{1}],\mathbb{H})$. By the equation in \eqref{eqnewnew4.2}, we get that $u_{tt}^{(n)}$ is a Cauchy sequence in $C([0,T_{1}],V')$. Since $u^{(n)}(t)$ is a classical solution,
		\begin{equation}\label{eqnewnew4.2222}
			\mathbf{u}^{(n)}(t)=T(t)\mathbf{u}_{0}^{(n)}+\int_{0}^{t}T(t-\tau)g(\mathbf{u}^{(n)}(\tau))d\tau.
		\end{equation}
		By the definition of Banach space, we assume that $\mathbf{u}^{(n)}(t)\to \tilde{\mathbf{u}}(t)$ in $\mathbb{H}$. Since $T(t)$ is a $C_{0}$-semigroup, by the property of $C_{0}$-semigroup (\cite[Theorem~1.2.2]{ref15}), there are constants $\omega\ge 0$ and $M\ge 1$ such that
		\begin{equation*}
			\left\|T(t) \right\|_{\mathcal{L}(\mathbb{H})}\le Me^{\omega t},\quad 0\le t<\infty.
		\end{equation*}
		Since $1+u^{(n)}\ge \frac{\kappa}{4}$,
		\begin{equation*}
			\left\|g(\mathbf{u}^{(n)}(\tau)) \right\|_{\mathbb{H}}=\left\|f(u^{(n)}) \right\|_{L^{2}}\le 16\lambda \kappa^{-2}|\Omega|^{\frac{1}{2}}.
		\end{equation*}
		Then
		\begin{equation*}
			\left\| T(t-\tau)g(\mathbf{u}^{(n)}(\tau))\right\|_{\mathbb{H}}\le 16\lambda \kappa^{-2}|\Omega|^{\frac{1}{2}}Me^{\omega t}.
		\end{equation*}
		By the definition of $C_{0}$-semigroup, $T(t)$ is a bounded linear operator in $\mathbb{H}$. Then we take limits on both sides of \eqref{eqnewnew4.2222} in $[0,T_{1}]$, we obtain that
		\begin{align*}
			\lim_{n\to +\infty}\mathbf{u}^{(n)}(t)=\tilde{\mathbf{u}}(t).
		\end{align*}
		On the other hands, since $T(t)$ is a bounded linear operator and the Lebesgue convergence theorem, taking limits in both sides of \eqref{eqnewnew4.2222} yields
		\begin{align*}
			&\lim_{n\to +\infty}\left\{T(t)\mathbf{u}_{0}^{(n)}+\int_{0}^{t}T(t-\tau)g(\mathbf{u}^{(n)}(\tau))d\tau\right\}\\
			&=\lim_{n\to +\infty}T(t)\mathbf{u}_{0}^{(n)}+\lim_{n \to +\infty}\int_{0}^{t}T(t-\tau)g(\mathbf{u}^{(n)}(\tau))d\tau\\
			&=T(t)\lim_{n\to +\infty}\mathbf{u}_{0}^{(n)}+\int_{0}^{t}\lim_{n \to +\infty}T(t-\tau)g(\mathbf{u}^{(n)}(\tau))d\tau\\
			&=T(t)\mathbf{u}_{0}+\int_{0}^{t}T(t-\tau)\lim_{n \to +\infty}g(\mathbf{u}^{(n)}(\tau))d\tau\\
			&=T(t)\mathbf{u}_{0}+\int_{0}^{t}T(t-\tau)g(\tilde{\mathbf{u}}(\tau))d\tau.
		\end{align*}
		The detalis are as follows:
		\begin{align}
			\|g(\mathbf{u}^{(n)})-g(\tilde{\mathbf{u}})\|_{\mathbb{H}}&=\|f(u^{(n)})-f(\tilde{u})) \|_{L^{2}}\notag\\
			&=\left\|-\frac{\lambda}{(1+u^{(n)})^{2}}+\frac{\lambda}{(1+\tilde{u})^{2}} \right\|_{L^{2}}\notag\\
			&=\lambda \left\|\frac{u^{(n)}-\tilde{u}}{(1+u^{(n)})^{2}(1+\tilde{u})^{2}} \right\|_{L^{2}}.\label{neweq10.9}
		\end{align}
	    Recall that
		\begin{align*}
			\lim_{n\to +\infty}\mathbf{u}^{(n)}(t)=\tilde{\mathbf{u}}(t).
		\end{align*}
	Using the definition of convergence in the product space, we get $u^{(n)}\to \tilde{u}$ in $H_{D}^{2}(\Omega)$. Then by the strong convergence in Banach space, fix $\varepsilon_{1}=\frac{1-\kappa/8}{C_{0}}-\frac{1-\kappa/4}{C_{0}}>0$, there is an integer $N_{1}\ge N_{0}$ such that when $n\ge N_{1}$,
		\begin{equation*}
			\big\|u^{(n)}-\tilde{u} \big\|_{H_{D}^{2}}\le \varepsilon_{1}.
		\end{equation*}
		By the triangle inequality, we have
		\begin{align*}
			\big\| \tilde{u}\big\|_{H_{D}^{2}}\le \big\|u^{(n)} \big\|_{H_{D}^{2}}+\varepsilon_{1}\le \frac{1-\kappa/4}{C_{0}}+ \frac{1-\kappa/8}{C_{0}}-\frac{1-\kappa/4}{C_{0}}=\frac{1-\kappa/8}{C_{0}}.
		\end{align*}
	Thus
		\begin{equation*}
			\left\|\tilde{u} \right\|_{L^{\infty}}\le C_{0}\left\|\tilde{u} \right\|_{H_{D}^{2}}\le 1-\frac{\kappa}{8}.
		\end{equation*}
		Since $H^{2}\hookrightarrow L^{2}(\Omega)$, $u^{(n)}\to \tilde{u}$ in $L^{2}(\Omega)$. When $n\to +\infty$, by \eqref{neweq10.9}, we obtain
		\begin{equation*}
			\big\|g(\mathbf{u}^{(n)})-g(\tilde{\mathbf{u}})\big\|_{\mathbb{H}}\le 32\lambda \kappa^{-2}\big\|u^{(n)}-\tilde{u} \big\|_{L^{2}}\to 0.
		\end{equation*}
		Now let us turn to our proof, taking limit on both sides of \eqref{eqnewnew4.2222} yields that $\tilde{\mathbf{u}}$ is a classical solution in any interval $[0,T_{1}]$. By the uniqueness of the solution to \eqref{neweq1.2}, we obtain that $\tilde{\mathbf{u}}=\mathbf{u}$. Due to the commutativity of norm and limit, taking limit on both sides of \eqref{eqnewnew4.1} yields \eqref{eqnewnewnewnew4.1}. Therefore, when $\mathbf{u}_{0}\in Z(\kappa)$, \eqref{eqnewnewnewnew4.1} holds for all $t>0$.
		
		By the uniqueness of the solution, we can deduce that $Q(t)$ is a $C_{0}$-semigroup.
		The lemma is proved
	\end{proof}
	
	By Lemma~\ref{newnewlemma 11111}, we define the following energy functional:
	\begin{equation}\label{eq4.10}
		E(\mathbf{u})=\int_{\Omega}\left(\frac{1}{2}\left|u_{t} \right|^{2}+\frac{1}{2}B\left|\Delta u \right|^{2}+\frac{1}{2}T\left|\nabla u \right|^{2}-\frac{\lambda}{1+u}\right)dx.
	\end{equation}
	This lays the foundation for proving that problem \eqref{neweq1.2} defines a gradient system in the next subsection.
	
		\subsection{Gradient System}\label{htdxt}
	Recall that in Lemma~\ref{newnewlemma 11111}, $Q(t):Z(\kappa)\to Z(\kappa)$ is a noninear $C_{0}$-semigroup. To study the property of $\omega$-limit set, it is necessary to prove that \eqref{neweq1.2} defines a gradient system. Motivated by \cite{ref1,ref5,ref9}, we will use the integral equation satisfied by $\mathbf{u}$ to obtain the relatively compactness of the orbit.
	\begin{lemma}\label{Lemma new100000}
		Under the assumptions of Theorem~\ref{theorem1.3}, $\left(Z(\kappa),Q(t),E\right)$ is a gradient system.
	\end{lemma}
	\begin{proof}
		It is easily to show that $E:Z(\kappa) \mapsto \mathbb{R}$ is a Lyapunov function. In addition, integrating \eqref{eqnewnewnewnew4.1} with respect to $t$ yields
		\begin{equation}\label{eqnew222}
			E(\mathbf{u}(t))+\int_{0}^{t}\left\|u_t\right\|_{L^2(\Omega)}^2d\tau=E(\mathbf{u}_{0}).
		\end{equation}
		This means if there is $t_{0}>0$ such that $E(\mathbf{u}(t_0))=E(Q(t_0)\mathbf{u}_0)=E(\mathbf{u}_{0})$, then for all $0\le t\le t_0$, $u_{t}=0$, i.e. $u(t)$ is a constant. When $0\le t \le t_{0}$,
		\begin{equation*}
			u(t)=Q(t)\mathbf{u}_{0}=Q(0)\mathbf{u}_{0}=\mathbf{u}_{0}.
		\end{equation*}
		Therefore $\mathbf{u}_{0}$ is a fixed point of $Q(t)$.
		
		Finally, we show that there is $t_{0}>0$ such that the orbit starting from $\mathbf{u}_{0}$
		\begin{equation*}
			\bigcup_{t\ge t_0}Q(t)\mathbf{u}_0
		\end{equation*}
		is relatively compact in $Z(\kappa)$. Due to the lack of regularity of $\mathbf{u}$, it is failed to use the priori estimate of ellptic equation to prove the relatively compactness of the orbit. Now we will use \cite[Proposition~3.2]{ref9}. Recall that $u$ satisfies
		\begin{equation*}
			u(t)=T(t)\mathbf{u}_{0}+\int_{0}^{t}T(t-\tau)g(\mathbf{u}(\tau))d\tau.
		\end{equation*}
		Consider the following problem
		\begin{equation}\label{eqnewnew4.111}
			\left.\left\{\begin{array}{ll}w_{tt}+w_{t}+B\Delta^2w-T\Delta w=0,\quad &t>0, \ x\in\Omega ,\\w=\partial_\nu u=0,&t>0, \  x\in\partial\Omega .\end{array}\right.\right.
		\end{equation}
		Let $I:V\rightarrow V'$ be the identity operator. Since $V\hookrightarrow L^{2}(\Omega)\cong {(L^{2}(\Omega))}'\hookrightarrow V'$, we get
		\begin{gather*}
			I\in \mathcal{L}(V,V'),\\
			\exists \alpha=1,\forall v\in V, \quad (I(v),v)_{V'\times V}=(v,v)_{L^{2}}\ge \alpha \left\|v \right\|_{L^{2}}^{2},\\
			\exists \alpha_{1}^{2}>0,\forall v\in V,\quad \left\|I(v) \right\|_{V'}^{2}\le \alpha_{1}^{2}\left\|I(v) \right\|_{L^{2}}^{2}=\alpha_{1}^{2}(I(v),v)_{V'\times V}.
		\end{gather*}
		where $\alpha_{1}$ is as Introduction. Let $$\mathbf{w}=(w,w_{t}),\quad \mathbf{w}_{0}=(w_{0},w_{1})=(w(0),w_{t}(0)).$$
		When $(w_{0},w_{1}) \in D(\mathbb{A})\cap Z(\kappa)$ 时, by the classical semigroup theory \cite[Theorem~2.2.2]{ref3}), problem \eqref{eqnewnew4.111} admits a unique solution $\mathbf{w}$ with regularity $$\mathbf{w}\in C([0,+\infty),D(\mathbb{A}))\cap C^{1}([0,+\infty),\mathbb{H}).$$
		Then it follows from \cite[Proposition~5.3.1]{ref5} that there are constants $\tilde{C}\ge 1$ and $\tilde{\gamma}>0$ such that
		\begin{equation*}
			\left\|w \right\|_{V}^{2}+\left\|w_{t} \right\|_{L^{2}}^{2}\le \tilde{C}e^{-\tilde{\gamma}_{0} t}\left(\left\|w_{0} \right\|_{V}^{2}+\left\|w_{1} \right\|_{L^{2}}^{2} \right),
		\end{equation*}
		Thus,
		\begin{equation}\label{eqnewnew 4.2}
			\left\|\mathbf{w} \right\|_{\mathbb{H}}^{2}\le \tilde{C}e^{-\tilde{\gamma}_{0} t}	\left\|\mathbf{w}_{0} \right\|_{\mathbb{H}}^{2}
		\end{equation}
		Recall that $\overline{D(\mathbb{A})}=\mathbb{H}$. When $\mathbf{w}_{0}\in \mathbb{H}$, there is a sequence $\mathbf{w}_{0}^{(n)}=(w_{0}^{(n)},w_{1}^{(n)})\in D(\mathbb{A})$ such that $\mathbf{w}_{0}^{(n)}\to \mathbf{w}_{0}$ in $\mathbb{H}$. Let $\mathbf{w}^{(n)}$ be the corresponding solution to the initial value $\mathbf{w}_{0}^{(n)}$, when $n$ is large enough, Let $z=w^{(n)}-w^{(m)}$ and write
		\begin{equation*}
			\mathbf{z}=(z,z_{t}),\quad \mathbf{z}_{0}=(z(0),z_{t}(0)).
		\end{equation*}
		Then $z$ satisfies
		\begin{equation}\label{neweq10.13}
			\left.\left\{\begin{array}{ll} z_{tt}+z_{t}+B\Delta^2 z-T\Delta z=0,\quad &t>0, \ x\in\Omega ,\\
				z=\partial_\nu z=0,& x\in\partial\Omega ,\\
				z(0,\cdot)=w_{0}^{(n)}-w_{0}^{(m)}, \ \ z_{t}(0,\cdot)=w_{1}^{(n)}-w_{1}^{(m)},& x\in \Omega,\end{array}\right.\right.
		\end{equation}
		Multiplying the equation in \eqref{neweq10.13} by $z_{t}$, then integrating over $\Omega$ to get for all $t\ge 0$,
		\begin{align}\label{neweq10.15}
			\frac{1}{2}\frac{d}{dt}\left(\left\|z_{t} \right\|_{L^{2}}^{2}+\left\|z \right\|_{H_{D}^{2}}^{2}\right)+\left\|z_{t} \right\|_{L^{2}}^{2}=0.
		\end{align}
		Integrating \eqref{neweq10.15} with respect to $t$ yields
		\begin{equation*}
			\left\|z_{t} \right\|_{L^{2}}^{2}+\left\|z \right\|_{H_{D}^{2}}^{2}\le \left\|z_{t}(0) \right\|_{L^{2}}^{2}+\left\|z(0) \right\|_{H_{D}^{2}}^{2},
		\end{equation*}
		Thus,
		\begin{equation}\label{neweq10.19}
			\left\|\mathbf{z} \right\|_{\mathbb{H}}^{2}\le 	\left\|\mathbf{z}_{0} \right\|_{\mathbb{H}}^{2}
		\end{equation}
		Since $\mathbf{w}_{0}^{(n)}\to \mathbf{w}_{0}$, when $n,m\to +\infty$, $\left\|\mathbf{z}_{0} \right\|_{\mathbb{H}}\to 0$, By \eqref{neweq10.19}, we obtain
		\begin{equation*}
			\left\|\mathbf{z} \right\|_{\mathbb{H}}\to 0.
		\end{equation*}
		Thus $\mathbf{w}^{(n)}$ is a Cauchy sequence in $\mathbb{H}$. We assume that $\mathbf{w}^{(n)} \to \tilde{\mathbf{w}}$.
		Using integral equation, we have
		\begin{equation*}
			\mathbf{w}^{(n)}=T(t)\mathbf{w}_{0}^{(n)}.
		\end{equation*}
		Since $T(t)$ is a bounded linear operator,
		\begin{equation}
			\tilde{\mathbf{w}}=\lim_{n \to +\infty}\mathbf{w}^{(n)}=\lim_{n \to +\infty}T(t)\mathbf{w}_{0}^{(n)}=T(t)\lim_{n\to +\infty}\mathbf{w}_{0}^{(n)}=T(t)\mathbf{w}_{0}.
		\end{equation}
		By the uniqueness of the solution to \eqref{neweq10.13}, $\mathbf{w}=\tilde{\mathbf{w}}$. Similar to \eqref{eqnewnew 4.2}, $\mathbf{w}^{(n)}$ satisfies
		\begin{equation}\label{neweq13.1}
			\left\|\mathbf{w}^{(n)} \right\|_{\mathbb{H}}^{2}\le \tilde{C}e^{-\tilde{\gamma} t}	\left\|\mathbf{w}_{0}^{(n)} \right\|_{\mathbb{H}}^{2}
		\end{equation}
		Taking limit on both sides of \eqref{neweq13.1} yields \eqref{eqnewnew 4.2}. Note that the solution to \eqref{neweq10.13} can be written as
		\begin{equation*}
			\mathbf{w}(t)=T(t)\mathbf{w}_{0}.
		\end{equation*}
		Then by \eqref{neweq13.1}, we obtain that
		$$\left\|T(t) \right\|_{\mathcal{L}(\mathbb{H})}\le \tilde{C}e^{-\tilde{\gamma} t}.$$
		Let $$T_{1}(t)=T(t),\quad T_{2}(t)=0,\quad \quad c(t)=\tilde{C}e^{-\tilde{\gamma} t}.$$
		Then $T(t)=T_{1}(t)+T_{2}(t)$. We can deduce the following results:
		\begin{itemize}
			\item[(\romannumeral1)] Under the assumptions of Theorem~\ref{theorem1.3}, the orbit $\bigcup_{t\ge0}Q(t)\mathbf{u}_{0}$ is bounded.
			\item[(\romannumeral2)] Recall that the solution $\mathbf{u}$ satisfies integral equation
			\begin{equation*}
				u(t)=T(t)\mathbf{u}_{0}+\int_{0}^{t}T(t-\tau)g(\mathbf{u}(\tau))d\tau,
			\end{equation*}
			where $T(t)$ is a $C_{0}$-semigroup.
			\item[(\romannumeral3)] $\left\|T_{1}(t) \right\|_{\mathcal{L}(\mathbb{H})}\le c(t)$, and $\lim_{t \to +\infty}c(t)=0$.
			\item[(\romannumeral4)] $T_{2}(t)=0$ is compact.
			\item[(\romannumeral5)] Since the embedding $H^{2}(\Omega)\hookrightarrow L^{2}(\Omega)$ is comapct. Thu  $g(u)=-\frac{\lambda}{(1+u)^{2}}:  X(\kappa)\hookrightarrow L^{2}(\Omega)$ is comapct, and the first element of $f(\mathbf{u})$ is $0$, , which is compact.  by Tychonoff theorem in topology, $f(\mathbf{u}):Z(\kappa)\to Z(\kappa)$ is compact.
		\end{itemize}
		Therefore, it follows from \cite[Proposition~3.2]{ref9} that the orbit
		\begin{equation*}
			\bigcup_{t\ge 0}Q(t)\mathbf{u}_{0}
		\end{equation*}
		is relatively compact in $Z(\kappa)$. Thus, by definition~\ref{predef2.8}, $(Z(\kappa),Q(t),E)$ is a gradient system.
	\end{proof}
	
	Since problem \eqref{neweq1.2} defines a gradient system, then by Theorem~\ref{prethm2.4}, $\omega$-limit set $\omega(u_{0},u_{1})$ is a connected comapct invariant set. Furthermore, there are $(\psi,\phi)$ satisfying
	\begin{equation}
		\mathbb{A}(\psi,\phi)^{T}=(\psi,\phi)^{T}
	\end{equation}
	and a sequence $\left\{t_{n}\right\}\to +\infty$, such that
	\begin{equation*}
		\lim_{t\to +\infty}\left\|(u(t_{n},\cdot)-\psi,u_{t}(t_{n},\cdot)-\phi) \right\|_{\mathbb{H}}\to 0\,\,\textup{in }Z(\kappa).
	\end{equation*}
	We can prove $u_{t} \to 0$ in $L^{2}(\Omega)$, which could simplify the representation of $\omega$-limit set.
	
	In fact, integrating \eqref{eqnewnewnewnew4.1} with respect to $t$ to get
	\begin{equation}\label{eq4.11}
		\int_{0}^{t}\left\|u_{t} \right\|_{L^{2}(\Omega)}^{2}d\tau=E(\mathbf{u_{0}})-E(\mathbf{u}(t))\le C_{2},\quad t> 0.
	\end{equation}
	Since $L^{2}(\Omega)\cong {(L^{2}(\Omega))}'\hookrightarrow V'$, we have \begin{equation*}
		\int_{0}^{t}\left\|u_{t} \right\|_{V'}^{2}d\tau\le \alpha_{1}\int_{0}^{t}\left\|u_{t} \right\|_{L^{2}}^{2}d\tau \le C_{3}.
	\end{equation*}
	Using the equation in \eqref{neweq1.2} and the boundedness of the orbit, we deduce that $u_{tt}\in L^{\infty}(0,\infty,V')$. This means
	\begin{equation*}
		\frac{d}{dt}\left\|u_{t} \right\|_{V'}^{2}\le 2\left\|u_{t} \right\|_{V'}\left\|u_{tt}\right\|_{V'}\le C_{4}.
	\end{equation*}
	By the classical theory of infinite integrals in mathematical analysis, when $t\to +\infty$,
	\begin{equation*}
		\left\| u_{t}\right\|_{V'}\to 0.
	\end{equation*}
	by the relatively compactness of the orbit, when $t\to +\infty$, we have
	\begin{equation*}
		\left\| u_{t}\right\|_{L^{2}}\to 0.
	\end{equation*}
	
	Combining the norm in $\mathbb{H}$ and the convergence of $u_{t}$, the  $\omega$-limit set can be simplified as:
	\begin{equation}\label{eq4.13}
		\omega (u_{0},u_{1})=\left\{(\psi,0):\exists \hspace{0.05cm}t_{n}\to +\infty \ \text{such that }\lim_{n\to \infty}\left\|u(t_{n},\cdot)-\psi(\cdot) \right\|_{H_{D}^{2}}=0\right\}.
	\end{equation}
	Thus, $$\omega(u_{0},u_{1})\subset \mathcal{S}\times \left\{0\right\}.$$
	When $(\psi,0)\in \omega(u_{0},u_{1})$, recall that $\left\|u(t) \right\|_{L^{\infty}}\le 1-\kappa$,
	by the definition of $\omega$-limit set \eqref{eq4.13}, there is a sequence $t_{n}\to +\infty$ such that
	\begin{equation*}
		\lim_{n\to \infty}\left\|u(t_{n})-\psi(\cdot) \right\|_{H_{D}^{2}}=0.
	\end{equation*}
	Since $H^{2}(\Omega)\hookrightarrow L^{\infty}(\Omega)$,
	\begin{equation*}
		\lim_{n\to \infty}\left\|u(t_{n})-\psi(\cdot) \right\|_{L^{\infty}}=0.
	\end{equation*}
	By the definition of limit, let $\varepsilon_{0}=\frac{\kappa}{2}$, then there is $n_{0}\in \mathbb{N}$ such that when $n\ge n_{0}$,
	\begin{equation*}
		\left\|u(t_{n})-\psi(\cdot) \right\|_{L^{\infty}}\le \frac{\kappa}{2}.
	\end{equation*}
	By the triangle inequality,
	\begin{equation*}
		\left\|\psi \right\|_{L^{\infty}}\le \left\|u(t_{n}) \right\|_{\infty}+\frac{\kappa}{2}\le 1-\frac{\kappa}{2}.
	\end{equation*}
	
		\subsection{Lojasiewicz-Simon inequality}\label{section3.2}
	In this subsection, we wil prove the Lojasiewicz-Simon inequality established in Lemma~\ref{newnewlemma4.3}, which plays an important role in the subsequent proof of the convergence of the solution.
	
	1999年, Haraux and Jendoubi \cite{ref5} proved generalized Lojasiewicz-Simon inequality, i.e. Lemmma~\ref{prethm2.8}. We will verify the conditions in this theorem to obtain the Lojasiewicz-Simon inequality in our case.
	
	Recall that:
	\begin{gather*}
		V=H_{D}^{2}(\Omega),\quad V'\text{is the dual space of $V$},\quad  A=-(B\Delta^{2}-T\Delta),\\
		f(u)=-\frac{\lambda}{(1+u)^{2}},\quad
		(\psi,0) \in \omega(u_{0},u_{1})\subset \mathcal{S}\times \left\{0\right\}.
	\end{gather*}
	Let $(\psi,0)\in \omega(u_{0},u_{1})$ be fixed. Let $u=v+\psi$, then $v$ satisfies
	\begin{equation}\label{eq4.14}
		\left.\left\{\begin{array}{ll} v_{tt}+v_{t}+B\Delta^2 v-T\Delta v=f(v,\psi),\quad &t>0, \ x\in\Omega ,\\
			v=\partial_\nu v=0,& x\in\partial\Omega ,\\
			v(0,\cdot)=u_{0}-\psi, \ \ v_{t}(0,\cdot)=u_{1},& x\in \Omega,\end{array}\right.\right.
	\end{equation}
	where
	\begin{align}
		f(v,\psi)=-\frac{\lambda}{(1+v+\psi)^{2}}+\frac{\lambda}{(1+\psi)^{2}},\label{eq4.17}
	\end{align}
	Let
	\begin{equation}\label{neweq11.1}
		F(v):=\int_{0}^{v}f(s,\psi)ds=\frac{\lambda v^{2}}{(1+v+\psi)^{2}(1+\psi)^{2}}.
	\end{equation}
	Now we are ready to prove the  Lojasiewicz-Simon lemma.
	\begin{lemma}\label{newnewlemma4.3}
		Under the assumptions of Theorem~\ref{theorem1.3}, there are two constants $\theta\in(0,\frac{1}{2})$ and $\sigma>0$ such that for all $v\in V$ satisfying $\left\|v \right\|_{V}<\sigma$,
		\begin{equation}\label{newlsinequality}
			\left\|-Av+f(v,\psi) \right\|_{V'}\ge |\mathcal{E}(v) |^{1-\theta}.
		\end{equation}
	\end{lemma}
	\begin{proof}
		Note that $\mathcal{E}(0)=0$, thus \eqref{newlsinequality} is equivalent to
		\begin{equation*}
			\left\|-Av+f(v,\psi) \right\|_{V'}\ge |\mathcal{E}(v)-\mathcal{E}(0)|^{1-\theta}.
		\end{equation*}
		This is consistent with form in the inequality in Theorem~\ref{prethm2.8}.
		
		By \cite[Section~3.1]{ref2}, $-A$ generates an analytic semigroup on $L^{2}(\Omega)$, then by th classical semigroup theory \cite[Theorem~2.6.8(c)]{ref15},
		$$\overline{H_{D}^{2}(\Omega)}=\overline{D(A^{\frac{1}{2}})}=L^{2}(\Omega).$$
		In addition, by Sobolev compact embedding theorem, the embedding $V\hookrightarrow L^{2}(\Omega)$ is compact.
		
		Let
		\begin{equation*}
			a(u,v)=(Au,v)_{L^{2}}.
		\end{equation*}
		We verify that $a(u,v)$ is a bilinear continuous form on $V$ which
		is symmetric and coercive:
		\begin{itemize}
			\item[(\romannumeral1)]  Bilinearity: Since $A$ is linear, for any $ u,v,w\in V$ and $k_{1},k_{2}\in \mathbb{R}$, we have
			\begin{align*}
				a(k_{1}u+k_{2}v,w)&=(A(k_{1} u+k_{2} v),w)_{L^{2}}\\
				&=k_{1}(Au,w)_{L^{2}}+k_{2}(Av,w)_{L^{2}}\\
				&=k_{1}a(u,w)+k_{2}a(v,w).
			\end{align*}
			\item[(\romannumeral2)] Continuity: Let $u,v,w\in V$ and assume that $u \to v$ in $V$, then
			\begin{align*}
				&\left|a(u,w)-a(v,w)\right|\\
				&=\left|a(u-v,w) \right|\\
				&= \left|\int_{\Omega}B\Delta^{2} (u-v) wdx-\int_{\Omega}T\Delta (u-v) wdx \right|\\
				&= \left|\int_{\Omega}B\Delta (u-v)\Delta wdx+\int_{\Omega}T\nabla (u-v)\nabla wdx \right|\\
				&\le B\left\|\Delta(u-v)\Delta w\right\|_{L^{1}}+T\left\|\nabla(u-v) \nabla w \right\|_{L^{1}}\\
				&\le B\left\|\Delta(u-v) \right\|_{L^{2}}\left\|\Delta w \right\|_{L^{2}}+T\left\|\nabla(u-v) \right\|_{L^{2}}\left\|\nabla w \right\|_{L^{2}}\to 0.
			\end{align*}
			\item[(\romannumeral3)] Symmetry: For any $ u,v\in V$,
			\begin{align*}
				a(u,v)&=\int_{\Omega}\left(B\Delta^{2}u-T\Delta u\right)vdx\\
				&=\int_{\Omega}\left(B\Delta u\Delta v+T\nabla u\nabla v \right)dx\\
				&=\int_{\Omega}\left(B\Delta^{2}v-T\Delta v\right)udx\\
				&=a(v,u).
			\end{align*}
			where Hölder inequality is employed.
			\item[(\romannumeral4)] Coercivity: For any $ u\in V$,
			\begin{equation*}
				a(u,u)=\int_{\Omega}\left(B\Delta^{2}u-T\Delta u\right)udx=\left\|u \right\|_{H_{D}^{2}}^{2}\ge C_{5}\left\|u \right\|_{L^{2}}^{2}.
			\end{equation*}
			where $H^{2}(\Omega)\hookrightarrow L^{2}(\Omega)$ is employed.
		\end{itemize}
		Let $p=2$, we verify the conditions (H1)-(H3):
		\begin{itemize}
			\item[(H1)] Note that $$A^{-1}(L^{2}(\Omega))=\left\{ v\in V|Av\in L^{2}(\Omega) \right\}.$$
			Recall that $H^{2}(\Omega)\hookrightarrow L^{\infty}(\Omega)$ is continuous, we obtain that $A^{-1}(L^{2}(\Omega))\hookrightarrow L^{\infty}(\Omega)$ is continuous.
			\item[(H2)] when $u\in D=      \left\{ v\in V|Av\in L^{2}(\Omega) \right\}$, $u\in V\subset L^{2}(\Omega)$.
			\item[(H3)] Recall that
			\begin{gather*}
				F(v):=\int_{0}^{v}f(s,\psi)ds=\frac{\lambda v^{2}}{(1+v+\psi)^{2}(1+\psi)^{2}}.
			\end{gather*}
			Similar to the discussion in Lemma~\ref{LSbudengshipaowupaowu}, if $\left\|h \right\|_{V}$ is small enough, $F(v)$ is analytic at $0$. Let $$B_{0}=\left \{v\in V,\left\|v \right\|_{V}<\rho \right \}.$$
			On the other hands, by \cite[引理~2.5]{ref5} and the proof of Lemma~\ref{LSbudengshipaowupaowu}, it suffices to show that $\frac{\partial F}{\partial v}$ and $\frac{\partial^{2} F}{\partial v^{2}}$ are bounded in $B_{0}$. Note that $(1+v+\psi)-(1+\psi)=v$, and $H^{2}(\Omega)\hookrightarrow L^{\infty}(\Omega)$, thus we can choose $\rho=\frac{\kappa}{4C}$ such that
			\begin{gather*}
				1+v+\psi\ge \frac{\kappa}{4},\\
				\left\|v \right\|_{L^{\infty}}\le C_{0}\left\|v \right\|_{H_{D}^{2}}\le \frac{\kappa}{4}.
			\end{gather*}
			Recall that in subsection~\ref{htdxt}, we obtain $\left\|\psi \right\|_{L^{\infty}}\le 1-\frac{\kappa}{2}$. Then
			\begin{gather*}
				\frac{\partial F}{\partial v}=f(v,\psi)=\frac{(2+v+2\psi)v}{(1+v+\psi)^{2}(1+\psi)^{2}}\le \frac{(2+\frac{\kappa}{4}+2-\kappa)(\frac{\kappa}{4})}{(\frac{\kappa}{4})^{2}(\frac{\kappa}{2})^{2}}=\frac{4(16-3\kappa)}{\kappa^{3}}.
			\end{gather*}
			Similarly, we can proof $\frac{\partial^{2} F}{\partial v^{2}}$ is bounded.
		\end{itemize}
		Finally, we show that $f\in C^{1}(B_{0},V')$. Similar to the discussion above, let $\rho=\frac{\kappa}{4C}$, for all $v\in B_{0}$, we have
		$$1+v+\psi\ge \frac{\kappa}{4},\quad  \left\|v \right\|_{L^{\infty}}\le \frac{\kappa}{4}.$$
		For given $v_{1}\in B_{0}$, for any $\varepsilon>0$, let $v_{2}\in B_{0}$ such that $$\left\|u_{1}-u_{2} \right\|_{V}\le \delta=\frac{\kappa^{4}}{128\alpha_{1}\alpha_{2}\lambda(8-\kappa)}\varepsilon.$$
		A direct calculation yields
		\begin{align*}
			\left\|f(v_{1})-f(v_{2}) \right\|_{V'}&=\left\|-\frac{\lambda}{(1+v_{1}+\psi)^{2}}+\frac{\lambda}{(1+v_{2}+\psi)^{2}} \right\|_{V'}\\
			&=\lambda\left\|\frac{(2+v_{1}+v_{2}+2\psi)(v_{1}-v_{2})}{(1+v_{1}+\psi)^{2}(1+v_{2}+\psi)^{2}} \right\|_{V'}\\
			&\le \frac{128\lambda(8-\kappa)}{\kappa^{4}}\left\|v_{1}-v_{2} \right\|_{V'}\\
			&\le \frac{128\alpha_{1}\alpha_{2}\lambda(8-\kappa)}{\kappa^{4}}\left\|v_{1}-v_{2} \right\|_{V}\le \varepsilon.
		\end{align*}
		Thus, $f\in C(B_{0},V')$.
		Similarly, we can prove $f'\in C(B_{0},V')$, which means $f\in C^{1}(B_{0},V')$.
		Therefore, the conditions listed in Theorem~\ref{prethm2.8} are satisfied, which indicates Lemma~\ref{newnewlemma4.3} is proved.
	\end{proof}
	
		\subsection{Proof of Theorem~\ref{theorem1.3}}
	In this subsection, we will prove the global solution to \eqref{neweq1.2} must converge to one stationary solution, and we establish the convergence rate.
	
	We will prove Theorem~\ref{Theorem 4.1}, which is consistent with the convergence established in \eqref{eq1.14}, our proof is motivated by \cite{ref1,ref5}. Recall that $V=H_{D}^{2}(\Omega)$, $u=v+\psi$, $v$ satisfies \eqref{eq4.14}. For $t\ge 0$, let
	\begin{gather*}
		\mathcal{E}(v):=\int_{\Omega}\left(\frac
		{1}{2}\left|v_{t}\right|^{2}+\frac{B}{2}\left|\Delta v \right|^{2}+\frac{T}{2}\left|\nabla v \right|^{2}-F(v,\psi)\right)dx,\\
		G(t):=\mathcal{E}(v)+\varepsilon\left(Av-f(v,\psi),v_{t}\right)_{V'\times V'},
	\end{gather*}
	where $f(v,\psi)$ is as in \eqref{eq4.17}, $F(v)$ is as in \eqref{neweq11.1}, $\varepsilon>0$ is being determined.
	\begin{theorem}\label{Theorem 4.1}
		Under the assumptions of Theorem~\ref{theorem1.3}, let $u$ be the global solution to \eqref{neweq1.2}, then there is $\psi \in \mathcal{S}$ such that
		\begin{equation*}
			\lim\limits_{t\to+\infty}\left\{\left\|u_{t} \right\|_{L^{2}}+\left\|u(t,\cdot)-\psi(\cdot)\right\|_{V}\right\}=0.
		\end{equation*}
	\end{theorem}
	\begin{proof}
		Since $v=u-\psi$, then the above equality is equivalent to
		\begin{equation*}
			\lim\limits_{t\to+\infty}\left\{\left\|v_{t} \right\|_{L^{2}}+\left\|v(t,\cdot)\right\|_{V}\right\}=0.
		\end{equation*}
		Using the equation in \eqref{eq4.14}, a direct calculation yields that
		\begin{gather}
			G'(t)=\left\|v_{t} \right\|_{L^{2}}^{2}+\varepsilon\Big\{\left(Av_{t}-f'(v,\psi)v_{t},v_{t}\right)_{V'\times V'}\notag\\
			+\left(Av-f(v,\psi),v_{tt}\right)_{V'\times V'}\Big\}.\label{eq4.18}
		\end{gather}
		where $$f'(v,\psi)=\frac{\partial f}{\partial v}(v,\psi).$$
		Using the equation satisfied by $v$ to get
		\begin{align}
			\left(Av-f(v,\psi),v_{tt}\right)_{V'\times V'}
			&=\left(Av-f(v,\psi),-Av+f(v,\psi)-v_{t}\right)_{V'\times V'}\notag\\
			&=-\left\|Av-f(v,\psi) \right\|_{V'}^{2}-\left(Av-f(v,\psi),v_{t}\right)_{V'\times V'}\label{eq4.20}
		\end{align}
		Combining \eqref{eq4.18} and \eqref{eq4.20} yields
		\begin{gather}
			G'(t)=-\left\|v_{t} \right\|_{L^{2}}^{2}+\varepsilon\Big\{\left(Av_{t}-f'(v,\psi)v_{t},v_{t}\right)_{V'\times V'}-\left\|Av-f(v,\psi) \right\|_{V'}^{2}\notag\\
			-\left(Av-f(v,\psi),v_{t}\right)_{V'\times V'}\Big\}\label{eq4.21},
		\end{gather}
		By Definition~\ref{neiji}, we have
		\begin{align*}
			( Av_{t},v_{t} )_{V'\times V'}=( (A_{\Delta^{2}}-I)v_{t},v_{t} )_{V'\times V'}\le( A_{\Delta^{2}}v_{t},v_{t} )_{V'\times V'}.
		\end{align*}
		Let $w=A_{\Delta^{2}}^{-1}v_{t}$, by Definition~\ref{neiji}, $w\in V$. Then $w$ satisfies
		\begin{equation*}
			\left.\left\{\begin{array}{ll}B\Delta^2w-T\Delta w+w=v_{t},\quad &t>0, \ x\in\Omega ,\\
				w=\partial_\nu w=0,& x\in\partial\Omega.\end{array}\right.\right.
		\end{equation*}
		Consider the homogeneous equation
		\begin{equation*}
			\left.\left\{\begin{array}{ll}B\Delta^2w-T\Delta w+w=0,\quad &t>0, \ x\in\Omega ,\\
				w=\partial_\nu w=0,& x\in\partial\Omega.\end{array}\right.\right.
		\end{equation*}
		Multiplying the above equation by $w$, then integrating over $\Omega$ yields
		\begin{equation*}
			\int_{\Omega}B|\Delta w|^{2}+T|\nabla w|^{2}+|w|^{2}=0.
		\end{equation*}
		This means $w\equiv 0$, i.e. $Ker(A)=0$. Then by priori estimate of ellptic equation, we get
		\begin{equation}\label{eq4.104}
			\left\|w \right\|_{H^{4}}\le C_{6}\left\|v_{t} \right\|_{L^{2}},
		\end{equation}
		where $C_{6}$ is a constant independent of $w$. Since $w$ and $v_{t}$ satisfies Dirichlet boundary condition, a direct calculation yields
		\begin{align*}
			( A_{\Delta^{2}}v_{t},v_{t} )_{V'\times V'}&=( A_{\Delta^{2}}v_{t},A_{\Delta^{2}}^{-1}v_{t} )_{V'\times V}\\
			&=\int_{\Omega}v_{t}wdx+T\int_{\Omega}\nabla v_{t} \nabla w dx+B\int_{\Omega}\Delta v_{t} \Delta wdx\\
			&=\int_{\Omega}v_{t}\left(B\Delta^{2}w-T\Delta w+w \right)dx.
		\end{align*}
		By Hölder inequality and \eqref{eq4.104}, we have
		\begin{align}
			( A_{\Delta^{2}}v_{t},v_{t} )_{V'\times V'}&\le \left\|v_{t} \right\|_{L^{2}} \left\|B\Delta^{2}w-T\Delta w+w \right\|_{L^{2}}\notag\\
			&\le C_{7}\left\|v_{t} \right\|_{L^{2}} \left\|w \right\|_{H^{4}}\notag\\
			&\le C_{6}C_{7} \left\|v_{t} \right\|_{L^{2}}^{2}.\label{eq4.102}
		\end{align}
		Since $u\ge -1+\kappa$,
		\begin{equation*}
			\left|f'(v,\psi) \right|=\left|\frac{2\lambda}{(1+u)^{3}} \right|\le2\lambda\kappa^{-3}.
		\end{equation*}
		Employing Cauchy-Schwartz inequality yields
		\begin{align}
			(-f'(v,\psi)v_{t},v_{t})_{V'\times V'}&\le \left\|f'(v,\psi)v_{t} \right\|_{V'}\left\|v_{t} \right\|_{V'}\notag\\
			&\le 2\lambda\kappa^{-3}\left\|u_{t} \right\|_{V'}^{2}\notag\\
			&\le 2\lambda \kappa^{-3}\alpha_{1} \left\|v_{t} \right\|_{L^{2}}^{2}.\label{eq4.103}
		\end{align}
		By Young's inequality, we get
		\begin{equation}\label{eq4.23}
			-\left(Av-f(v,\psi),v_{t}\right)_{V'}\le \frac{1}{2}\left\|Av-f(v,\psi) \right\|_{V'}^{2}+\frac{1}{2}\left\|v_{t} \right\|_{V'}^{2}.
		\end{equation}
		Combining \eqref{eq4.21}, \eqref{eq4.102}-\eqref{eq4.23}, we obtain that
		\begin{equation*}
			G'(t)\le \left(-1+C_{8}\varepsilon+\frac{\varepsilon}{2}\right)\left\|v_{t} \right\|_{L^{2}(\Omega)}^{2}-\frac{\varepsilon}{2}\left\|Av-f(v,\psi) \right\|_{V'}^{2}.
		\end{equation*}
		Let $\varepsilon>0$ be small enough, we obtain that there is a constant $C_{9}>0$ such that for all $t\ge 0$
		\begin{align}
			-G'(t)&\ge C_{9}\left(\left\|v_{t} \right\|_{L^{2}}^{2}+\left\|-Av+f(v,\psi) \right\|_{V'}^{2}\right)\notag\\
			&\ge \frac{C_{9}}{2}\left(\left\|v_{t} \right\|_{L^{2}}+\left\|-Av+f(v,\psi) \right\|_{V'}\right)^{2}.\label{eq4.24}
		\end{align}
		This means $G(t)$ is monotone non-increasing on $[0,+\infty)$. Since $(\psi,0)\in\omega(u_{0},u_{1})\subset\mathcal{S}\times \{0\}$, there is a sequence $t_{n}\to +\infty$ such that $v(t_{n},\cdot)\to 0$ in $V$. Recall that $\left\|v_{t} \right\|_{L^{2}}\to 0$, we have $\mathcal{E}(v(t_{n}))\to0$. On the other hands, since $Av\in V'$ and $f\in C^{1}(V,V')$, by Cauchy-Schwartz inequality, we obtain that
		\begin{align*}
			(-Av+f(v,\psi),v_{t})_{V'\times V'}&\le \left\|-Av+f(v,\psi) \right\|_{V'}\left\|v_{t} \right\|_{V'}\\
			&\le \alpha_{1}\left\|-Av+f(v,\psi) \right\|_{V'}\left\|v_{t} \right\|_{L^{2}}\to 0.
		\end{align*}
		Thus by the definition of $G(t)$, we get $G(t_{n})\to 0$. Since $G'(t)\le 0$, we obtain $G(t_{n})\ge 0$. Then for any $\varepsilon_{0}>0$, there is an integer $n_{0}>0$ such that when $n\ge n_{0}$, $G(t_{n})<\varepsilon_{0}$. For $t\in [t_{n_{0}},t_{n_{0}+1}]$, we have
		\begin{equation*}
			0\le G(t_{n_{0}+1})\le G(t)\le G(t_{n_{0}})<\varepsilon.
		\end{equation*}
		This means for any $\varepsilon_{0}>0$, there is $t_{n_{0}}>0$ such that when $t\in [t_{n_{0}},t_{n_{0}+1}]$, $0\le G(t)<\varepsilon$. Since $G$ is monontone non-increasing, when $t\ge t_{n_{0}}$, $0<G(t)<\varepsilon$. Thus, by the definition of limit, when $t \to +\infty$, $G(t)\to 0$. Therefore $G(t)\ge 0$ on $[0,\infty)$.
		
		Let us continue the proof of Theorem~\ref{Theorem 4.1}. Let $\theta\in(0,\frac{1}{2})$ is as in Lemma~\ref{newnewlemma4.3}, we have
		\begin{equation}\label{eq4.26}
			-\frac{d}{dt}\left(G(t)\right)^{\theta}=-\theta G'(t)\left(G(t)\right)^{\theta-1},\quad t\ge 0.
		\end{equation}
		Applying the generalized triangle inequality and Cauchy Schwartz inequality, we obtain
		\begin{equation}\label{eq4.27}
			\left(G(t)\right)^{1-\theta}\le C_{10}\left\{\left\|v_{t} \right\|_{L^{2}}^{2(1-\theta)}+|{\mathcal{E}}(v)|^{1-\theta}+\left\|-Av+f(v,\psi) \right\|_{V'}^{1-\theta}\left\|v_{t} \right\|_{L^{2}}^{1-\theta}\right\}.
		\end{equation}
		Let $a=\left\|-Av+f(v,\psi) \right\|_{V'}^{1-\theta}$, $b=\left\|v_{t} \right\|_{L^{2}}^{1-\theta}$. By Young's ineuqality, we have
		\begin{equation*}
			ab\le \frac{a^{p}}{p}+\frac{b^{q}}{q},
		\end{equation*}
		where $p,q\ge 1$ and $\frac{1}{p}+\frac{1}{q}=1$. let $p=\frac{1}{1-\theta}$, $q=\frac{1}{\theta}$, we get
		\begin{align}
			\left\|-Av+f(v,\psi) \right\|_{V'}^{1-\theta}\left\|v_{t} \right\|_{L^{2}}^{1-\theta}&\le (1-\theta)\left\|-Av+f(v,\psi) \right\|_{V'}+\theta \left\|v_{t} \right\|_{L^{2}}^{\frac{1-\theta}{\theta}}\notag\\
			&\le \left\|-Av+f(v,\psi) \right\|_{V'}+\left\|v_{t} \right\|_{L^{2}}^{\frac{1-\theta}{\theta}}.\label{eq4.28}
		\end{align}
		By\eqref{eq4.27} and \eqref{eq4.28},
		\begin{equation*}
			\left(G(t)\right)^{1-\theta}\le C_{11}\left\{\left\|v_{t} \right\|_{L^{2}}^{2(1-\theta)}+|\mathcal{E}(v)|^{1-\theta}+\left\|-Av+f(v,\psi) \right\|_{V'}+\left\|v_{t} \right\|_{L^{2}}^{\frac{1-\theta}{\theta}}\right\}.
		\end{equation*}
	Since $0<\theta<\frac{1}{2}$, $2(1-\theta)>1$ and $\frac{1-\theta}{\theta}$. This means when $t$ is large enough,
		\begin{equation*}
			\left\|v_{t} \right\|_{L^{2}}^{2(1-\theta)}\le \left\|v_{t} \right\|_{L^{2}},\quad \left\|v_{t} \right\|_{L^{2}}^{\frac{1-\theta}{\theta}}\le \left\|v_{t} \right\|_{L^{2}}.
		\end{equation*}
		Recall that $\lim_{t\to +\infty}\left\|v_{t} \right\|_{L^{2}}=0$. Thus, when $t\ge \widetilde{T}_{0}$,
		\begin{equation}\label{eq4.29}
			\left(G(t)\right)^{1-\theta}\le 2C_{11}\left\{\left\|v_{t} \right\|_{L^{2}}+|\mathcal{E}(v)|^{1-\theta}+\left\|-Av+f(v,\psi) \right\|_{V'}\right\}.
		\end{equation}
		Since $\lim_{t\to +\infty}G(t)= 0$, for any $0<\eta<\sigma$ (Here $\sigma>0$ is as in Lemma~\ref{newnewlemma4.3}), there is $N\in \mathbb{N}$ such that when $n\ge N$,
		\begin{equation}\label{eq4.30}
			\left\|v(t_{n},\cdot) \right\|_{V}<\frac{\eta}{2},\quad\left\|v(t_{n},\cdot) \right\|_{L^{2}}<\frac{\eta}{2},\quad \frac{8C_{11}}{\theta C_{9}}\left(G(t_{n})\right)^{\theta}<\frac{\eta}{2}.
		\end{equation}
		We can choose $N$ large enough such that $t_{N}\ge \widetilde{T}_{0}$ and $\left\|v(t_{N},\cdot) \right\|_{V}<\sigma$. Define
		\begin{equation*}
			\bar{t}_{N}=\text{sup}\left\{t\ge t_{N}:\left\|v(s,\cdot) \right\|_{V}<\sigma,\ \forall s\in[t_{N},t]\right\}.
		\end{equation*}
		We claim that $\bar{t}_{N}=\infty$. If $\bar{t}_{N}<\infty$, when $t\in[t_{N},\bar{t}_{N}]$, it follows from \eqref{newlsinequality} and \eqref{eq4.29} that
		\begin{equation}\label{eq4.31}
			\left(G(t)\right)^{1-\theta}\le 4C_{11}\left\{\left\|v_{t} \right\|_{L^{2}}+\left\|-Av+f(v,\psi) \right\|_{V'}\right\}.
		\end{equation}
		Then we can deduce from \eqref{eq4.24}, \eqref{eq4.26} and \eqref{eq4.31} that
		\begin{equation}\label{eq4.32}
			-\frac{d}{dt}\left(G(t)\right)^{\theta}\ge \frac{\theta C_{9}}{8C_{11}}\left\{\left\|v_{t} \right\|_{L^{2}}+\left\|-Av+f(v,\psi) \right\|_{V'}\right\}\ge \frac{\theta C_{9}}{8C_{11}}\left\|v_{t} \right\|_{L^{2}}.
		\end{equation}
		Since $G(t)\ge 0$. Integrating \eqref{eq4.32} with respect to $t$ yields
		\begin{equation}\label{eq4.33}
			\int_{t_{N}}^{\bar{t}_{N}}\left\|v_{t} \right\|_{L^{2}}d\tau\le \frac{8C_{11}}{\theta C_{9}}\left(G(t_{N})\right)^{\theta}-\frac{8C_{11}}{\theta C_{9}}\left(G(\bar{t}_{N})\right)^{\theta}\le \frac{8C_{11}}{\theta C_{9}}\left(G(t_{N})\right)^{\theta}.
		\end{equation}
		On the other hands, using \eqref{eq4.30} and \eqref{eq4.33}, we get
		\begin{align*}
			\left\|v({\bar{t}_{N}}) \right\|_{L^{2}}&\le\int_{t_{N}}^{\bar{t}_{N}}\left\|v_{t} \right\|_{H}d\tau+\left\|v({{t}_{N}}) \right\|_{L^{2}}\\
			&\le \frac{8C_{11}}{\theta C_{9}}\left(G(t_{N})\right)^{\theta}+\left\|v({{t}_{N}}) \right\|_{L^{2}}\\
			&<\eta.
		\end{align*}
		This indicates
		\begin{equation*}
			\lim_{n\to \infty}\left\|v(\bar{t}_{N}) \right\|_{L^{2}}=0.
		\end{equation*}
		By the relatively compactness of the orbit $$\bigcup_{t\ge t_0}\left\{u(t,\cdot),u_{t}(t,\cdot)\right\},$$, there is a subsequence of $u(\bar{t}_{N})$, still denote itself, such that
		\begin{equation*}
			\lim_{n\to \infty}\left\|v(\bar{t}_{N}) \right\|_{V}=0.
		\end{equation*}
		Thus, there is $N'\ge N$ such that when $n\ge N'$,
		\begin{equation*}
			\left\|u(\bar{t}_{N})-\psi \right\|_{V}<\frac{\eta}{2}<\frac{\sigma}{2},
		\end{equation*}
		which contradicts the definition of $\bar{t}_{N}$. Thus, $\bar{t}_{N}=\infty$. Since $\lim_{t\to+\infty}G(t)\to 0$, by \eqref{eq4.33}, we obtain
		\begin{equation}\label{eq4.34}
			\int_{t_{N}}^{\infty}\left\|v_{t} \right\|_{L^{2}}d\tau<\infty.
		\end{equation}
		Note that
		\begin{equation}\label{eq4.35}
			\left\|v(t) \right\|_{L^{2}}\le\int_{t}^{\infty}\left\|v_{t} \right\|_{L^{2}}d\tau.
		\end{equation}
		Thus, we can deduce the convergence of $v$ in $L^{2}(\Omega)$ from \eqref{eq4.34} and \eqref{eq4.35}. Similar to the discussion in \eqref{htdxt}, by the relatively compactness of the orbit, we obtain the convergence of $v$ in $V$. Therefore, Theorem~\ref{Theorem 4.1} is proved.
	\end{proof}
	
	Based on Theorem~\ref{Theorem 4.1}, we will establish the corresponding convergence rate.
	
	Recall that $G(t)$ is monontone non-increasing on $[0,+\infty)$ and $G(t)\ge 0$. If $G(0)=0$, then the situation is trivial. Thus we assume that $G(0)>0$.
	By \eqref{eq4.24} and \eqref{eq4.31}, we get
	\begin{equation}\label{neweq15.1}
		G'(t)+c_{1}G(t)^{2(1-\theta)}\le 0.
	\end{equation}
	Similar to the proof of Lemma~\ref{L2shoulian}, we obtain that when $t$ is large enough,
	\begin{equation}\label{neweq14.1}
		G(t)\le (c_{2}+c_{3}t)^{-\gamma}.
	\end{equation}
	where $c_{2}=G(0)^{2\theta-1}$, $c_{3}=c_{1}(1-2\theta)$, $\gamma=\frac{1}{1-2\theta}$. Motivated by \cite{ref31}, we will prove the following theorem, which establish the convergence rate.
	\begin{theorem}\label{thm4.54.5}
		Assume that $\theta\in(0,\frac{1}{2})$ is as in Lemma~\ref{newnewlemma4.3}, then there are constants $C>0$, $\gamma=\frac{1}{1-2\theta}>0$ and $T_{0}>0$ such that when $t\ge T_{0}$,
		\begin{equation}\label{neweq17.1}
			\left\|u-\psi \right\|_{L^{2}}\le C(1+t)^{-\gamma}.
		\end{equation}
	\end{theorem}
	\begin{proof}
		Recall that $u=v+\psi$, then \eqref{neweq17.1} is equivalent to
		\begin{equation}\label{neweq17.3}
			\left\|v \right\|_{L^{2}}\le C(1+t)^{-\gamma}.
		\end{equation}
		By \eqref{eq4.24}, we get
		\begin{equation}\label{eq4.58}
			\left\|v_{t} \right\|_{L^{2}}^{2}\le -\frac{1}{C_{9}}G'(t).
		\end{equation}
		When $t$ is large enough, for $k\in \mathbb{N}$, integrating \eqref{eq4.58} with respect to $t$ yields
		\begin{equation*}
			\int_{2^{k}t}^{2^{k+1}t}\left\|v_{t}(\tau) \right\|_{L^{2}}^{2}d\tau \le \frac{1}{C_{9}}(G(2^{k}t)-G(2^{k+1}t))\le \frac{1}{C_{9}}G(2^{k}t)\le \frac{c_{3}^{-\gamma}}{C_{9}}(2^{k}t)^{-\gamma}.
		\end{equation*}
		By Hölder inequality,
		\begin{equation*}
			\int_{2^{k}t}^{2^{k+1}t}\left\|v_{t}(\tau) \right\|_{L^{2}}d\tau\le \sqrt{2^{k}t}\left(\int_{2^{k}t}^{2^{k+1}t}\left\|v_{t}(\tau) \right\|_{L^{2}}^{2}d\tau\right)^{\frac{1}{2}}\le c_{4}(2^{k}t)^{\frac{1-\gamma}{2}}.
		\end{equation*}
		Then when $t$ is large enough,
		\begin{equation}\label{eq4.59}
			\int_{t}^{\infty}\left\|v_{t}(\tau) \right\|_{L^{2}}d\tau=\sum_{k=0}^{+\infty}\int_{2^{k}t}^{2^{k+1}t}\left\|v_{t}(\tau) \right\|_{L^{2}}d\tau \le c_{4} \frac{1}{1-2^{-\frac{\theta}{1-2\theta}}}(1+t)^{-\frac{\gamma-1}{2}}.
		\end{equation}
		In fact, we have
		\begin{gather*}
			\lim_{t\to+\infty}\frac{(c_{2}+c_{3}2^{k}t)^{-\gamma}}{c_{3}^{-\gamma}(2^{k}t)^{-\gamma}}=1,\\
			\lim_{t\to +\infty}\frac{t^{-\frac{\gamma-1}{2}}}{(1+t)^{-\frac{\gamma-1}{2}}}=1.
		\end{gather*}
		It follows from \eqref{eq4.59} that
		\begin{equation*}
			\left\|v \right\|_{L^{2}}\le 	\int_{t}^{+\infty}\left\|u_{t} \right\|_{L^{2}}d\tau\le c_{5}(1+t)^{-\frac{\theta}{1-2\theta}}.
		\end{equation*}
		Theorem~\ref{thm4.54.5} is proved.
	\end{proof}
	
	Based on Theorem~\ref{thm4.54.5}, we can further obtain the convergence rate of the polynomial $\left\| (v,v_{t})\right\|_{\mathbb{H}}$, as shown in the following theorem.
	\begin{theorem}\label{date4.11thm}
		Assume that $\theta\in(0,\frac{1}{2})$ is as in Lemma~\ref{newnewlemma4.3}, then there are constants $\widetilde{C}>0$, $\gamma=\frac{1}{1-2\theta}>0$ and $\widetilde{T}_{0}>0$ such that when $t\ge \widetilde{T}_{0}$,
		\begin{equation}\label{date4.11}
			\left\|u_{t} \right\|_{L^{2}}+\left\|u-\psi \right\|_{H_{D}^{2}}\le \widetilde{C}(1+t)^{-\gamma}.
		\end{equation}
	\end{theorem}
	\begin{proof}
		Recall that $u=v+\psi$, then \eqref{date4.11} is equivalent to
		\begin{equation}\label{eqdate4.40}
			\left\|v_{t} \right\|_{L^{2}}+\left\|v \right\|_{H_{D}^{2}}\le \widetilde{C}(1+t)^{-\gamma}.
		\end{equation}
		Recall that $v$ satisfies \eqref{eq4.14}, we note that
		$$f(v,\psi)=\frac{\lambda(2+u+\psi)}{(1+u)^{2}(1+\psi)^{2}}v.$$
		Let $$\tilde{f}=\frac{\lambda(2+u+\psi)}{(1+u)^{2}(1+\psi)^{2}}.$$
		Then $v$ satisfies
		\begin{equation}\label{date4.12}
			v_{tt}+v_{t}+B\Delta^{2}v-T\Delta v=\tilde{f}v.
		\end{equation}
		Recall that in subsection~\ref{nlhds}, we obtain $$\left\|\psi \right\|_{L^{\infty}}\le 1-\frac{\kappa}{2}.$$
		By $\left\|u \right\|_{L^{\infty}}\le 1-\kappa$,  we obtain that $\tilde{f}$ is bounded:
		\begin{equation}
			|\tilde{f}|\le \frac{\lambda(2+1-\kappa+1-\kappa/2)}{(\kappa)^{2}(\kappa/2)^{2}}:=c_{6}.
		\end{equation}
		It is seen from \eqref{eqnewnewnewnew4.1}, the following equality holds:
		\begin{equation}\label{eqdate4.14}
			\frac{d}{dt}\left(\frac{1}{2}\left\|v_{t} \right\|_{L^{2}}^{2}+\frac{1}{2}\left\|v \right\|_{H_{D}^{2}}^{2}\right)+\left\|v_{t} \right\|_{L^{2}}^{2}=\int_{\Omega}\tilde{f}vv_{t}dx.
		\end{equation}
		Since $\tilde{f}$ is bounded,
		\begin{align}
			\int_{\Omega}\tilde{f}vv_{t}dx&\le |\int_{\Omega}\tilde{f}vv_{t}dx|\notag\\
			&\le |\tilde{f}|\int_{\Omega}|vv_{t}|dx\notag\\
			&\le c_{6}\int_{\Omega}\left(\frac{c_{6}}{2}v^{2}+\frac{1}{2c_{6}}v_{t}^{2}\right)\notag\\
			&=\frac{c_{6}^{2}}{2}\left\|v \right\|_{L^{2}}^{2}+\frac{1}{2}\left\|v_{t} \right\|_{L^{2}}^{2}.\label{eqdate4.13}
		\end{align}
		Combining \eqref{eqdate4.14} and \eqref{eqdate4.13} yields
		\begin{equation}\label{eqdate4.15}
			\frac{d}{dt}\left(\left\|v_{t} \right\|_{L^{2}}^{2}+\left\|v \right\|_{H_{D}^{2}}^{2}\right)+\left\|v_{t} \right\|_{L^{2}}^{2}\le c_{6}^{2}\left\|v \right\|_{L^{2}}^{2}.
		\end{equation}
		Multiplying the equation in \eqref{date4.12} by $v$, then integrating over $\Omega$ yields
		\begin{align}
			\frac{d}{dt}\left(\frac{1}{2}\left\|v \right\|_{L^{2}}^{2}+\int_{\Omega}vv_{t}dx\right)+\left\|v \right\|_{H_{D}^{2}}^{2}&=\int_{\Omega}\tilde{f}v^{2}+\left\|v_{t} \right\|_{L^{2}}^{2}\notag\\
			&\le c_{6}\left\|v \right\|_{L^{2}}^{2}+\left\|v_{t} \right\|_{L^{2}}^{2}\label{eqdate4.17}.
		\end{align}
		 Multiplying \eqref{eqdate4.17} by $\frac{1}{2}$, then adding \eqref{eqdate4.15} to the resultant, we obtain
		\begin{equation}\label{eqdate4.20}
			\frac{d}{dt}\widetilde{E}+\frac{1}{2}\left\|v_{t} \right\|_{L^{2}}^{2}+\frac{1}{2}\left\|v \right\|_{H_{D}^{2}}^{2}\le (c_{6}^{2}+\frac{c_{6}}{2})\left\|v \right\|_{L^{2}}^{2},
		\end{equation}
		where
		\begin{equation*}
			\widetilde{E}=\left\|v_{t} \right\|_{L^{2}}^{2}+\left\|v \right\|_{H_{D}^{2}}^{2}+\frac{1}{4}\left\|v \right\|_{L^{2}}^{2}+\frac{1}{2}\int_{\Omega}vv_{t}dx.
		\end{equation*}
		Adding $\left\|v \right\|_{L^{2}}^{2}$ on both sides of \eqref{eqdate4.20} yields
		\begin{equation}\label{eqdate4.21}
			\frac{d}{dt}\widetilde{E}+\frac{1}{2}\left\|v_{t} \right\|_{L^{2}}^{2}+\frac{1}{2}\left\|v \right\|_{H_{D}^{2}}^{2}+\left\|v \right\|_{L^{2}}^{2}\le (c_{6}^{2}+\frac{c_{6}}{2}+1)\left\|v \right\|_{L^{2}}^{2}.
		\end{equation}
		By Young's inequality, we get
		\begin{equation}\label{eqdate4.30}
			\widetilde{E}\le \frac{5}{4} \left\|v_{t} \right\|_{L^{2}}^{2}+\left\|v \right\|_{H_{D}^{2}}^{2}+\frac{1}{2}\left\|v \right\|_{L^{2}}^{2}.
		\end{equation}
	Note that there is $c_{7}=\frac{2}{5}$ such that
		\begin{equation}\label{eqdate4.28}
			c_{7}\left(	\frac{5}{4} \left\|v_{t} \right\|_{L^{2}}^{2}+\left\|v \right\|_{H_{D}^{2}}^{2}+\frac{1}{2}\left\|v \right\|_{L^{2}}^{2}\right)\le \frac{1}{2}\left\|v_{t} \right\|_{L^{2}}^{2}+\frac{1}{2}\left\|v \right\|_{H_{D}^{2}}^{2}+\left\|v_{t} \right\|_{L^{2}}^{2}.
		\end{equation}
		By \eqref{eqdate4.21} and \eqref{eqdate4.28}, we obtain that
		\begin{equation*}
			\frac{d}{dt}\widetilde{E}+c_{7}\widetilde{E}\le (c_{6}^{2}+\frac{c_{6}}{2}+1)\left\|v \right\|_{L^{2}}^{2}.
		\end{equation*}
		Since $\left\|v \right\|_{L^{2}}$ exists polynomial estimation, then by the proof of Theorem~\ref{shoulianH4}, we can similarly deduce that, there is $c_{8}>0$, $\gamma=\frac{1}{1-2\theta}$ and $T_{1}>0$ such that
		\begin{equation}
			\widetilde{E}(t)\le c_{8}^{2}(1+t)^{-2\gamma}, \quad t\ge T_{1}.
		\end{equation}
		By the definition of $\widetilde{E}$, we get
		\begin{align*}
			\left\|v_{t} \right\|_{L^{2}}^{2}+\left\|v \right\|_{H_{D}^{2}}^{2}+\frac{1}{4}\left\|v \right\|_{L^{2}}^{2}&\le c_{8}^{2}(1+t)^{-2\gamma}- \frac{1}{2}\int_{\Omega}vv_{t}dx\\
			&\le c_{8}^{2}(1+t)^{-2\gamma}+ \frac{1}{4}\left\|v \right\|_{L^{2}}^{2}+\frac{1}{4}\left\|v_{t} \right\|_{L^{2}}^{2}.
		\end{align*}
		By the polynomial estimation of $\left\|v \right\|_{L^{2}}$, there is $c_{9}>0$ such that
		\begin{equation*}
			\left\|v_{t} \right\|_{L^{2}}^{2}+\left\|v \right\|_{H_{D}^{2}}^{2}\le c_{9}^{2}(1+t)^{-2\gamma}.
		\end{equation*}
		Thus
		\begin{gather*}
			\left\|v_{t} \right\|_{L^{2}}\le c_{9}(1+t)^{-\gamma}\\
			\left\|v \right\|_{H_{D}^{2}}\le c_{9}(1+t)^{-\gamma}.
		\end{gather*}
	Adding the two inequality yields \eqref{eqdate4.40}. Theorem~\ref{date4.11thm} is proved.
	\end{proof}
	
	Theorem~\ref{theorem1.3} is completly proved.
	
	\subsection{Comparison with parabolic MEMS equation}
 In this subsection, we summarize the two types of problems addressed in Section 3 and 4, highlighting their similarities and differences.
	
	\textbf{Similarities:}
	\begin{itemize}
		\item[1.] Both types of equations can be formulated as classical Cauchy problems:
		\begin{equation*}
			\left\{\begin{array}{ll}\dot{u}+Au=f(u),\\
				u(0)=u_{0}.\end{array}\right.
		\end{equation*}
		Furthermore, a solution-defined $C_{0}$-semigroup exists:
		$$u(t)=S(t)u_{0}.$$
		\item[2.] Both problems are approached using the same methodology:
		\begin{itemize}
			\item[] \textbf{Step 1.} Prove the problem defines a gradient system. Utilize properties of gradient systems to establish that the $\omega$-limit set is non-empty and consists of steady-state solutions.
			\item[] \textbf{Step 2.} Prove the Lojasiewicz-Simon inequality corresponding to the problem.
			\item[] \textbf{Step 3.} Construct an auxiliary function. Apply the Lojasiewicz-Simon inequality to prove convergence of the solution in $L^{2}$. Leverage the relative compactness of the trajectory to extend this convergence to spaces of higher regularity.
			\item[] \textbf{Step 4.} Utilize the differential inequality satisfied by the auxiliary function to derive the convergence rate in $L^{2}$.
			\item[] \textbf{Step 5.} Utilize the differential inequality satisfied by a new energy functional to extend the convergence rate from $L^{2}$ to spaces of higher regularity.
		\end{itemize}
	\end{itemize}
	
		\textbf{Differences:}
	\begin{itemize}
		\item[1.] When proving the relative compactness of the trajectory, the hyperbolic problem cannot boost the solution's regularity by leveraging \textit{a priori} estimates for elliptic problems in the same way the parabolic problem can. Consequently, we decompose the operator semigroup and employ Webb's Theorem \cite{ref9} to establish relative compactness of the trajectory.
		\item[2.] For a fixed time $t$, the solution spaces differ: the parabolic problem considers $H_{D}^{4}(\Omega)$, while the hyperbolic problem considers $H_{D}^{2}(\Omega)\times L^{2}(\Omega)$. Additionally, the spaces in which the expression $-Au+f(u)$ resides within the respective Lojasiewicz-Simon inequalities differ: $L^{2}(\Omega)$ for the parabolic problem and $(H_{D}^{2}(\Omega))^{'}$ (the dual space of $H_{D}^{2}(\Omega)$) for the hyperbolic problem.
		\item[3.] Due to the favorable regularity of solutions to the parabolic problem, multiplying both sides of the equation by $u_{t}$ and integrating over the domain $\Omega$ directly yields the desired energy functional. However, for the hyperbolic problem, the insufficient solution regularity necessitates proving the energy dissipation identity using a density argument.
	\end{itemize}
	
\section{Numerical Simulations}\label{szmn666}
In this section, we employ numerical methods to visualize solutions of the parabolic problem \eqref{neweq1.1} and hyperbolic problem \eqref{neweq1.2} discussed in the previous two chapters under fixed parameters. We then analyze and summarize these visualizations.
\subsection{Numerical Solutions for the Parabolic MEMS Equation}\label{paowumems}
This section focuses on numerical solutions of problem \eqref{neweq1.1} with fixed parameters. Setting $(B,T)=(0.01,1)$, $\Omega=(-1,1)$, $u_{0}=0$, and taking a step size of 0.04539 for varying $\lambda$, we obtain Figure~\ref{tgudingde2d} and Figure~\ref{xgudingde2d}.

In Figure~\ref{tgudingde2d}, we observe that the solution is symmetric about the domain $\Omega$ and attains its minimum at $x=0$. For fixed $x$, the solution value decreases as $\lambda$ increases. In Figures~\ref{parabolic2D0}-\ref{parabolic2D1}, we note that when $\lambda\le 0.4538$, the solution stabilizes over an extended time interval. Furthermore, at any fixed time $t\in(0,70)$, the solution value decreases as $\lambda$ increases. When $\lambda$ increases from $0.4538$ to $0.4539$, the solution value rapidly decreases, as shown in Figures~\ref{parabolic2D2}-\ref{parabolic2D3}. In Figures~\ref{parabolic2D4}-\ref{parabolic2D5}, we observe that for $\lambda\ge 0.4539$, the solution exhibits a rapid decreasing trend, and at any fixed time $t\in(0,0.4)$, the solution value decreases as $\lambda$ increases.

Additionally, we plot three-dimensional visualizations for parameters $\lambda=0.4538$ and $\lambda=0.4539$ to enhance observation, shown in Figure~\ref{parabolic3D}. Based on the above analysis, we propose the following conjecture.

\begin{conjecture}
	Assume $(B,T)$, $u_{0}$, $d$, and $\Omega$ are given, and $u$ is the unique maximal solution to problem \eqref{neweq1.1}. We make the following conjectures: there exists a critical value $\lambda_{p}^{*}>0$ such that
	\begin{enumerate}
		\item[\textup{1.}] For $0<\lambda<\lambda_{p}^{*}$, $u(t,x)$ exists globally, and for any fixed $(t,x)\in \mathbb{R}^{+}\times \Omega$, $u(t,x)$ is monotonically decreasing in $\lambda$;
		\item[\textup{2.}] For $\lambda >\lambda_{p}^{*}$, $u$ reaches the value $-1$ in finite time, and this time is monotonically decreasing in $\lambda$.
	\end{enumerate}
\end{conjecture}

\begin{figure}[H]
	\centering
	\subfigure[$t=10000$]{
		\label{parabolic2DNEW0}
		\includegraphics[width=0.39\linewidth]{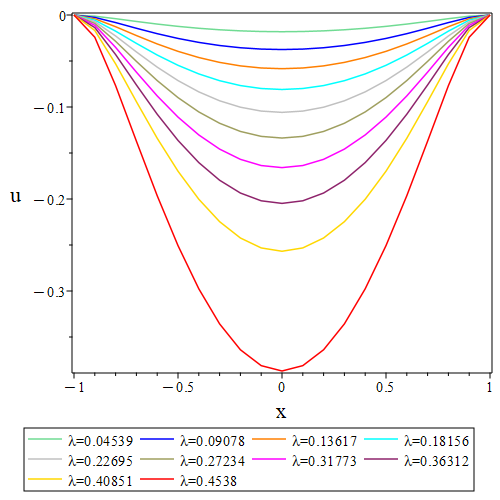}}
	\\
	\subfigure[$t=78.39$]{
		\label{parabolic2DNEW1}
		\includegraphics[width=0.39\linewidth]{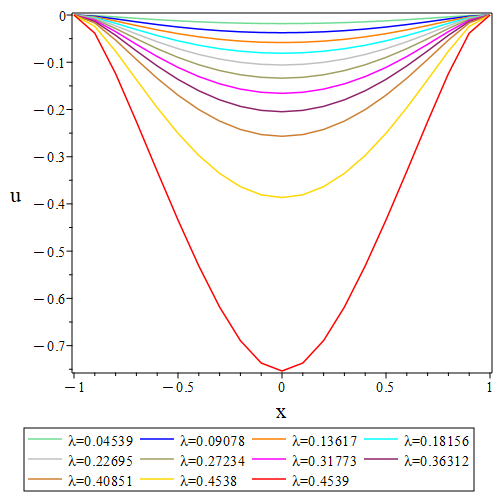}}
	\\
	\subfigure[$t=0.5$]{
		\label{parabolic2DNEW2}
		\includegraphics[width=0.39\linewidth]{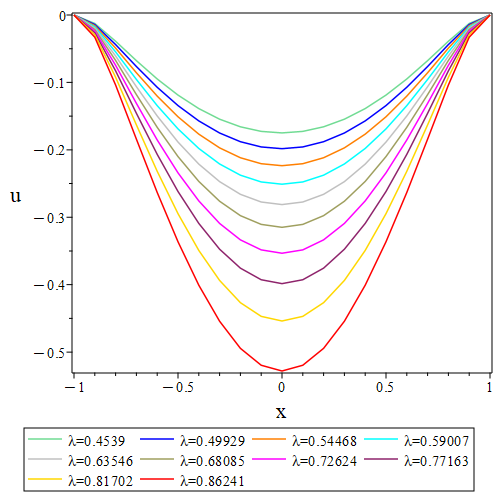}}
	\caption{Plot of the solution $u(t,x)$ for the parabolic MEMS equation \eqref{neweq1.1} at fixed $t$\label{tgudingde2d}}
\end{figure}

\begin{figure}[!htp]
	\begin{center}
		\subfigure[$x=0$]{
			\label{parabolic2D0}
			\includegraphics[width=0.4\linewidth]{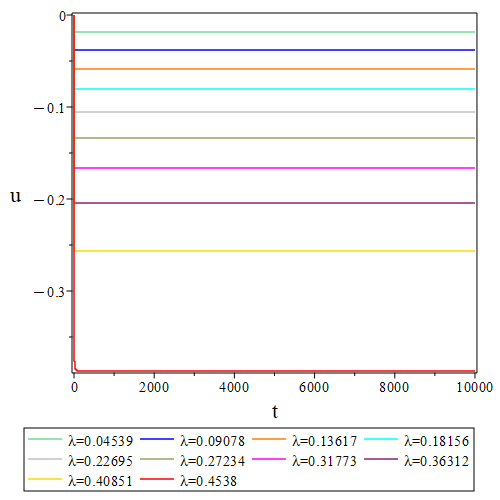}}
		\quad 
		\subfigure[$x=0$]{
			\label{parabolic2D1}
			\includegraphics[width=0.4\linewidth]{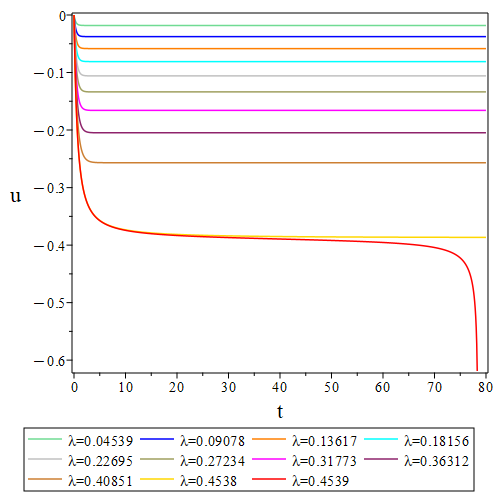}}\\
		
		\subfigure[$x=0$]{
			\label{parabolic2D2}
			\includegraphics[width=0.4\linewidth]{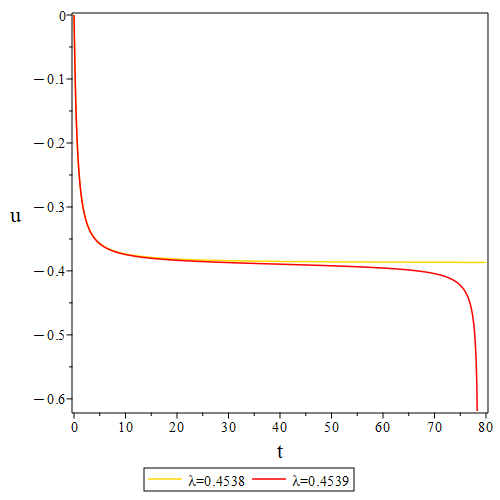}}
		\subfigure[$x=0$]{
			\label{parabolic2D3}
			\includegraphics[width=0.4\linewidth]{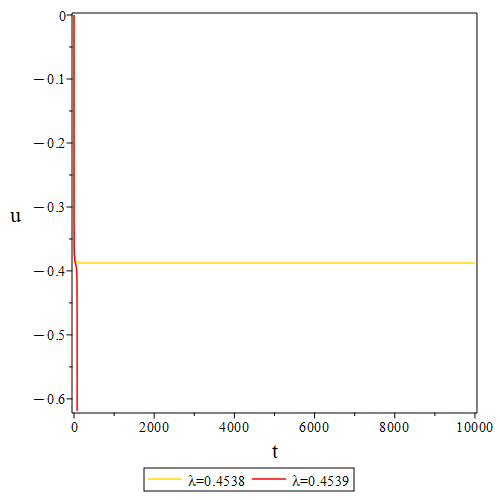}}\\
		
		\subfigure[$x=0$]{
			\label{parabolic2D4}
			\includegraphics[width=0.4\linewidth]{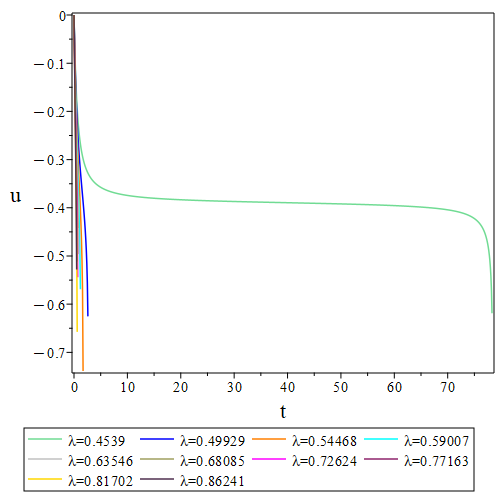}}
		\quad 
		\subfigure[$x=0$]{
			\label{parabolic2D5}
			\includegraphics[width=0.4\linewidth]{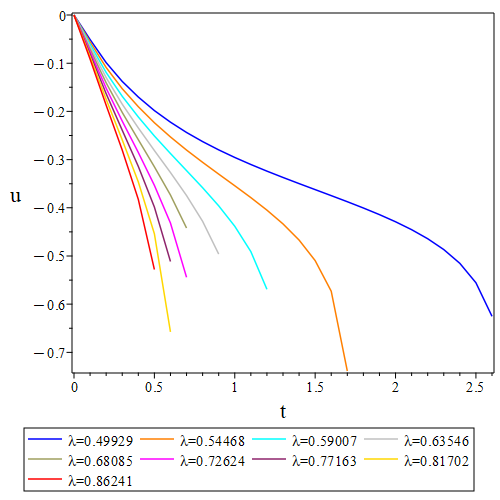}}
		\caption{Plot of the solution $u(t,x)$ for the parabolic MEMS equation \eqref{neweq1.1} at fixed $x=0$\label{xgudingde2d}}
	\end{center}
\end{figure}

\begin{figure}[!htp]
	\begin{center}
		\subfigure[$\lambda=0.4538$]{
			\label{parabolic3D1}
			\includegraphics[width=0.65\linewidth]{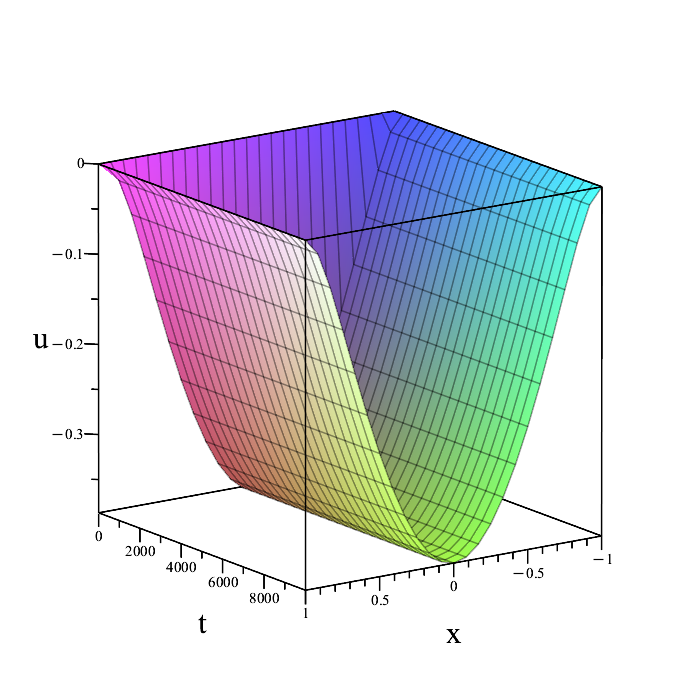}}\\ 
		\subfigure[$\lambda=0.4539$]{
			\label{parabolic3D2}
			\includegraphics[width=0.65\linewidth]{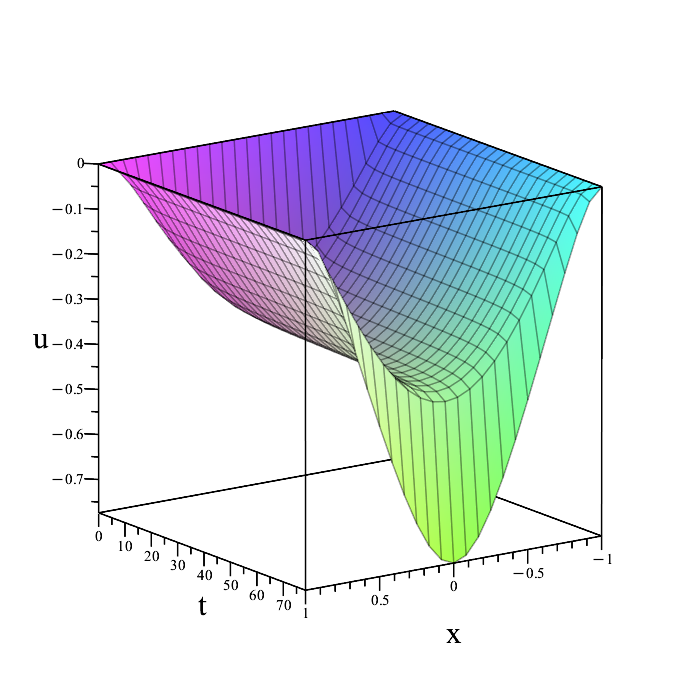}}
		\caption{Plot of the solution $u(t,x)$ for the parabolic MEMS equation \eqref{neweq1.1} at fixed $(t,x)$\label{parabolic3D}}
	\end{center}
\end{figure}
\FloatBarrier

\subsection{Numerical Solutions for the Hyperbolic MEMS Equation}\label{shuangqumemsshizhimoni111}
This section focuses on numerical solutions of problem \eqref{neweq1.2} with fixed parameters. Setting $(B,T)=(1,1)$, $\Omega=(-1,1)$, $(u_{0},u_{1})=(0,0)$, and taking a step size of 0.42864 for varying $\lambda$, we obtain Figure~\ref{tgudingde2dd} and Figure~\ref{xgudingde2dd}.

In Figure~\ref{tgudingde2dd}, we observe that the solution is symmetric about the domain $\Omega$ and attains its minimum at $x=0$. For fixed $x$, the solution value decreases as $\lambda$ increases. In Figures~\ref{hyperbolic2D0}-\ref{hyperbolic2D1}, we note that when $\lambda \le 4.2864$, the solution stabilizes over an extended time interval. Furthermore, at any fixed time $t\in (0,1)$, the solution value decreases as $\lambda$ increases. When $\lambda$ increases from 4.2863 to 4.2864, the solution value rapidly decreases, as shown in Figures~\ref{hyperbolic2D2}-\ref{hyperbolic2D3}. In Figures~\ref{hyperbolic2D4}-\ref{hyperbolic2D5}, we observe that for $\lambda\ge 4.2864$, the solution exhibits a rapid decreasing trend, and at any fixed time $t\in (0,0.3)$, the solution value decreases as $\lambda$ increases.

Additionally, we plot three-dimensional visualizations for parameters $\lambda=4.2863$ and $\lambda=4.2864$ to enhance observation, shown in Figure~\ref{hyperbolic3D}. Based on the above analysis, we propose the following conjecture.
\begin{conjecture}
	Assume $(B,T)$, $(u_{0},u_{1})$, $d$, and $\Omega$ are given, and $u$ is the unique maximal solution to problem \eqref{neweq1.2}. We make the following conjectures: there exist critical values $0<\lambda_{h,1}^{*}<\lambda_{h,2}^{*}$ such that
	\begin{enumerate}
		\item[\textup{1.}] For $0<\lambda<\lambda_{h,2}^{*}$, $u(t,x)$ exists globally;
		\item[\textup{2.}] For $0<\lambda<\lambda_{h,1}^{*}$, at any fixed $(t,x)\in \mathbb{R}^{+}\times \Omega$, $u(t,x)$ is monotonically decreasing in $\lambda$;
		\item[\textup{3.}] For $\lambda> \lambda_{h,2}^{*}$, $u$ reaches the value $-1$ in finite time, and this time is monotonically decreasing in $\lambda$.
	\end{enumerate}
\end{conjecture}

\begin{figure}[H]
	\centering
	\subfigure[$t=10000$]{
		\label{hyperbolic2DNEW0}
		\includegraphics[width=0.32\linewidth]{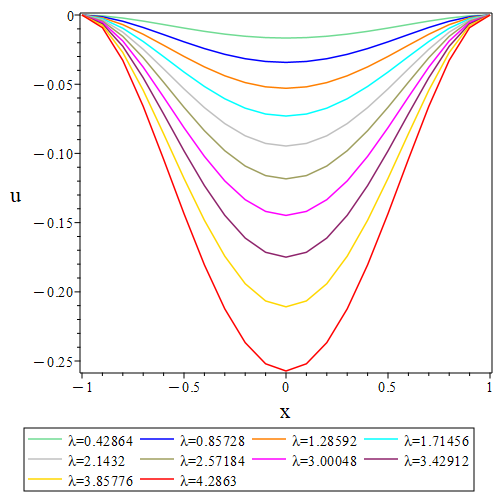}}
	\hfill
	\subfigure[$t=2.5$]{
		\label{hyperbolic2DNEW1}
		\includegraphics[width=0.32\linewidth]{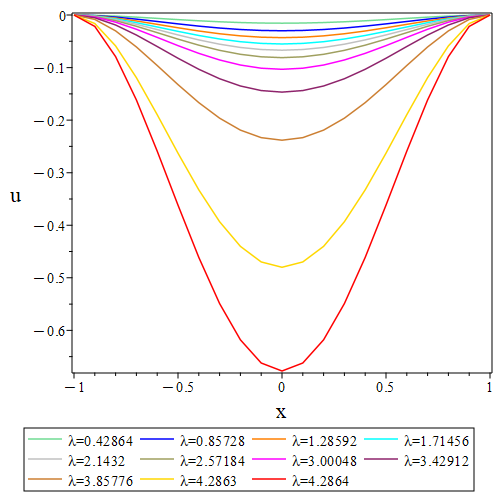}}
	\hfill
	\subfigure[$t=0.4$]{
		\label{hyperbolic2DNEW2}
		\includegraphics[width=0.32\linewidth]{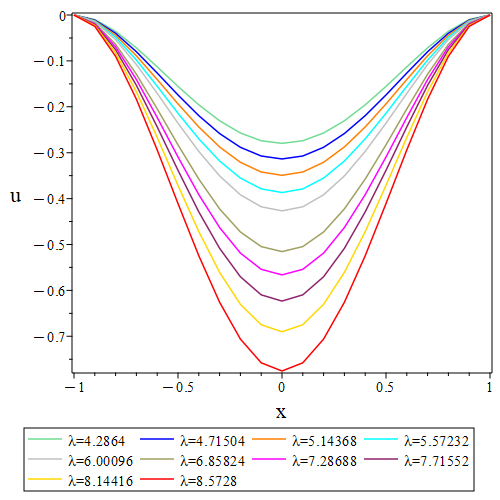}}
	
	\caption{Plot of the solution $u(t,x)$ for the hyperbolic MEMS equation \eqref{neweq1.2} at fixed $t$\label{tgudingde2dd}}
\end{figure}

\begin{figure}[!htp]
	\begin{center}
		\subfigure[$x=0$]{
			\label{hyperbolic2D0}
			\includegraphics[width=0.4\linewidth]{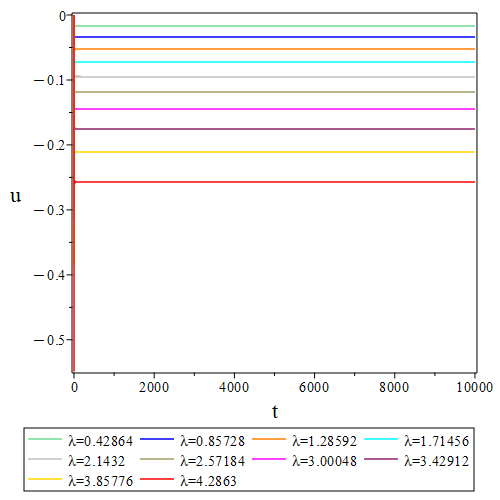}}
		\quad 
		\subfigure[$x=0$]{
			\label{hyperbolic2D1}
			\includegraphics[width=0.4\linewidth]{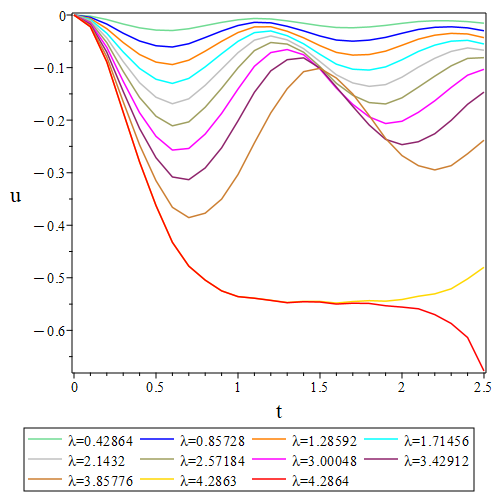}}\\
		
		\subfigure[$x=0$]{
			\label{hyperbolic2D2}
			\includegraphics[width=0.4\linewidth]{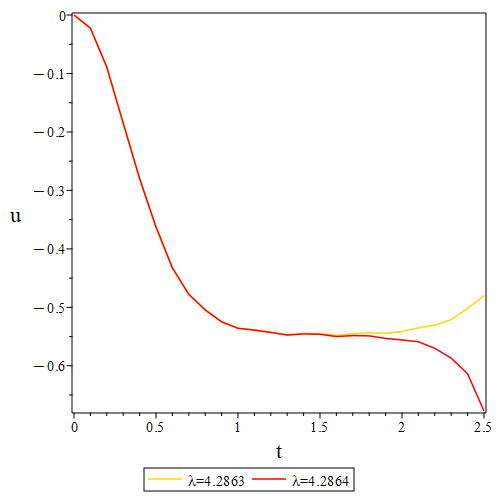}}
		\subfigure[$x=0$]{
			\label{hyperbolic2D3}
			\includegraphics[width=0.4\linewidth]{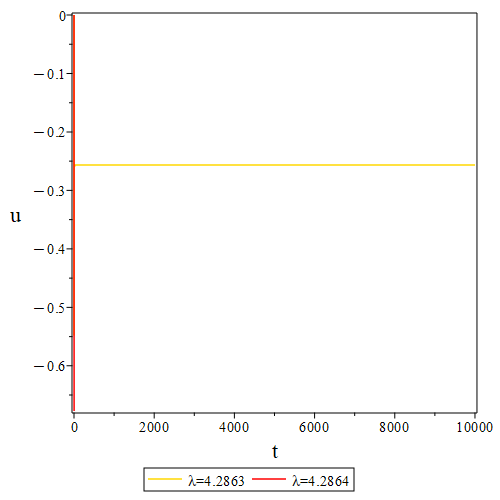}}\\
		
		\subfigure[$x=0$]{
			\label{hyperbolic2D4}
			\includegraphics[width=0.4\linewidth]{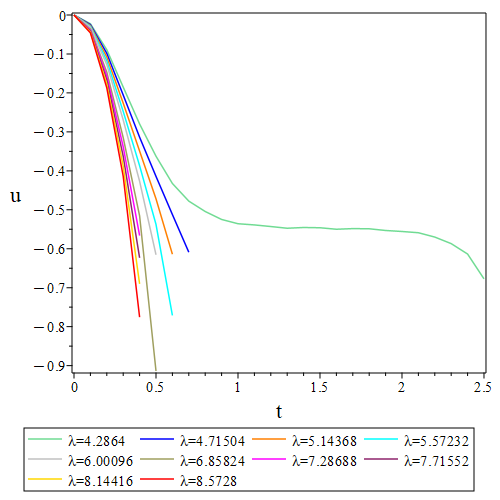}}
		\quad 
		\subfigure[$x=0$]{
			\label{hyperbolic2D5}
			\includegraphics[width=0.4\linewidth]{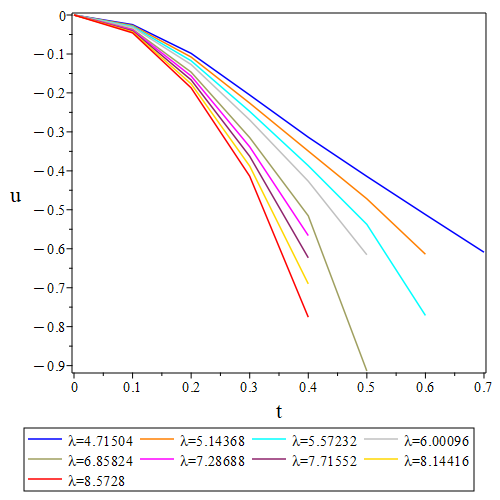}}
		\caption{Plot of the solution $u(t,x)$ for the hyperbolic MEMS equation \eqref{neweq1.2} at fixed $x=0$\label{xgudingde2dd}}
	\end{center}
\end{figure}

\begin{figure}[!htp]
	\begin{center}
		\subfigure[$\lambda=4.2863$]{
			\label{hyperbolic3D1}
			\includegraphics[width=0.65\linewidth]{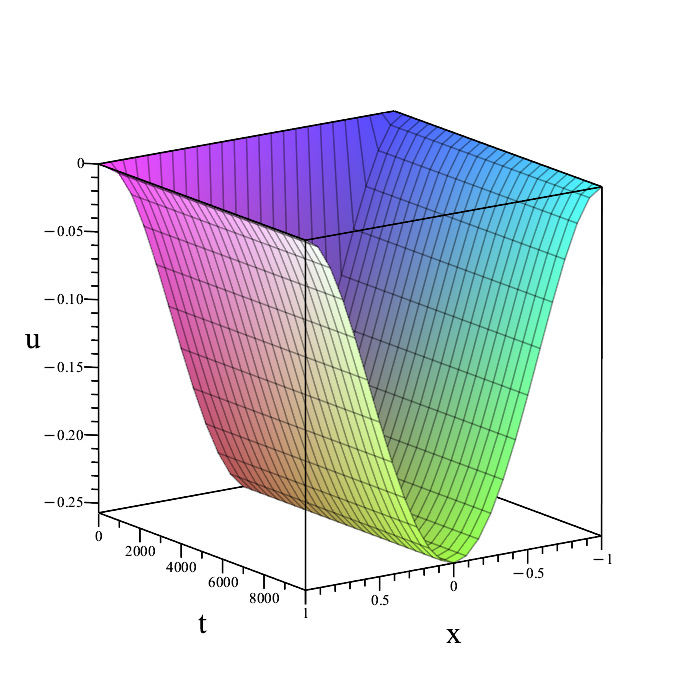}}\\ 
		\subfigure[$\lambda=4.2864$]{
			\label{hyperbolic3D2}
			\includegraphics[width=0.65\linewidth]{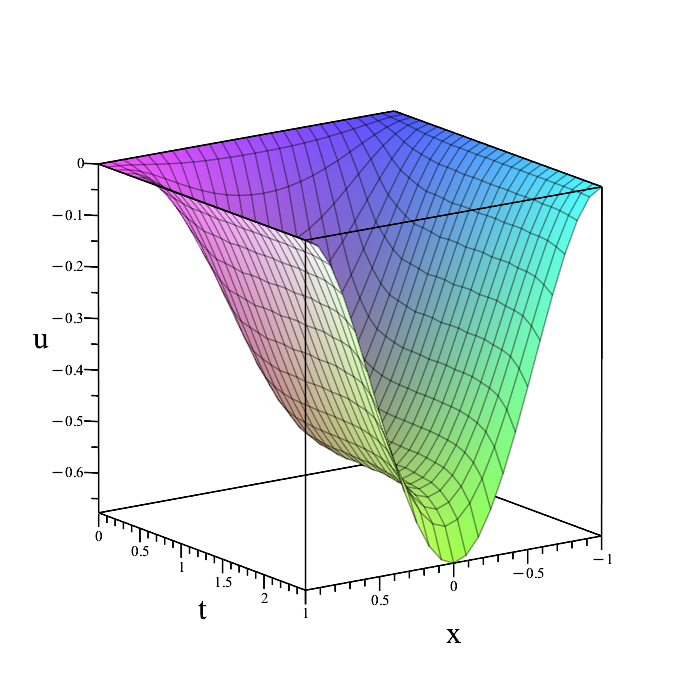}}
		\caption{Plot of the solution $u(t,x)$ for the hyperbolic MEMS equation \eqref{neweq1.2} at fixed $(t,x)$\label{hyperbolic3D}}
	\end{center}
\end{figure}
\FloatBarrier

\section{Summary}\label{zongjiezhanwang}
In this section, we summarize the content of the previous two sections and identify similarities and differences in the visualizations of solutions for both types of equations.

\textbf{Similarities:}
\begin{itemize}
	\item[1.] Both equations use zero initial conditions and share the same domain $\Omega = (-1,1)$.
	\item[2.] The solutions are symmetric about the domain $\Omega$ and attain their minimum at $x=0$.
	\item[3.] For fixed $(t,x)$, the solution value decreases as $\lambda$ increases.
	\item[4.] When $\lambda$ is less than or equal to some $\lambda_{*}$, the solution stabilizes over an extended time interval.
	\item[5.] When $\lambda$ is greater than or equal to some $\lambda^{*}$, the solution exhibits a rapid decreasing trend.
\end{itemize}

\textbf{Differences:}
Figure~\ref{parabolic2D1} and Figure~\ref{hyperbolic2D1} exhibit a notable difference:
\begin{itemize}
	\item[] For $\lambda \leq \lambda_{*}$ and within a certain time period starting from $t=0$, the solution in Figure~\ref{parabolic2D1} remains stable throughout;
	\item[] Whereas the solution in Figure~\ref{hyperbolic2D1} displays oscillatory behavior: first rising, then falling, and subsequently rising again.
\end{itemize}

\end{document}